\documentclass[11pt,reqno,twosdie]{amsart}

\usepackage{amsmath,amssymb,mathrsfs}

\usepackage[asymmetric,top=3cm,bottom=3.3cm,left=3.1cm,right=3.1cm]{geometry}
\usepackage{amsmath,amsfonts,amsthm,mathrsfs,amssymb,cite}
\usepackage[usenames]{color}

\usepackage{mathtools} 
\usepackage{graphics}
\usepackage{framed}

\usepackage{enumerate}

\allowdisplaybreaks

\usepackage{graphicx}
\usepackage{latexsym}
\usepackage{enumerate}

\theoremstyle{plain}
\newtheorem{theorem}{Theorem}[section]
\newtheorem{lemma}[theorem]{Lemma}
\newtheorem{corollary}[theorem]{Corollary}

\theoremstyle{remark}
\newtheorem{definition}[theorem]{Definition}
\newtheorem{remark}[theorem]{Remark}

\renewcommand{\eqref}[1]{\textnormal{(\ref{#1})}}

\numberwithin{equation}{section}

\newcommand{\rmd}{\mathrm{d}}

\newcommand{\bfx}{\mathbf{x}}

\DeclareMathOperator{\sgn}{sgn}

\newcommand{\R}{\mathbb{R}}

\def\bsi{{\mathbf{i}}}
\def\Oh{{\mathcal  O}}

\def\bsl{ \boldsymbol {l} }
 \def\bsL{ \boldsymbol {L} }
\def\bfx{ \mathbf {x} }
 
\usepackage{color}
\DeclareMathOperator*{\argmin}{\arg\min}

\title[Unique continuation for Maxwell's system and applications]{Unique continuation from a generalized impedance edge-corner for Maxwell's system and applications to inverse problems} 

\author{Huaian Diao}
\address{School of Mathematics and Statistics, Northeast Normal University,
Changchun, Jilin 130024, China.}
\email{hadiao@nenu.edu.cn}

\author{Hongyu Liu}
\address{Department of Mathematics, City University of Hong Kong, Kowloon, Hong Kong, China.}
\email{ hongyu.liuip@gmail.com, hongyliu@cityu.edu.hk}

\author{Long Zhang}
\address{School of Mathematics and Statistics, Northeast Normal University,
Changchun, Jilin 130024, China.}
\email{907278586@qq.com}

\author{Jun Zou}
\address{Department of Mathematics, The Chinese University of Hong Kong, Hong Kong, China}
\email{zou@math.cuhk.edu.hk}

\begin{document}

\begin{abstract}

We consider the time-harmonic Maxwell system in a domain with a generalized impedance edge-corner, namely the presence of two generalized impedance planes that intersect at an edge. The impedance parameter can be $0, \infty$ or a finite non-identically vanishing variable function. We establish an accurate relationship between the vanishing order of the solutions to the Maxwell system and the dihedral angle of the edge-corner. In particular, if the angle is irrational, the vanishing order is infinity, i.e. strong unique continuation holds from the edge-corner. The establishment of those new quantitative results involve a highly intricate and subtle algebraic argument. The unique continuation study is strongly motivated by our study of a longstanding inverse electromagnetic scattering problem. As a significant application, we derive several novel unique identifiability results in determining a polyhedral obstacle as well as it surface impedance by a single far-field measurement. We also discuss another potential and interesting application of our result in the inverse scattering theory related to the information encoding.

\medskip
\noindent{\bf Keywords} Maxwell's system, generalized impedance plane, edge-corner, vanishing order, inverse electromagnetic scattering, single far-field measurement

\medskip
\noindent{\bf Mathematics Subject Classification (2010)}: 35P05, 35P25, 35R30, 35Q60

\end{abstract}

\maketitle

\section{Introduction}\label{sec:Intro}

Let $\Omega$ be an open set in $\mathbb{R}^3$, bounded or unbounded, and set
\[
\begin{split}
H_{loc}(\mathrm{curl}, \Omega)=& \big\{U|_{B}\in H(\mathrm{curl}, B); B\ \mbox{is any bounded subdomain of $\Omega$}\big\}, \\
H(\mathrm{curl}, B)=& \big\{U\in L^2(B)^3; \ \nabla\wedge U\in L^2(B)^3\big\}. 
\end{split}
\]
Consider the time-harmonic Maxwell equations for $(\mathbf{E}, \mathbf{H})\in H_{loc}(\mathrm{curl}, \Omega)\times H_{loc}(\mathrm{curl}, \Omega)$:
\begin{equation}\label{eq:eig}
\nabla\wedge\mathbf{E}-\mathbf{i} k\mathbf{H}={\mathbf 0},\quad \nabla\wedge\mathbf{H}+\mathbf{i} k\mathbf{E}={\mathbf 0},
\end{equation}
where $\mathbf{i}:=\sqrt{-1}$ and $k\in\mathbb{R}_+$. In this paper, we are concerned with the unique continuation property (UCP) of the Maxwell system \eqref{eq:eig} in a particular scenario, which is strongly motivated by our study of a longstanding problem in the inverse electromagnetic scattering theory. In what follows, we first present the mathematical setup for our UCP study. 

Let $B_\rho(\mathbf{x})$ denote a ball of radius $\rho\in\mathbb{R}_+$ and centered at $\mathbf{x}\in\mathbb{R}^3$. In the sequel, for a set $K\subset\mathbb{R}^3$, $B_\rho(K):=\{\mathbf{x}; \mathbf{x}\in B_\rho(\mathbf{y})\ \mbox{for any}\ \mathbf{y}\in K\}$.
Let $\Pi_1$ and $\Pi_2$ be two planes in $\mathbb{R}^3$ such that $\Pi_1\cap\Pi_2=\bsL$, where $\bsL$ is a straight line. We suppose that there exists an open line segment $\bsl\Subset \bsL$ and $\rho\in\mathbb{R}_+$ such that $B_\rho(\bsl)\Subset\Omega$. Let $\mathcal{W}(\Pi_1, \Pi_2)$ denote one of the wedge domains formed by $\Pi_1$ and $\Pi_2$, then
$\partial \mathcal{W}(\Pi_1,\Pi_2)\cap B_\rho(\bsl)$ is called an edge-corner associated with $\Pi_1$ and $\Pi_2$;
see Fig.~\ref{fig:coordinate1} for a schematic illustration. In the sequel, we let $\widetilde{\Pi}_j$, $j=1,2$, denote the two flat faces of the edge-corner lying on $\Pi_j$, respectively, and denote it by ${\mathcal E}(\widetilde\Pi_1, \widetilde\Pi_2,\bsl)$. Any $\mathbf{x}\in \bsl$ is said to be an edge-corner point of ${\mathcal E}(\widetilde\Pi_1, \widetilde\Pi_2,\bsl)$.
\begin{figure}[htbp]
	\centering
\vspace*{-1cm}	\includegraphics[width=0.25\linewidth]{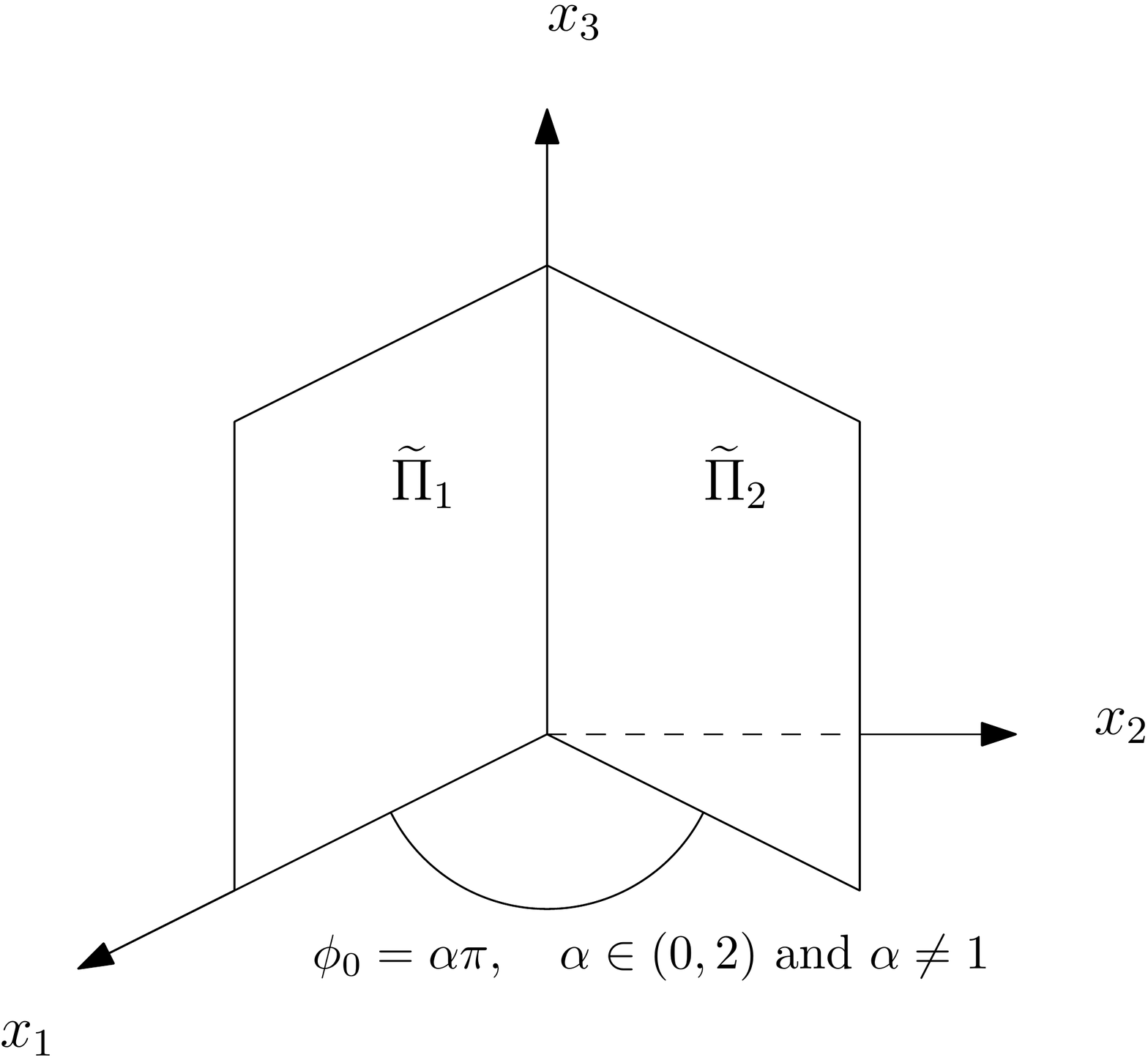}\\[-25pt]
	\caption{Schematic illustration of two intersecting planes with an edge-corner ${\mathcal E}(\widetilde\Pi_1, \widetilde\Pi_2,\bsl)$ and the dihedral angle $\phi_0$.}
	\label{fig:coordinate1}
\end{figure}

%

 Let $\boldsymbol{ \eta}_j$ denote a generalized impedance parameter on $\widetilde{\Pi}_j$, whose value must fulfil one of the following three possibilities 
\begin{equation}\label{eq:imp1}
({\rm i})~\boldsymbol{ \eta}_j\equiv 0;\quad ({\rm ii})~\boldsymbol{ \eta}_j\equiv \infty;\quad ({\rm iii})~\boldsymbol{ \eta}_j\in L^\infty(\widetilde{\Pi}_j). 
\end{equation}
Let $\nu_j\in\mathbb{S}^2$ be the unit normal vector to $\Pi_j$, pointing to the exterior of $\mathcal{W}(\Pi_1,\Pi_2)$. We introduce the following {\it generalized impedance condition} on $\widetilde{\Pi}_j$ associated with $(\mathbf{E}, \mathbf{H})$ to the Maxwell system \eqref{eq:eig}: 
\begin{equation}\label{eq:imp2}
\nu_j \wedge (\nabla\wedge \mathbf{E})+\boldsymbol{ \eta}_j (\nu_j\wedge\mathbf{E})\wedge\nu_j\big|_{\widetilde{\Pi}_j}=0. 
\end{equation}
In the case $ \boldsymbol{ \eta}_j\equiv \infty$, \eqref{eq:imp2} is understood as
\begin{equation}\label{eq:imp3}
(\nu_j\wedge\mathbf{E})\wedge\nu_j\big|_{\widetilde{\Pi}_j}=0. 
\end{equation}
An edge-corner ${\mathcal E}(\widetilde\Pi_1, \widetilde\Pi_2,\bsl)$ with the generalized impedance condition \eqref{eq:imp2} imposed on $\widetilde{\Pi}_j$, $j=1,2$, is called a generalized impedance edge-corner associated with the Maxwell system \eqref{eq:eig}. In this paper, we shall consider the unique continuation property of the solution $(\mathbf{E}, \mathbf{H})$ to \eqref{eq:eig} with the presence of a generalized impedance edge-corner. 

The UCP for differential equations from a crack in the domain has been the subject of many existing studies in the literature, see e.g. \cite{AF,CDD,Dal} and the references cited therein. However, the corresponding study to the Maxwell system is rather rare. Moreover, there are several other features that make our current study interestingly new and distinct from many existing UCP studies from cracks. First, the Maxwell system \eqref{eq:eig} is defined in the whole domain $\Omega$, instead of the exterior of the crack, namely $\Omega\backslash {\mathcal E}(\widetilde\Pi_1, \widetilde\Pi_2,\bsl)$. Usually, for a typical UCP problem from a crack, the differential equation is given over the exterior of the crack, and hence the solution inherits a certain singularity from the pathological geometry of the crack. But in our case, by the standard PDE theory, we know that $(\mathbf{E}, \mathbf{H})$ are real analytic in the interior of $\Omega$, and in particular in $B_\rho(\bsl)$ which is a neighbourhood of the edge-corner. This makes our UCP study seemingly rather ``artificial". However, on the one hand, the UCP problem in this work is strongly motivated by our study of the inverse electromagnetic scattering problems. This shall become more evident in Section~\ref{sec4}, and the UCP results shall generate some significant applications that are of both theoretical and practical importance. On the other hand, it turns out that the analyticity of the solutions around the edge-corner is a key factor that helps us to develop an algebraic argument in achieving the desired UCP, though highly intricate and subtle. Second, the edge-corner geometry enables us to establish an accurate relationship between the vanishing order of the solutions to the Maxwell system and the angle of the edge-corner. In particular, if the angle is irrational, then the vanishing order is infinity, i.e. strong unique continuation holds from the edge-corner. We would like to point out that it seems that the extension to the other more general geometry seems rather unpractical, though certain quantitative estimates are more plausible. Third, it is remarked that in our UCP study, the Robin-type generalized impedance condition \eqref{eq:imp2} is considered on the crack, namely the edge-corner, whereas in most of the existing studies of UCP from cracks, homogeneous Dirichlet-type or Neumann-type conditions are more concerned, which correspond to $ \boldsymbol{ \eta} \equiv 0$ or $\boldsymbol{ \eta} \equiv \infty$, respectively.

As mentioned earlier, we shall consider two interesting and significant applications of the new UCP results to the study of inverse electromagnetic scattering problems. We postpone the mathematical formulation of the inverse problem to Section~\ref{sec5} and we are mainly concerned with the determination of an impenetrable obstacle as well as its boundary impedance by a single electromagnetic far-field measurement. This constitutes a longstanding problem in the inverse scattering theory (cf. \cite{CK18}). In \cite{LiuA,Liu3,Liu09}, the case $\boldsymbol \eta\equiv 0$ or $\boldsymbol \eta\equiv\infty$ was considered, and it is shown that a single far-field measurement can uniquely determine an obstacle of the general polyhedral shape and the corresponding stability estimate was established in \cite{LRX}. The proofs are mainly based on the path argument originated in \cite{Liu-Zou} for the acoustic problem as well as a certain reflection principle for the Maxwell system establish in \cite{Liu3,Liu09}. However, the arguments developed therein cannot be extended to tackle the case that the impedance parameter $\boldsymbol \eta$ is finite and non-identically zero, even if in the simplest case that it is a finite and nonzero constant, and a fortiori a variable function in our study. Using the UCP results derived in this paper, we are able to establish several novel unique identifiability results for this challenging problem in the polyhedral case, especially in the case that $\boldsymbol \eta$ is a finite and non-identically zero variable function. Nevertheless, it is our intention to point out that we shall require certain mild but unobjectionable a-priori knowledge of the underlying polyhedral obstacle as well it surface impedance. The other interesting application of our UCP results is about the ``information encoding" for the inverse electromagnetic scattering problems. Indeed, we shall regard our UCP results as generalizing the classical Holmgren's principle \cite{CK,TF} for the Maxwell equations. With this view, we can provide an alternative means of electromagnetic scattering measurements for inverse problems that might have some practical implications. 

The rest of the paper is organized as follows. Section \ref{sec:2}  is devoted to some preliminary knowledge and auxiliary results. In Sections \ref{sec:5} and  \ref{sec:6}, we establish the UCP results from a generalized impedance edge-corner for the Maxwell equations \eqref{eq:eig} in two different scenarios. In Section \ref{sec5}, we consider the inverse electromagnetic scattering problems and present two applications of the newly established UCP results.

\section{Preliminaries and auxiliary lemmas}\label{sec:2}

In this section, we collect some preliminary knowledge for the Maxwell system \eqref{eq:eig} as well as derive several auxiliary lemmas for our subsequent use. 

First, we note that the Maxwell system \eqref{eq:eig} is invariant under rigid motions (cf. \cite{BLZ,LZh}). Hence, throughout the rest of this paper and without loss of generality, we can assume that the edge-corner ${\mathcal E}(\widetilde\Pi_1, \widetilde\Pi_2,\bsl) \Subset \Omega$ satisfies
\begin{equation*}\label{notation}
\bsl=\big\{~\bfx=(\bfx',x_3)\in \R^3; \bfx':=(x_1,x_2)=\mathbf{0},\ x_3\in (-h,h )\big\}\Subset\Omega,
\end{equation*}
where $2h\in\mathbb{R}_+$ is the length of $\bsl$, and furthermore $\Pi_1$ coincides with the $(x_1,x_3)$-plane while $\Pi_2$ possesses a dihedral angle $\phi_0= \alpha\pi$ away from $\Pi_1$ in the anti-clockwise direction; see Fig.~\ref{fig:coordinate1} for a schematic illustration. Throughout, it is assumed that
\begin{equation}\label{eq:angle2a}
	\alpha\in(0, 2)\quad\mbox{and}\quad \alpha\neq 1. 
\end{equation}
It can directly verified that the exterior unit normal vectors $\nu_j$ to $\Pi_j$, $j=1,2$ are given by
    \begin{equation}\label{l1}
    \begin{split}
    &\nu_1=(0,-1,0)^\top  , \quad \nu_2=(-\sin\phi_0,\cos\phi_0,0)^\top.
    \end {split}
    \end{equation}
As specified earlier, we have the generalized impedance condition \eqref{eq:imp2} imposed on $\widetilde\Pi_j$, where the boundary impedance parameter $\boldsymbol \eta_j$ fulfils \eqref{eq:imp1}. 
In order to consider the unique continuation from the edge-corner as described above, we introduce the following definition. 
\begin{definition}\label{def:3}
Let $\mathbf{E}\in H_{loc}(\mathrm{curl},\Omega)$ be a solution to \eqref{eq:eig} and suppose there exists an edge-corner ${\mathcal E}(\widetilde\Pi_1, \widetilde\Pi_2,\bsl) \Subset \Omega$ as described above. For a given point $\mathbf{x}_0 \in \bsl$, if there  exits a number $N \in {\mathbb N}\cup\{0\}$ such that
\begin{equation}\label{eq:normal3}
	\lim_{\rho\rightarrow +0} \frac{1}{\rho^m} \int_{B_\rho(\mathbf{x}_0)}\, |\mathbf{E}(\mathbf{x})|\, {\rm d} \mathbf{x}=0\ \ \mbox{for}\ \ m=0,1,\ldots, {{N+2}},
\end{equation}
we say that $\mathbf{E}$ vanishes at $\mathbf{x}_0$ up to the order $N$. The largest possible $N$ such that \eqref{eq:normal3} is fulfilled is called the vanishing order of $\mathbf{E}$ at $\mathbf{x}_0$, and we write
\begin{equation*}\label{eq:normal4}
\mathrm{Vani}(\mathbf{E}; \mathbf{x}_0)=N.
\end{equation*}
If \eqref{eq:normal3} holds for any $N\in\mathbb{N}$, then we say that the vanishing order is infinity.
\end{definition}

Since $\mathbf{E}$ is (real) analytic in $\Omega$, we immediately see that if the vanishing order of $\mathbf{E}$ at any point $\mathbf{x}_0\in\bsl$ is infinity, then $\mathbf{E}\equiv 0$ in $\Omega$, namely the strong unique continuation property holds. In what follows, it is sufficient to consider the UCP at the origin $\mathbf{0}\in\bsl$. Moreover, due to the symmetry role between $(\mathbf{E}, \mathbf{H})$ and $(-\mathbf{H}, \mathbf{E})$, namely both of them satisfy the same Maxwell system \eqref{eq:eig}, we only consider the vanishing order of $\mathbf{E}$, and the same result equally holds for $\mathbf{H}$. It turns out that the vanishing order of $\mathbf{E}$ is related to the {\it rationality} of the edge-corner angle, i.e. $\alpha\pi$, and we shall make it more rigorous in the following.

In the subsequent analysis,  we shall make frequent use of the spherical coordinate of a point $\bfx$ in $\R^3$:
\begin{equation}\label{eq:x sph}
		 \mathbf{ x}=(r\sin\theta\cos\phi, r\sin\theta\sin\phi, r\cos\theta):=(r,\theta,\phi),\ r\geq 0, \, {{\theta\in [0,\pi),\, \phi \in [0, 2\pi)}}\,.
	\end{equation}
It is noted that
    \begin{equation}\label{w1}
    \begin{split}
    &\boldsymbol{\hat{r}}=\sin\theta\cos\phi\cdot\hat{\mathbf{x}} +\sin\theta\sin\phi\cdot\hat{\mathbf{y}}+\cos\theta\cdot\hat{\mathbf{z}}\\
    &\boldsymbol{\hat{\theta}}=\cos\theta\cos\phi\cdot\hat{\mathbf{x}}+\cos\theta\sin\phi\cdot\hat{\mathbf{y}}-\sin\theta\cdot\hat{\mathbf{z}}\\
    &\boldsymbol{\hat{\phi}}=-\sin\phi\cdot\hat{\mathbf{x}}+\cos\phi\cdot\hat{\mathbf{y}}
    \end{split}
    \end{equation}
    constitutes an orthonormal basis in the spherical coordinate system, where $\hat{\mathbf{x}}=(1,0,0)^{\top},\hat{\mathbf{y}}=(0,1,0)^{\top},\hat{\mathbf{z}}=(0,0,1)^{\top}$. 
    
    \begin{definition}\label{def:class1}
    Suppose that $\psi(r,\theta)$ is a complex-valued function for $(r, \theta)\in \Sigma:=[0, r_0]\times [-\theta_0, \theta_0]$, where $r_0, \theta_0\in\mathbb{R}_+$. $\psi$ is said to belong to class $\mathcal{A}$ in $\Sigma$ if it allows an absolutely convergent series representation as follows
    \begin{equation}\label{eq:series1}
    \psi(r,\theta)=a_0+\sum_{j=1}^\infty a_j(\theta) r^j,
    \end{equation}
    where $a_0\in\mathbb{C}\backslash\{0\}$ and $a_j(\theta)\in C[-\theta_0, \theta_0]$. 
    \end{definition}
    
   Two simple scenarios for $\psi(r,\theta)$ to belong to the class $\mathcal{A}$: first, $\psi$ is a non-zero constant; second, $\psi(r, \theta)$ is real-analytic in $\Sigma$ with $r_0, \theta_0$ sufficiently small and $\psi(0,\theta)$ independent of $\theta$. For an impedance parameter $\boldsymbol{ \eta}_j$ in \eqref{eq:imp2} in the third case, namely $\boldsymbol{ \eta}_j\in L^\infty(\widetilde\Pi_j)$, we readily see that in the $(r,\theta,\phi)$-coordinate, $\phi|_{\widetilde\Pi_1}=0$ and $\phi|_{\widetilde\Pi_2}=\phi_0$. In what follows, if for any $\bfx_0\in\bsl$ there exists a neighbourhood $\Sigma_{\bfx_0}$ of $\bfx_0$ which is of the form in Definition~\ref{def:class1} and is contained in  $\overline{\widetilde{\Pi}_j}$ such that $\psi_{\bfx_0}(r,\theta):=\boldsymbol{ \eta}_j(\bfx-\bfx_0)$ belongs to the class $\mathcal{A}$ in $\Sigma_{\bfx_0}$, then we say that $\eta_j$ belongs to the class $\mathcal{A}(\bsl)$. It is emphasized that $\boldsymbol{ \eta}_j$ belonging to the class $\mathcal{A}(\bsl)$ is a local property, which is localized around a neighbourhood of $\bsl$ on $\widetilde{\Pi}_j$. In fact, in our subsequent analysis of the UCP from the edge-corner ${\mathcal E}(\widetilde\Pi_1, \widetilde\Pi_2,\bsl)$ is confined locally around a neighbourhood of $\bsl$, and indeed, around a neighbourhood of the origin $\mathbf{0}$ according to our earlier discussion.

Next, we consider the Fourier representations of the solutions to \eqref{eq:eig} in terms of the spherical waves. Throughout the rest of the paper, for a fixed $l \in \mathbb N$ we adopt the notation
  \begin{equation}\label{eq:lindex}
    	[ l ]_0:=\{0,\pm 1, \ldots, \pm l \}, \quad [ l ]_1:=\{\pm 1, \ldots, \pm l \}.
    \end{equation}
 Recall  that the spherical harmonics  $Y_l^m(\theta,\phi)$ is given by
	\begin{equation}\label{sphe harmonic}
    \begin{split}
	Y_l^m(\theta, \phi)=c_l^mP_l^{|m|}(\cos\theta)e^{\bsi m\phi},\quad c_l^m=\sqrt{\frac{2l+1}{4\pi}\frac{(l-|m|)!}{(l+|m|)!}},
    \end{split}
	\end{equation}
where $P_l^m(t)$ is the Legendre function. For simplicity, we use the notation $Y_l^m$ for $Y_l^m(\theta, \phi)$ from the clear context. For our subsequent use, the following lemma presents some important properties of the associated Legendre functions, which can be conveniently found in \cite{Abr}.

\begin{lemma}\label{base21}
	In the spherical coordinate system, the Legendre functions fulfil the following orthogonality condition
	for any fixed $n \in \mathbb N$, and any two integers $m\ge 0$ and $l\leq n$:
    \begin{equation}\label{ortho3}
    \int_{-\pi}^{\pi}\frac{P_n^m(\cos\theta)P_n^l(\cos\theta)}{\sin\theta}\,d\theta=
  \begin{cases}
    0 &\mbox{ if }\quad l\neq m,\medskip\\
    \frac{(n+m)!}{m(n-m)!} &\mbox{ if }\quad l=m \neq 0.
    \end{cases}
    \end{equation}
    Furthermore, the following recursive relationships hold
    \begin{equation}\label{uu}
      \begin{split}
\frac{{\rm d}  P_l^{|m|}(\cos\theta)}{{\rm d} \theta}&=\frac{1}{2}\big[(l+|m|)(l-|m|+1)P_l^{|m|-1}(\cos\theta)-P_l^{|m|+1}(\cos\theta)\big], \\
\frac{|m|}{\sin\theta}P_l^{|m|}(\cos\theta)&=-\frac{1}{2}\big[P_{l-1}^{|m|+1}(\cos\theta)+(l+|m|-1)(l+|m|)P_{l-1}^{|m|-1}(\cos\theta)\big],
 \end{split}
\end{equation}
where $l\in \mathbb N$ and  $m\in [ l]_0.$  If $P_l^m(\cos \theta)$ is evaluated at $\theta=0$, for $l\in \mathbb N \cup \{0\}$ we have
\begin{equation}\label{eq:plm0}
	   P_l^m(1)=0, \quad m\in [l]_1 ; \quad  P_l^0(1)=1.
\end{equation}
For a fixed $n\in \mathbb N \cup \{0\}$ and $m\in \mathbb N$ with $m\leq n$, it holds that
\begin{equation}\label{eq:pnm neg}
	P_{n}^{-m} (\cos \theta  ) =(-1)^m  \frac{ (n-m)!}{(n+m)!}P_n^m (\cos \theta ).
\end{equation}
\end{lemma}

 Recall that the spherical Bessel function  $j_\ell (t)$ of the order $\ell $ is defined by
  \begin{equation}\label{eq:bess sph}
	 	 j_\ell (t)=\frac{t^\ell }{ (2\ell+1)!!}\left  (1-\sum_{l=1}^\infty  \frac{(-1)^l t^{2l }}{ 2^l l! (2\ell+3)\cdots  (2\ell+2l+1)  }\right  )=\frac{t^\ell }{ (2\ell+1)!!}+\Oh (t^{\ell+2}) .
	 \end{equation}
There holds the following recursive relationships \cite{Abr}:	
\begin{equation}\label{eq:bessel}
	\begin{split}
		\frac{j_\ell(t)}{t}=\frac{ j_{\ell-1}(t)+j_{\ell+1}(t)} {2\ell+1},   \quad  j_{\ell }'(t) = \frac{\ell  j_{\ell-1}(t)-(\ell+1) j_{\ell+1}(t)}{2\ell+1} , \quad \ell \in \mathbb N.
	\end{split}
\end{equation}

	\begin{lemma}\cite[Lemma 2.5]{CDL2}\label{lem:coeff0}
	Suppose that for $t\in(0, h)$, $h\in\mathbb{R}_+$,
	\begin{equation}\label{coef1}
	\sum_{n=0}^{\infty}\alpha_nj_n(t)=0,
	\end{equation}
	where $j_n(t)$ is the $n$-th spherical Bessel function. Then
	$\alpha_n=0,\quad n=0,1,2,\ldots.$
\end{lemma}

	 \begin{lemma}\cite{CK}\label{k2}
	 Recall that $\hat{\boldsymbol{r}},\hat{\boldsymbol\theta}$ and  $\hat{\boldsymbol\phi}$ are defined  in \eqref{w1}. Denote
	 \begin{equation}\label{ww}
\begin{split}
	&\mathbf{M}_l^m(\bfx)=j_l{(kr)}\cdot\mathbf{X}_l^m,\quad \mathbf{N}_l^m(\bfx)=\boldsymbol{\mathrm{i}}\bigg(\frac{ j_l(kr)}{kr}+j'_l\big(kr\big)\bigg)\mathbf{Z}_l^m-\frac{\sqrt{l(l+1)}}{kr}\cdot j_l(kr)Y_l^m\cdot\boldsymbol{\hat{r}},
    \end{split}
    \end{equation}
    where $k\in \mathbb R_+$,  $j'_l\big(kr\big)$ is the derivative of $j_l(kr)$ with respect to $kr$, and
    $$\mathbf{X}_l^m=\frac{\boldsymbol{\mathrm{i}}}{\sqrt{l(l+1)}}\bigg(\frac{\boldsymbol{\mathrm{i}}\cdot m}{\sin\theta}Y_l^m\hat{\boldsymbol{\theta}}
    -\frac{\partial{Y_l^m}}{\partial\theta}\cdot \hat{\boldsymbol\phi}\bigg),\quad  \mathbf{Z}_l^m=\frac{\boldsymbol{\mathrm{i}}}{\sqrt{l(l+1)}}
    \bigg(\frac{\partial{Y_l^m}}{\partial\theta}\hat{\boldsymbol\theta}+\frac{\boldsymbol{\mathrm{i}}\cdot m}{\sin\theta}Y_l^m\hat{\mathbf{\boldsymbol\phi}}\bigg) .
    $$
    The solution $ \mathbf{E}(\bfx)$ to \eqref{eq:eig} has the following Fourier expansion around $\mathbf{0}$,
    \begin{equation}\notag
    \begin{split}
      \mathbf{E}(\bfx)=\sum_{l=1}^{\infty}\sum_{m=-l}^{l}\bigg(a_l^m\cdot\mathbf{M}_l^m(\bfx)+b_l^m\cdot\mathbf{N}_l^m(\bfx)\bigg),\quad a_l^m, b_l^m \in \mathbb{C},
      \end{split}
      \end{equation}
 which (together with its derivatives) converges uniformly in  $B_{\rho_0}({\mathbf 0})$ for a sufficiently small $\rho_0\in\mathbb{R}_+$. 
        \end{lemma}

Using \eqref{eq:bessel}, from Lemma \ref{k2}, we can derive that
    \begin{equation}
     \begin{split}\label{mix pi21}
\mathbf{E}(\bfx)=&-\sum_{l=1}^{\infty}\sum_{m=-l}^{l}\frac{1}{\sqrt{l(l+1)}}\Bigg\{ b_l^m\cdot{l(l+1)}p_l(kr)\cdot Y_l^m \cdot\hat{\boldsymbol{r}}\\
 &+ \bigg[a_l^m\cdot j_l\big(kr\big)\frac{m}{\sin\theta}Y_l^m+b_l^m\cdot
    q_l(kr)\cdot\frac{\partial{Y_l^m}}{\partial\theta}\bigg]\cdot \hat{\boldsymbol\theta}\\
       &+\boldsymbol{\mathrm{i}} \bigg[a_l^m\cdot j_l(kr) \frac{\partial{Y_l^m}}{\partial\theta}+b_l^m\cdot
     q_l(kr)\frac{ m}{\sin\theta}Y_l^m
     \bigg]\cdot\hat{\boldsymbol\phi}\Bigg\},
      \end{split}
   \end{equation}
where
	 \begin{equation}\label{eq:plql}
\begin{split}
      p_l(kr)=\frac{j_{l-1}\big(kr\big)+j_{l+1}\big(kr\big)}{2l+1},\quad q_l(kr)=\frac{(l+1)j_{l-1}\big(kr\big)-lj_{l+1}\big(kr\big)}{2l+1}.
          \end{split}
    \end{equation}

    \begin{remark}\label{i2}
    In view of \eqref{eq:bess sph}, the lowest order terms of $  p_l(kr)$ and $  q_l(kr)$ with respect to the power of $r$ are 
    $$
    \frac{k^{l-1} }{(2l+1) (2l-1)!!} r^{l-1} \mbox{ and } \frac{(l+1)k^{l-1} }{(2l+1) (2l-1)!!} r^{l-1}
    $$ 
    respectively.
 \end{remark}

\begin{lemma}\cite[Proposition 2.1.7]{krantz}\label{lem:kra}
	If the power series $\sum_{\mu } a_{\mu } {\mathbf x}^\mu $ converges at a point ${\mathbf x}_0$, then it converges uniformly and absolutely on compact subsets of $U(\bf{x}_0)$, where
	$$
	U({\mathbf x}_0)=\{(r_1 x_{0,1},\ldots, r_n x_{0,n}):-1<r_j<1,j=1,\ldots,n\}, \,    {\mathbf x}_0=(x_{0,1},\ldots,  x_{0,n}) \in {\mathbb R}^n.
	$$
\end{lemma}

Using Definition \ref{def:3}, in view of \eqref{mix pi21}, we can obtain the following lemma.

\begin{lemma}\label{lem:vani}
	Let $\mathbf E$ be a solution to \eqref{eq:eig}. Recall that $\mathbf E$ has the radial wave expansion \eqref{mix pi21} in $B_{\rho_0}(\bf0)$. For a fixed $N \in \mathbb N$, if
	\begin{equation}\label{eq:216 cond1}
		a_l^m=b_l^m=0,\quad m\in [l]_0, \quad l=1,2,\ldots,N,
	\end{equation}
	where $[l]_0$  is defined in \eqref{eq:lindex}, then
	\begin{equation}\label{eq:217 cond}
		\mathrm{Vani}(\mathbf{E}; \mathbf{0})\geq N.
	\end{equation}
Conversely, if there exits $N\in \mathbb N$ such that \eqref{eq:217 cond} holds then we have \eqref{eq:216 cond1}.
\end{lemma}

\begin{proof} From Lemma \ref{lem:kra}, we know that \eqref{mix pi21} converges uniformly and absolutely in  $B_{\rho_1}({\mathbf 0})$, where $0<\rho_1<\rho_0$.
	Substituting \eqref{eq:216 cond1} into \eqref{mix pi21}, we have
	\begin{equation}
     \begin{split}\label{eq:E N+1}
\mathbf{E}(\bfx)=&-\sum_{l=N+1}^{\infty}\sum_{m=-l}^{l}\frac{1}{\sqrt{l(l+1)}}\Bigg\{ b_l^m\cdot{l(l+1)}p_l(kr)\cdot Y_l^m\cdot\hat{\boldsymbol{r}}\\
 &+ \bigg[a_l^m\cdot j_l\big(kr\big)\frac{m}{\sin\theta}Y_l^m+b_l^m\cdot
    q_l(kr)\cdot\frac{\partial{Y_l^m}}{\partial\theta}\bigg]\cdot \hat{\boldsymbol\theta}\\
       &+\boldsymbol{\mathrm{i}} \bigg[a_l^m\cdot j_l(kr) \frac{\partial{Y_l^m}}{\partial\theta}+b_l^m\cdot
     q_l(kr)\frac{ m}{\sin\theta}Y_l^m
     \bigg]\cdot\hat{\boldsymbol\phi}\Bigg\}.
      \end{split}
   \end{equation}
   From Remark \ref{i2}, the lowest order of $r$ with respect to the power of $r$ in \eqref{eq:E N+1} is $N$. Therefore,
   	\begin{equation}
     \begin{split}\label{eq:E N+1}
\frac{\mathbf{E}(\bfx)}{r^N}=&-\sum_{l=N+1}^{\infty}\sum_{m=-l}^{l}\frac{1}{\sqrt{l(l+1)}}\Bigg\{ b_l^m\cdot{l(l+1)}\frac{ p_l(kr)}{r^N }\cdot Y_l^m\cdot\hat{\boldsymbol{r}}\\
 &+ \bigg[a_l^m\cdot \frac{j_l\big(kr\big)}{r^N}\frac{m}{\sin\theta}Y_l^m+b_l^m\cdot
    \frac{q_l(kr)}{r^N}\cdot\frac{\partial{Y_l^m}}{\partial\theta}\bigg]\cdot \hat{\boldsymbol\theta}\\
       &+\boldsymbol{\mathrm{i}} \bigg[a_l^m\cdot \frac{j_l(kr)}{r^N} \frac{\partial{Y_l^m}}{\partial\theta}+b_l^m\cdot
     \frac{q_l(kr)}{r^N}\frac{ m}{\sin\theta}Y_l^m
     \bigg]\cdot\hat{\boldsymbol\phi}\Bigg\}
      \end{split}
   \end{equation}
converges uniformly and absolutely in  $B_{\rho_1}({\mathbf 0})$, which implies
\begin{equation}\label{eq:221}
	\left| \frac{\mathbf{E}(\bfx)}{r^N} \right| =	\Oh(1), \quad 	\mbox{ as } r \rightarrow +0 .
\end{equation}
In view of  Definition \ref{def:3}, by virtue of \eqref{eq:221},  we have
\begin{equation}\notag 
	\lim_{\rho\rightarrow +0} \frac{1}{\rho^m} \int_{B_\rho(\mathbf{0} )}\, |\mathbf{E}(\mathbf{x})|\, {\rm d} \mathbf{x}\leq \lim_{\rho\rightarrow +0} \frac{\rho^{N+2}}{\rho^m} \int_{0}^\rho\int_{0}^\pi \int_{0}^{2\pi} \left| \frac{\mathbf{E}(\bfx)}{r^N} \right| {\rmd}r {\rmd}\theta {\rmd}\phi=0,
\end{equation}
which holds for $m=0,1,\ldots, {{N+2}},$ and this proves \eqref{eq:217 cond}. The other direction of the conclusion can be proved by using similar arguments.
\end{proof}
\begin{lemma}\label{eq:e1e21}
Let $\mathbf E$ be a solution to \eqref{eq:eig}. Recall that $\mathbf{E}$  has the radial wave expansion \eqref{mix pi21} in $B_{\rho_0}(\mathbf{0})$. Consider an edge-corner ${\mathcal E}(\widetilde\Pi_1, \widetilde\Pi_2,\bsl)\Subset \Omega$ associated with $\mathbf{E}$. Recall that $\nu_i$ defined in \eqref{l1} are the outward unit normal vectors to $\Pi_i$, $i=1,2$. Then
\begin{equation}\label{mix pi2}
\begin{split}
&	\nu_1 \wedge \mathbf{E}|_{\widetilde\Pi_1}=  \sum_{l=1}^{\infty}\sum_{m=-l}^{l}-\frac{1}{\sqrt{l(l+1)}}\Bigg\{b_l^ml(l+1)p_l(kr)Y_l^m\Big|_{\phi=0}
  \boldsymbol{e_1}(\theta,0)\\
    & \hspace{2cm} +\bigg(
    a_l^mj_l\big(kr\big)\frac{m}{\sin\theta}Y_l^m\Big|_{\phi=0}+b_l^m\cdot q_l(kr)
   \frac{\partial{Y_l^m}}{\partial\theta}\Big|_{\phi=0}
    \bigg)\boldsymbol{e_2}(\theta,0)\Bigg\},\\
 &\nu_2 \wedge \mathbf{E}|_{\widetilde\Pi_2}= \sum_{l=1}^{\infty}\sum_{m=-l}^{l}-\frac{1}{\sqrt{l(l+1)}}\Bigg\{b_l^ml(l+1)p_l(kr)Y_l^m\Big|_{\phi=\phi_0}
   \boldsymbol{e_1}(\theta,\phi_0) \\
    &\hspace{2cm} +\bigg(
     a_l^mj_l\big(kr\big)\frac{m}{\sin\theta}Y_l^m \Big|_{\phi=\phi_0}
     +b_l^m\cdot q_l(kr)\frac{\partial{Y_l^m}}{\partial\theta}\Big|_{\phi=\phi_0}\bigg)
 \boldsymbol{e_2}(\theta,\phi_0)\Bigg\},
  \end{split}
 \end{equation}
 where
 \begin{equation}\label{eq:e1e2}
 	 \boldsymbol{e}_{1}\left(\theta, \phi\right)=\left[\begin{array}{c}\cos \phi \cos \theta \\ \sin \phi \cos \theta \\ -\sin \theta\end{array}\right] \mbox{ and }  \boldsymbol{e}_{2}\left(\theta, \phi\right)=-\left[\begin{array}{c} \cos \phi \sin \theta \\ \sin \phi \sin \theta \\ \cos \theta\end{array}\right],
 \end{equation}
 are linearly independent for any $\theta$ and $\phi$.
Furthermore, we have
\begin{equation}\label{gg}
\begin{split}
	 &\nu_1 \wedge(\nabla\wedge \mathbf{E}|_{\widetilde\Pi_1})={\mathbf{i}k}\sum_{l=1}^{\infty}\sum_{m=-l}^{l}\frac{1}{\sqrt{l(l+1)}}\Bigg\{a_l^ml(l+1)p_l(kr)Y_l^m \Big|_{\phi=0}
\cdot\boldsymbol{e_1}(\theta,0)\\
    &\hspace{3cm} +\bigg(- b_l^mj_l(kr)\cdot\frac{m}{\sin\theta}Y_l^m+a_l^m\cdot
    q_l(kr)\cdot\frac{\partial{Y_l^m}}{\partial\theta}\Big|_{\phi=0}\bigg)
    \cdot\boldsymbol{e_2}(\theta,0)\Bigg\},\\
     &\nu_2 \wedge(\nabla\wedge \mathbf{E}|_{\widetilde\Pi_2})={\mathbf{i}k}\sum_{l=1}^{\infty}\sum_{m=-l}^{l}\frac{1}{\sqrt{l(l+1)}}\Bigg\{a_l^ml(l+1)p_l(kr)Y_l^m \Big|_{\phi=\phi_0}
\cdot\boldsymbol{e_1}(\theta,\phi_0)\\
    &\hspace{3cm}+\bigg( -b_l^mj_l(kr)\cdot\frac{m}{\sin\theta}Y_l^m \Big|_{\phi=\phi_0}+a_l^m
    q_l(kr)\cdot\frac{\partial{Y_l^m}}{\partial\theta}\Big|_{\phi=\phi_0}\bigg)
    \cdot\boldsymbol{e_2}(\theta,\phi_0)\Bigg\}.
     \end{split}
     \end{equation}
\end{lemma}

 \begin{proof} Using the fact that $\phi=\phi_0$ for $\bfx=(r,\theta, \phi) \in \Pi_2$, it is easy to see that
\begin{equation}
\label{eq:nu2 r}
\begin{split}
 & \nu_2 \wedge (\hat{\boldsymbol{r}}|_{\phi=
\phi_0})=\begin{bmatrix}
	\cos\phi_0\cos\theta\\ \sin\phi_0\sin\theta \\-\sin\theta\end{bmatrix} ,
	\quad\  \nu_2 \wedge (\hat{\boldsymbol\theta}|_{\phi=\phi_0})=\begin{bmatrix}
		-\cos\phi_0\sin\theta\\ -\sin\phi_0\sin\theta \\-\cos\theta
	\end{bmatrix},\quad  \nu_2 \wedge (\hat{\boldsymbol\phi}|_{\phi=\phi_0})={\mathbf 0},
 \end{split}
 \end{equation}
 from which we can derive the second equation of \eqref{mix pi2}. The first equation of \eqref{mix pi2} can be obtained in a similar way.

 Recall that $\mathbf{M}_l^m(\bfx)$ and $\mathbf{N}_l^m(\bfx)$ are defined in  \eqref{ww}. Using the identity
 $
 \nabla\wedge \mathbf{M}_l^m(\bfx)=-\mathbf{i}k \mathbf{N}_l^m(\bfx)$ and $ \nabla\wedge \mathbf{N}_l^m(\bfx)=\mathbf{i}k\mathbf{M}_l^m(\bfx)
 $  (cf. \cite{CK})
 we can obtain that
 \begin{align}
\nabla\wedge \mathbf{E}|_{\widetilde\Pi_1}=\mathbf{i}k\sum_{l=1}^{\infty}\sum_{m=-l}^{l}&\frac{1}{\sqrt{l(l+1)}}\Bigg\{a_l^m\cdot{l(l+1)}p_l(kr)Y_l^m\cdot\nu_1\wedge\hat{\boldsymbol{r}}|_{\phi=0}\notag\\
&+\bigg(-b_l^mj_l(kr)\frac{m}{\sin\theta}Y_l^m+a_l^mq_l(kr)\frac{\partial{Y_l^m}}{\partial\theta}\bigg)\cdot\nu_1\wedge\hat{\boldsymbol\theta}|_{\phi=0}\notag\\
 &+\bigg(-b_l^mj_l(kr)\boldsymbol{\mathrm{i}}\frac{\partial{Y_l^m}}{\partial\theta}+a_l^m\cdot
 q_l(kr)\frac{\boldsymbol{\mathrm{i}}m}{\sin\theta}Y_l^m\bigg)\cdot\nu_1\wedge\hat{\boldsymbol\phi}|_{\phi=0}\Bigg\},\notag\\
 	\nabla\wedge \mathbf{E}|_{\widetilde\Pi_2}=\mathbf{i}k\sum_{l=1}^{\infty}\sum_{m=-l}^{l}&\frac{1}{\sqrt{l(l+1)}}\Bigg\{a_l^m\cdot{l(l+1)}p_l(kr)Y_l^m\cdot\nu_2\wedge\hat{\boldsymbol{r}}|_{\phi=\phi_0}\notag\\
 &+\bigg(-b_l^mj_l(kr)\frac{m}{\sin\theta}Y_l^m+a_l^mq_l(kr)\frac{\partial{Y_l^m}}{\partial\theta}\bigg)\nu_2\wedge\hat{\boldsymbol\theta}|_{\phi=\phi_0}\notag\\
 &+\bigg(-b_l^mj_l(kr)\boldsymbol{\mathrm{i}}\frac{\partial{Y_l^m}}{\partial\theta}+a_l^m\cdot
 q_l(kr)\frac{\boldsymbol{\mathrm{i}}m}{\sin\theta}Y_l^m\bigg)\nu_2\wedge\hat{\boldsymbol\phi}|_{\phi=\phi_0}\Bigg\}.\label{eq:curl E}
 \end{align}
 Combing \eqref{eq:curl E} with \eqref{eq:nu2 r}, together with straightforward though a bit tedious calculations, one can deduce the second equation of \eqref{gg}. The first equation of \eqref{gg} can be shown in a similar manner. 
 
 The proof is complete. 
 \end{proof}

\begin{lemma}
Let $\mathbf E$ be a solution to \eqref{eq:eig}. Recall that $\mathbf{E}$  has the radial wave expansion \eqref{mix pi21} in $B_{\rho_0}(\mathbf{0})$. Consider an edge-corner ${\mathcal E}(\widetilde\Pi_1, \widetilde\Pi_2,\bsl)\Subset \Omega$ associated with $\mathbf{E}$. Recall that $\nu_i$ defined in \eqref{l1} are the outward unit normal vectors to $\Pi_i$, $i=1,2$.  Assume that $ \boldsymbol{ \eta}_1,\boldsymbol{ \eta}_2 $ belong to the class $\mathcal{A}(\bsl)$. Then we have
\begin{align}
    &\nu_1 \wedge (\nabla\wedge\mathbf{E}|_{\widetilde{ \Pi}_1})+\boldsymbol{ \eta}_1(\nu_1 \wedge\mathbf{E}|_{\widetilde{ \Pi}_1})\wedge\nu_1\notag\\
    =&\sum_{l=1}^{\infty}\sum_{m=-l}^{l}  \frac{1}{\sqrt{l(l+1)}}\Bigg\{\bigg(
    \mathbf{i}ka_l^ml(l+1)p_l(kr)Y_l^m -\boldsymbol{ \eta}_1 a_l^mj_l(kr)\frac{m}{\sin\theta}Y_l^m
   \notag \\
    &-\boldsymbol{ \eta}_1 b_l^mq_l(kr)\frac{\partial{Y_l^m}}{\partial\theta}\bigg) \boldsymbol{e_1}(\theta,0)+\bigg(-\mathbf{i}kb_l^mj_l(kr)\frac{m}{\sin\theta}Y_l^m+\boldsymbol{\mathrm{i}}ka_l^m
    q_l(kr)\frac{\partial{Y_l^m}}{\partial\theta} \notag\\
    &
    +\boldsymbol{ \eta}_1 b_l^ml(l+1)p_l(kr)Y_l^m \bigg)\cdot\boldsymbol{e_2}(\theta,0)\Bigg\},\label{ss1}\\
    &\nu_1 \wedge (\nabla\wedge\mathbf{E}|_{\widetilde{ \Pi}_1})+\boldsymbol{ \eta}_1(\nu_1 \wedge\mathbf{E}|_{\widetilde{ \Pi}_1})\wedge\nu_1\notag\\
    =&\sum_{l=1}^{\infty}\sum_{m=-l}^{l}  \frac{1}{\sqrt{l(l+1)}}\Bigg\{\bigg(
    \mathbf{i}ka_l^ml(l+1)p_l(kr)Y_l^m -\boldsymbol{ \eta}_1 a_l^mj_l(kr)\frac{m}{\sin\theta}Y_l^m\notag
    \\
    &-\boldsymbol{ \eta}_1 b_l^mq_l(kr)\frac{\partial{Y_l^m}}{\partial\theta}\bigg) \boldsymbol{e_1}(\theta,0)+\bigg(-\mathbf{i}kb_l^mj_l(kr)\frac{m}{\sin\theta}Y_l^m+\boldsymbol{\mathrm{i}}ka_l^m
    q_l(kr)\frac{\partial{Y_l^m}}{\partial\theta}\notag \\
    &
    +\boldsymbol{ \eta}_1 b_l^ml(l+1)p_l(kr)Y_l^m \bigg)\cdot\boldsymbol{e_2}(\theta,0)\Bigg\},\label{ss1}
\end{align}

    and
\begin{equation}\label{ss2}
\begin{split}
    &\nu_2 \wedge (\nabla\wedge\mathbf{E}|_{\widetilde{ \Pi}_2})+\boldsymbol{ \eta}_2(\nu_2 \wedge\mathbf{E}|_{\widetilde{ \Pi}_2})\wedge\nu_2\\
    =&\sum_{l=1}^{\infty}\sum_{m=-l}^{l} \frac{1}{\sqrt{l(l+1)}}\Bigg\{\bigg(
   \mathbf{i}ka_l^ml(l+1)p_l(kr)Y_l^m  -\boldsymbol{ \eta}_2 a_l^mj_l(kr)\frac{m}{\sin\theta}Y_l^m\\
   & -\boldsymbol{ \eta}_2 b_l^mq_l(kr)\frac{\partial{Y_l^m}}{\partial\theta}\bigg)
   \boldsymbol{e_1}(\theta,\phi_0) +\bigg(-\mathbf{i}kb_l^mj_l(kr)\frac{m}{\sin\theta}Y_l^m+\boldsymbol{\mathrm{i}}ka_l^m
    q_l(kr)\frac{\partial{Y_l^m}}{\partial\theta}\\
    &+\boldsymbol{ \eta}_2 b_l^ml(l+1)p_l(kr)Y_l^m\bigg)
   \cdot\boldsymbol{e_2}(\theta,\phi_0)\Bigg\},
   \end{split}
      \end{equation}
      where $\boldsymbol{e_1}(\theta, 0)$, $\boldsymbol{e_2}(\theta,0)$   $\boldsymbol{e_1}(\theta,\phi_0)$ and $\boldsymbol{e_2}(\theta,\phi_0)$ are defined in \eqref{eq:e1e2}.
      \end{lemma}

      \begin{proof} Recall that $\nu_2$ is defined in \eqref{l1},  $\hat{\boldsymbol{r}}, \hat{\boldsymbol{\theta}}$ and $\hat{\boldsymbol{\phi}}$ are given by \eqref{w1}. Then it is easy to see that
     \begin{equation}
  \label{eq:22 cross}
     \begin{split}
     (\nu_2\wedge\hat{\boldsymbol{r}})\wedge\nu_2&=(\cos\phi_0\sin\theta,\sin\phi_0\sin\theta,\cos\theta)^\top ,\\
      (\nu_2\wedge \hat{\boldsymbol{\theta}})\wedge\nu_2&=(\cos\phi_0\cos\theta,\sin\phi_0\cos\theta,-\sin\theta)^\top,\quad (\nu_2\wedge\hat{\boldsymbol{\phi}})\wedge\nu_2 ={\mathbf 0}.
     \end{split}
     \end{equation}
Using \eqref{mix pi21} and \eqref{gg}, we can derive that
\begin{align}
&\nu_2 \wedge (\nabla\wedge\mathbf{E}|_{\widetilde{ \Pi}_2})+\boldsymbol{ \eta} _2(\nu_2\wedge\mathbf{E}|_{\widetilde{ \Pi}_2})\wedge\nu_2\notag\\
     =& \sum_{l=1}^{\infty}\sum_{m=-l}^{l}\frac{1}{\sqrt{l(l+1)}}\Bigg\{\mathbf{i}k\Bigg(a_l^ml(l+1)p_l(kr)Y_l^m
\cdot\nu_2 \wedge \hat{\boldsymbol{r}}\notag \\
    &+\bigg( -b_l^mj_l(kr)\cdot\frac{m}{\sin\theta}Y_l^m+a_l^m
    q_l(kr)\frac{\partial{Y_l^m}}{\partial\theta}\bigg)
    \cdot\nu_2 \wedge \hat{\boldsymbol\theta}\notag \\
    &+\bigg(-b_l^mj_l(kr)\mathbf{i}\frac{\partial{Y_l^m}}{\partial\theta}+a_l^m
 q_l(kr)\frac{\mathbf{i}m}{\sin\theta}Y_l^m\bigg)\nu_2 \wedge \hat{\boldsymbol\phi}\Bigg)\notag \\
&-\boldsymbol{ \eta} _2\Bigg(\bigg(b_l^m\cdot{l(l+1)}p_l(kr)\cdot Y_l^m\bigg)\cdot(\nu_2\wedge\hat{\mathbf{r}})\wedge\nu_2\notag\\
&+ \bigg(a_l^m j_l\big(kr\big)\frac{m}{\sin\theta}Y_l^m+b_l^m
    q_l(kr) \frac{\partial{Y_l^m}}{\partial\theta}\bigg) (\nu_2\wedge \hat{\boldsymbol{\theta}})\wedge\nu_2 \notag \\
       &+\mathbf{i} \bigg(a_l^m j_l(kr)\frac{\partial{Y_l^m}}{\partial\theta}+\frac{ m b_l^m
     q_l(kr)}{\sin\theta}Y_l^m
     \bigg) (\nu_2\wedge\hat{\boldsymbol{\phi}})\wedge\nu_2\Bigg)\Bigg\}.    \label{eq:23}
\end{align} 
Substituting \eqref{eq:nu2 r} and  \eqref{eq:22 cross} into \eqref{eq:23}, together with straightforward calculations, we can obtain \eqref{ss2}. \eqref{ss1} can be derived in a similar manner.
      \end{proof}

  \section{Vanishing orders for an edge-corner ${\mathcal E} ( \widetilde{ \Pi}_1,  \widetilde{ \Pi}_2,\bsl)$ with $\boldsymbol{\eta}_j\in \mathcal{A}(\bsl)$}\label{sec:5}
  
  In this section, we consider the case that ${\mathcal E} ( \widetilde{ \Pi}_1,  \widetilde{ \Pi}_2,\bsl)$ and edge-corner with both $\boldsymbol{\eta}_1$ and $\boldsymbol{\eta}_2$ belong to the class $\mathcal{A}(\bsl)$. We shall derive the vanishing order of $\mathbf{E}$ to \eqref{eq:eig} at the origin $\mathbf{0}\in\bsl$. The major idea is to make use of the radial wave expansion \eqref{mix pi21} of $\mathbf{E}$ in $B_{\rho_0 }( {\mathbf 0}) $, and to investigate the relationships between $a_n^{\pm 1}$, $a_n^0$ and $b_n^{\pm 1}$, $b_n^0$. Henceforth, according to Definition~\ref{def:class1}, we assume that $\boldsymbol{ \eta}_j$, $j=1,2,$ are given by the following absolutely convergent series at $\mathbf 0\in \bsl$:
\begin{subequations}
	\begin{align}
		\boldsymbol{\eta}_1&=\eta_{1}+\sum_{j=1}^\infty \eta_{1,j}(\theta) r^j, \label{eq:eta1 ex} \\
		\boldsymbol{\eta}_2&=\eta_{2}+\sum_{j=1}^\infty \eta_{2,j}(\theta) r^j \label{eq:eta2 ex}
	\end{align}
\end{subequations}
    where $\eta_{\ell}\in\mathbb{C}\backslash\{0\}$, $\eta_{\ell,j}(\theta)\in C[-\pi, \pi]$ and $r\in [-h,h]$, $\ell=1,2$.  Next, based on the above setting, we derive several critical lemmas.

\begin{lemma}\label{lem:imp pi12}
	Let $\mathbf{E}$ be a a solution to \eqref{eq:eig}, whose radial wave expansion in $B_{\rho_0}(\mathbf{0}) $ is given by \eqref{mix pi21}.  Consider a generalized impedance  edge-corner ${\mathcal E}(\widetilde{ \Pi}_1, \widetilde{ \Pi}_2,\bsl) \Subset \Omega$ with $\angle(\Pi_{1},\Pi_2)=\phi_0=\alpha \pi$, where $\alpha\in(0,2)$ and $\alpha \neq 1$. Suppose that the generalized impedance parameters $ \boldsymbol{ \eta}_j$ on $\widetilde{ \Pi}_j $, $j=1, 2$, are given by \eqref{eq:eta1 ex} and \eqref{eq:eta2 ex} respectively. It holds that
	\begin{subequations}
	 \begin{align}
   0&=\frac{4\mathbf{i}kc_1^1\sin^2\phi_0}{6\sqrt{2}}(a_1^1+a_1^{-1})-\frac{4k c_1^1\sin\phi_0\cos\phi_0}{6\sqrt{2}}(a_1^1-a_1^{-1})-\frac{(\eta_{2}\cos\phi_0+\eta_1)\sqrt{2}c_1^0}{3} b_1^0 ,\label{eq:lem51 a1}  \\
0&=-\frac{4\mathbf{i}k c_1^1\sin\phi_0\cos\phi_0}{6\sqrt{2}}
(a_1^1+a_1^{-1})-\frac{4k c_1^1\sin^2\phi_0}{6\sqrt{2}}(a_1^1-a_1^{-1})-\frac{\eta_{2}\sqrt{2} c_1^0\sin\phi_0}{3}b_1^0,\label{eq:lem51 a2}\\
0&=-\frac{4c_1^1(-\eta_{1}+\eta_{2}\cos\phi_0)}{6\sqrt{2}}(b_1^1+b_1^{-1}) +\frac{4\eta_{2} c_1^1\sin\phi_0\mathbf{i}}{6\sqrt{2}}(b_1^1-b_1^{-1}). \label{eq:beta 3rd}
\end{align}
	\end{subequations}
	Assume that there exists   $n\in \mathbb N \backslash\{1\}$ such that
\begin{equation}\label{eq:lem41 cond}
a_l^0=b_l^0=a_l^{\pm 1}=b_l^{\pm 1}=0, \quad l =1,\ldots, n-1.
\end{equation}
Then we have
	\begin{subequations} 
\begin{align}
	 &\frac{\eta_{1}\sqrt{n(n+1)}c_n^0}{2n+1}b_n^0=\frac{\mathbf{i}kn(n+1)^2c_n^1\sin^2\phi_0}{2(2n+1)\sqrt{n(n+1)}}(a_n^1+a_n^{-1}) \label{eq:52a}
  \\
 &\quad -\frac{kn(n+1)^2c_n^1\sin\phi_0\cos\phi_0}{2(2n+1)\sqrt{n(n+1)}}(a_n^1-a_n^{-1})-\frac{\eta_{2}\sqrt{n(n+1)}c_n^0\cos\phi_0}{2n+1} b_n^0 ,  \notag \\
&\frac{kn(n+1)^2c_n^1}{2(2n+1)\sqrt{n(n+1)}}(a_n^1-a_n^{-1})=-\frac{\mathbf{i}kn(n+1)^2c_n^1\sin\phi_0\cos\phi_0}{2(2n+1)\sqrt{n(n+1)}}
(a_n^1+a_n^{-1})  \label{eq:52b} \\
&\quad +\frac{kn(n+1)^2c_n^1\cos^2\phi_0}{2(2n+1)\sqrt{n(n+1)}}(a_n^1-a_n^{-1})-\frac{\eta_2\sqrt{n(n+1)} c_n^0\sin\phi_0}{2n+1}b_n^0,  \notag \\
&-\frac{\eta_{1} n(n+1)^2c_n^1}{2(2n+1)\sqrt{n(n+1)}}(b_n^1+b_n^{-1})=\frac{\eta_2n(n+1)^2c_n^1\cos\phi_0}{2(2n+1)\sqrt{n(n+1)}}(b_n^1+b_n^{-1})\label{eq:52c} \\
&\quad -\frac{n(n+1)^2\eta_2\sin\phi_0\mathbf{i}}{2(2n+1)\sqrt{n(n+1)}}(b_n^1-b_n^{-1}). \notag
	\end{align}
\end{subequations}
\end{lemma}

\begin{proof} We shall first derive \eqref{eq:52a}, \eqref{eq:52b} and \eqref{eq:52c}. \eqref{eq:lem51 a1}, \eqref{eq:lem51 a2} and \eqref{eq:beta 3rd} can be obtained in a similar way and we shall sketch the corresponding derivations at the end of the proof.

We first note that
\begin{equation}\label{eq:41}
	(\nu_2\wedge\mathbf{E})\wedge\nu_2=-\nu_2\wedge(\nu_2\wedge\mathbf{E})=-\big(\nu_2\cdot(\nu_2\cdot\mathbf{E})-\mathbf{E}(\nu_2\cdot\nu_2)\big)=\mathbf{E}-(\nu_2\cdot\mathbf{E})\cdot\nu_2.
\end{equation}
Hence, we have from 
\begin{equation}
	\nu_2 \wedge (\nabla\wedge\mathbf{E}|_{\widetilde \Pi_2})+\boldsymbol{ \eta}_2(\nu_2 \wedge\mathbf{E}|_{\widetilde \Pi_2})\wedge\nu_2=\mathbf 0,
\end{equation}
that
\begin{equation}\label{eq:43}
	\nu_2\wedge(\nabla\wedge \mathbf{E})|_{\widetilde \Pi_2} +\boldsymbol \eta_2\big(\mathbf{E}|_{\widetilde \Pi_2}-(\nu_2\cdot\mathbf{E}|_{\widetilde \Pi_2})\cdot\nu_2\big)=\mathbf 0.
\end{equation}
Multiplying  the cross product with $\nu_2$ from left on both sides \eqref{eq:43}, by using the fact that
$$\nu_2\wedge\big(\nu_2\wedge(\nabla\wedge \mathbf{E})|_{\widetilde \Pi_2}\big)=\nu_2\cdot\big(\nu_2\cdot(\nabla\wedge \mathbf{E})|_{\widetilde \Pi_2}\big)-(\nu_2\cdot\nu_2)(\nabla\wedge \mathbf{E}) |_{\widetilde \Pi_2},
$$
we can obtain that
\begin{equation}\label{eq:44}
	\big(\nu_2\cdot(\nabla\wedge \mathbf{E})|_{ \widetilde \Pi_2}\big)\nu_2+\boldsymbol \eta_2(\nu_2\wedge\mathbf{E}|_{\widetilde \Pi_2})=\nabla\wedge \mathbf{E} |_{\widetilde \Pi_2}.
\end{equation}
Similarly, since the generalized impedance condition \eqref{eq:imp2} associated with $\boldsymbol{ \eta}_1$  is imposed on $\widetilde {\Pi}_1$, using the above argument, we can deduce that
\begin{equation}\label{eq:45}
	\big(\nu_1\cdot(\nabla\wedge \mathbf{E})|_{\widetilde \Pi_1}\big)\nu_1+\boldsymbol \eta_1(\nu_1\wedge\mathbf{E}|_{\widetilde \Pi_1})=\nabla\wedge \mathbf{E} |_{\widetilde \Pi_1}.
\end{equation}
Since  $\bsl \in \widetilde \Pi_1 \cap  \widetilde \Pi_2$, combing \eqref{eq:44} with \eqref{eq:45}, it yields that
\begin{equation}\label{y1}
 \big(\nu_1\cdot(\nabla\wedge \mathbf{E}|_{\bsl})\big)\nu_1+\boldsymbol \eta_1(\nu_1\wedge\mathbf{E}|_{\bsl})=\big(\nu_2\cdot(\nabla\wedge \mathbf{E}|_{\bsl})\big)\nu_2+\boldsymbol \eta_2(\nu_2 \wedge\mathbf{E}|_{\bsl}).
 \end{equation}



Due to \eqref{eq:lem41 cond},  using \eqref{mix pi2} and \eqref{eq:curl E}, by virtue of  \eqref{uu},   it yields that
\begin{align}
\label{eq:nu2 curl E}
\nabla\wedge\mathbf{E}|_{\widetilde \Pi_2}=&\sum_{l=n}^{\infty}\sum_{m=-l }^{l}\frac{\mathbf{i}k}{\sqrt{l(l+1)}}\Bigg\{a_l^m\cdot{l(l+1)}p_l(kr)c_l^mP_l^{|m|}\cdot\hat{\boldsymbol{r}}|_{\widetilde \Pi_2 }\notag  \\
 &+\bigg(b_l^mj_l(kr)c_l^m\frac{\sgn(m)}{2}\big[P_{l-1}^{|m|+1}(\cos\theta)+(l+|m|-1)(l+|m|)P_{l-1}^{|m|-1}(\cos\theta)\big]\notag \\
 &+a_l^mq_l(kr)c_l^m\frac{1}{2}\big[(l+|m|)(l-|m|+1)P_l^{|m|-1}(\cos\theta)-P_l^{|m|+1}(\cos\theta)\big]\bigg)\hat{\boldsymbol\theta}|_{\widetilde \Pi_2 } \notag \\
 &+\bigg(-b_l^mj_l(kr)\mathbf{i}c_l^m\frac{1}{2}\big[(l+|m|)(l-|m|+1)P_l^{|m|-1}(\cos\theta)-P_l^{|m|+1}(\cos\theta)\big] \notag \\
 &-a_l^m\cdot
 q_l(kr)\mathbf{i} c_l^m\frac{\sgn(m)}{2}\big[P_{l-1}^{|m|+1}(\cos\theta)\notag \\
 &+(l+|m|-1)(l+|m|)P_{l-1}^{|m|-1}(\cos\theta)\big]\bigg)\hat{\boldsymbol\phi}|_{\Pi_2 }\Bigg\},\notag \\
 \end{align}
 and
 \begin{align}\label{eq:312 nu2}
	\nu_2\wedge\mathbf{E}|_{\widetilde \Pi_2}&=\sum_{l=n}^{\infty}\sum_{m=-l }^{l}\Bigg\{\bigg\{-b_l^m\cdot{\sqrt{l(l+1)}}p_l(kr)\cdot c_l^mP_l^{|m|}\bigg\}\cdot\nu_2 \wedge\hat{\boldsymbol{r}}|_{\widetilde \Pi_2 }\notag \\
 &+ \bigg\{a_l^m\cdot j_l\big(kr\big)\frac{1}{\sqrt{l(l+1)}} c_l^m\frac{\sgn(m)}{2} \big[P_{l-1}^{|m|+1}(\cos\theta)+(l+|m|-1)(l+|m|)\notag \\
 &\times P_{l-1}^{|m|-1}(\cos\theta)\big]-b_l^m\cdot
    q_l(kr)\cdot\frac{1}{\sqrt{l(l+1)}}c_l^m\frac{1}{2}\big[(l+|m|)(l-|m|+1)\notag \\
    &\times P_l^{|m|-1}(\cos\theta)-P_l^{|m|+1}(\cos\theta)\big]\bigg\}\cdot\nu_2 \wedge \hat{\boldsymbol\theta}|_{\widetilde \Pi_2 } +\bigg\{-a_l^m\cdot j_l(kr)\frac{\mathbf{i}}{\sqrt{l(l+1)}}\frac{c_l^m }{2}\notag \\
    &\times \big[(l+|m|)(l-|m|+1)  P_l^{|m|-1}(\cos\theta)-P_l^{|m|+1}(\cos\theta)\big]+b_l^m 
     q_l(kr)\frac{\mathbf{i}}{\sqrt{l(l+1)}}\notag \\
     &\times   c_l^m\frac{\sgn(m)}{2}\big[P_{l-1}^{|m|+1}(\cos\theta)+(l+|m|-1)(l+|m|)P_{l-1}^{|m|-1}(\cos\theta)\big]
     \bigg\}\times  \nu_2 \wedge\hat{\boldsymbol\phi}|_{\widetilde \Pi_2 } \Bigg\},
\end{align}
where
\begin{equation}\notag
	\sgn(m)=1\ \mbox{when}\ m>0;\ 0\ \mbox{when}\ m=0; \ -1\ \mbox{when}\ m<0. 
\end{equation}
Recall that if $\bfx \in \bsl $ one has
 \begin{equation}\label{eq:l}
 \theta=\phi=0,	\quad 0\leq r\leq h,
 \end{equation}
where  $r$, $\theta$ and $\phi$ are the spherical coordinates  of $\bfx \in \bsl$ defined in \eqref{eq:x sph}.   It is straightforward to calculate that
 \begin{equation}\label{eq:48}
\begin{split}
 \nu_2\wedge\boldsymbol{\hat{r}}|_{\theta=\phi=0}&=\begin{bmatrix} \cos\phi_0 \\ \sin\phi_0 \\ 0\end{bmatrix},\,  \nu_2\wedge\boldsymbol{\hat{\theta}}|_{\theta=\phi=0}=-\begin{bmatrix} 0 \\ 0 \\ \cos\phi_0\end{bmatrix},\,  \nu_2\wedge\boldsymbol{\hat{\phi}}|_{\theta=\phi=0}=-\begin{bmatrix} 0 \\0 \\\sin\phi_0 \end{bmatrix},  \\
\nu_2\cdot\boldsymbol{\hat{r}}|_{\theta=\phi=0}&=0,\quad \nu_2\cdot\boldsymbol{\hat{\theta}}|_{\theta=\phi=0}=-\sin\phi_0,\quad \nu_2\cdot\boldsymbol{\hat{\phi}}|_{\theta=\phi=0}=\cos\phi_0,
\end{split}
\end{equation}
 where $\boldsymbol{\hat{r}}$, $\boldsymbol{\hat{\theta }}$ and $\boldsymbol{\hat{\phi }}$ are defined in \eqref{eq:x sph}.


Evaluating \eqref{eq:nu2 curl E} and \eqref{eq:312 nu2} at $\bsl$, by virtue of \eqref{eq:plm0} and \eqref{eq:48},  we can derive that	
\begin{equation}
 \label{eq:nu2 curl E416}
\begin{split}
\nabla\wedge\mathbf{E}|_{\bsl}&=\sum_{l=n}^{+\infty}\frac{\mathbf{i}k}{\sqrt{l(l+1)}}\Bigg\{ a_l^0l(l+1)p_l(kr)c_l^0\cdot\hat{\boldsymbol{r}}|_{\theta=\phi=0}   \\
&+\frac{(l+1)l}{2}\bigg((b_l^1-b_l^{-1})\cdot j_l(kr)\cdot c_l^1+(a_l^1+a_l^{-1})q_l(kr)c_l^1
\bigg)\cdot\boldsymbol{\hat{\theta}}|_{\theta=\phi=0} \\
&+\mathbf{i}\frac{(l+1)l}{2}\bigg(-(b_l^1+b_l^{-1})\cdot j_l(kr)c_l^1-(a_l^1-a_l^{-1})q_l(kr)c_l^1\bigg)\cdot\boldsymbol{\hat{\phi}} |_{\theta=\phi=0} \Bigg\},\\
	\nu_2\wedge\mathbf{E}|_{\bsl }&=\sum_{l=n}^{+\infty}\frac{1}{\sqrt{l(l+1)}}\Bigg\{\Bigg(-b_l^0l(l+1)p_l(kr)c_l^0\cdot\nu_2\wedge\boldsymbol{\hat{r}}|_{\theta=\phi=0}\\
&-\frac{(l+1)l}{2}\bigg(-(a_l^1-a_l^{-1}) j_l(kr)c_l^1+(b_l^1+b_l^{-1})q_l(kr)c_l^1\bigg)\cdot\nu_2\wedge\boldsymbol{\hat{\theta}}|_{\theta=\phi=0}\\
&-\mathbf{i}\frac{(l+1)l}{2}\bigg((a_l^1+a_l^{-1})j_l(kr)c_l^1-(b_l^1-b_l^{-1})q_l(kr)c_l^1\bigg)
\cdot\nu_2\cdot\boldsymbol{\hat{ \phi  }}|_{\theta=\phi=0}\Bigg)\Bigg\}.
\end{split}
\end{equation}	
Therefore, from \eqref{eq:nu2 curl E416} we obtain that
\begin{align}
&\nu_2^\top (\nabla\wedge\mathbf{E}|_{\bsl})
\nu_2+\boldsymbol{ \eta}_2(\nu_2\wedge\mathbf{E}|_{\bsl})\notag\\
=&\sum_{l=n}^{+\infty}\frac{1}{\sqrt{l(l+1)}}\Bigg\{\mathbf{i}k\Bigg[ -\sin\phi_0\frac{ c_l^1(l+1)l}{2}\bigg((b_l^1-b_l^{-1}) j_l(kr) +(a_l^1+a_l^{-1})q_l(kr)c_l^1
\bigg)\notag\\
&-\cos\phi_0\frac{\mathbf{i}(l+1)l}{2} \times \bigg((b_l^1+b_l^{-1}) j_l(kr)c_l^1 +(a_l^1-a_l^{-1})q_l(kr)c_l^1\bigg)\Bigg] \notag \\
&\quad\times \begin{bmatrix} -\sin\phi_0\\ \cos\phi_0 \\ 0\end{bmatrix} +\boldsymbol{ \eta}_2\bigg[ -b_l^0l(l+1)p_l(kr)c_l^0  \begin{bmatrix}\cos\phi_0\\ \sin\phi_0\\0\end{bmatrix}\notag\\
&-\frac{(l+1)l}{2} \bigg(-(a_l^1-a_l^{-1}) j_l(kr)c_l^1 +(b_l^1+b_l^{-1}) q_l(kr)c_l^1\bigg) \begin{bmatrix} 0\\0\\-\cos\phi_0\end{bmatrix} \notag\\
&-\frac{\mathbf{i}(l+1)l}{2}\bigg((a_l^1+a_l^{-1})  j_l(kr)c_l^1  -(b_l^1-b_l^{-1})q_l(kr)c_l^1\bigg)
\cdot\begin{bmatrix}0\\0\\-\sin\phi_0\end{bmatrix}\bigg] \Bigg\}.  \label{eq:411 x31}
\end{align}

\begin{equation}
\label{eq:411 x31}
\begin{split}
&\nu_2^\top (\nabla\wedge\mathbf{E}|_{\bsl})
\nu_2+\boldsymbol{ \eta}_2(\nu_2\wedge\mathbf{E}|_{\bsl})\\
=&\sum_{l=n}^{+\infty}\frac{1}{\sqrt{l(l+1)}}\Bigg\{\mathbf{i}k\Bigg[ -\sin\phi_0\frac{ c_l^1(l+1)l}{2}\bigg((b_l^1-b_l^{-1}) j_l(kr) +(a_l^1+a_l^{-1})q_l(kr)c_l^1
\bigg)\\
&-\cos\phi_0\frac{\mathbf{i}(l+1)l}{2} \times \bigg((b_l^1+b_l^{-1}) j_l(kr)c_l^1 +(a_l^1-a_l^{-1})q_l(kr)c_l^1\bigg)\Bigg]  \\
&\quad\times \begin{bmatrix} -\sin\phi_0\\ \cos\phi_0 \\ 0\end{bmatrix} +\boldsymbol{ \eta}_2\bigg[ -b_l^0l(l+1)p_l(kr)c_l^0  \begin{bmatrix}\cos\phi_0\\ \sin\phi_0\\0\end{bmatrix}\\
&-\frac{(l+1)l}{2} \bigg(-(a_l^1-a_l^{-1}) j_l(kr)c_l^1 +(b_l^1+b_l^{-1}) q_l(kr)c_l^1\bigg) \begin{bmatrix} 0\\0\\-\cos\phi_0\end{bmatrix} \\
&-\frac{\mathbf{i}(l+1)l}{2}\bigg((a_l^1+a_l^{-1})  j_l(kr)c_l^1  -(b_l^1-b_l^{-1})q_l(kr)c_l^1\bigg)
\cdot\begin{bmatrix}0\\0\\-\sin\phi_0\end{bmatrix}\bigg] \Bigg\}.  
\end{split}
\end{equation}	
Using a similar argument for deriving \eqref{eq:411 x31}, we have
\begin{equation}\label{eq:412 x3}
\begin{split}
& \nu_1^\top  (\nabla\wedge \mathbf{E}|_{\bsl})  \nu_1+\eta_1(\nu\wedge\mathbf{E}|_{\bsl})\\
=&-\sum_{l=n}^{+\infty}\Bigg\{ \frac{kl(l+1)}{2\sqrt{l(l+1)}}\bigg[(b_l^1+b_l^{-1})\cdot j_l(kr)\cdot c_l^1 +(a_l^1-a_l^{-1}) \notag  q_l(kr) c_l^1\bigg]\begin{bmatrix}0\\-1\\0\end{bmatrix}\\
&+\boldsymbol{ \eta}_1\Bigg(-b_l^0\sqrt{l(l+1)}p_l(kr) c_l^0  \begin{bmatrix}-1\\0\\0\end{bmatrix}  +\frac{l(l+1)}{2\sqrt{l(l+1)}}\bigg((a_l^1-a_l^{-1})  j_l(kr)c_l^1 \\
 &-(b_l^1+b_l^{-1})  q_l(kr) c_l^1 \bigg) \begin{bmatrix}0\\0\\1\end{bmatrix}\Bigg)\Bigg\}. 
 \end{split}
\end{equation}

Note that $\boldsymbol{\eta}_\ell$, $\ell=1,2$, have the expansions \eqref{eq:eta1 ex} and \eqref{eq:eta2 ex} respectively, where the coefficients of $r^0$ in \eqref{eq:eta1 ex} and \eqref{eq:eta2 ex}  are the non-zero numbers $\eta_1$ and $\eta_2$.  From Remark \ref{i2}, it is easy to see that the lowest order  of \eqref{eq:411 x31}  and \eqref{eq:412 x3} with respect to the power of $r$ is $n-1$, which  is contributed by $p_{n}\big(kr\big)$ and $q_{n}\big(kr\big)$ in \eqref{eq:411 x31}  and \eqref{eq:412 x3}.  Substituting \eqref{eq:411 x31}  and \eqref{eq:412 x3} into \eqref{y1}, and comparing the coefficients of $r^{n-1}$ on both sides of  the first, second and third component of \eqref{y1} respectively, we can derive \eqref{eq:52a}, \eqref{eq:52b} and \eqref{eq:52c}.

 We can derive \eqref{eq:lem51 a1}, \eqref{eq:lem51 a2} and \eqref{eq:beta 3rd}  by similar arguments for \eqref{eq:52a}, \eqref{eq:52b} and \eqref{eq:52c}. Indeed, the Fourier expansions of \eqref{eq:411 x31} and \eqref{eq:412 x3} can be rewritten  with the starting  summation index  $n=1$. Hence  we can obtain  \eqref{eq:lem51 a1}, \eqref{eq:lem51 a2} and \eqref{eq:beta 3rd} by comparing the coefficients of $r^0$ on both sides of  \eqref{y1} by virtue of  \eqref{eq:411 x31} and \eqref{eq:412 x3} .
 
 The proof is complete. 
\end{proof}

	\begin{lemma}\label{lem:imp pi2}
Under the same setup in Lemma~\ref{lem:imp pi12}, it holds that
	\begin{align}
0&=-\frac{4 \eta_2 c_1^1\cos^2\phi_0}{6\sqrt{2}}(b_1^1+b_1^{-1})+\frac{4\mathbf{i}\eta_2 c_1^1\sin\phi_0\cos\phi_0}{6\sqrt{2}}(b_1^1-b_1^{-1})+\frac{\mathbf{i}k\sqrt{2}c_1^0\cos\phi_0}{3}a_1^0, \label{eq:517a} \\
0&=\frac{4\eta_2 c_1^1\sin\phi_0\cos\phi_0}{6\sqrt{2}}(b_1^1+b_1^{-1})+ \frac{4\mathbf{i}\eta_2 c_1^1\sin^2\phi_0}{6\sqrt{2}}(b_1^1-b_1^{-1})+\frac{\mathbf{i}k\sqrt{2}c_1^0\sin\phi_0}{3}a_1^0,\label{eq:517b} \\
		0&=\frac{4\mathbf{i}k c_1^1\cos\phi_0}{6\sqrt{2}}(a_1^1+a_1^{-1})+\frac{4k c_1^1\sin\phi_0}{6\sqrt{2}}(a_1^1-a_1^{-1})+ \frac{\eta_2 \sqrt{2}c_1^0}{3}b_1^0\label{eq:517c}.
	\end{align}
Furthermore, if we assume that there exists   $n\in \mathbb N \backslash\{1\}$ such that \eqref{eq:lem41 cond}  is fulfilled, then it holds that

\begin{subequations}
	\begin{align}
	0=&\frac{\mathbf{i}k\sqrt{n(n+1)}c_n^0\cos\phi_0}{2n+1}a_n^0-\frac{\eta_2n(n+1)^2c_n^1\cos^2\phi_0}{2(2n+1)\sqrt{n(n+1)}}(b_n^1+b_n^{-1})\label{eq:z14}
\\
&+\frac{\mathbf{i} \eta_2n(n+1)^2c_n^1\sin\phi_0\cos\phi_0}{2(2n+1)\sqrt{n(n+1)}}(b_n^1-b_n^{-1}), \notag\\
0=&\frac{\mathbf{i}kc_n^0\sqrt{n(n+1)}\sin\phi_0}{2n+1}a_n^0+\frac{\eta_2n(n+1)^2c_n^1\sin\phi_0\cos\phi_0}{2(2n+1)\sqrt{n(n+1)}}(b_n^1+b_n^{-1}) \label{eq:z15}
\\
&+\frac{\mathbf{i}\eta_2n(n+1)^2c_n^1\sin^2\phi_0}{2(2n+1)\sqrt{n(n+1)}}(b_n^1-b_n^{-1}),
\notag \\
0=&\frac{\mathbf{i}kn(n+1)^2c_n^1\cos\phi_0}{2(2n+1)\sqrt{n(n+1)}}(a_n^1+a_n^{-1})+\frac{kn(n+1)^2c_n^1\sin\phi_0}{2(2n+1)\sqrt{n(n+1)}}(a_n^1-a_n^{-1}) \label{eq:z16}
 \\
&+ \frac{\eta_2 \sqrt{n(n+1)}c_n^0}{2n+1}b_n^0.  \notag
	\end{align}
\end{subequations}
\end{lemma}

\begin{proof}

We first prove \eqref{eq:z14}, \eqref{eq:z15} and \eqref{eq:z16}. Since the generalized impedance condition \eqref{eq:imp2} associated with $\boldsymbol{ \eta}_2$  is imposed on $\widetilde {\Pi}_2$, we have
\begin{equation}\label{z1}
\nu_2\wedge(\nabla\wedge \mathbf{E}|_{\bsl})+ \boldsymbol{ \eta}_2(\nu_2\wedge\mathbf{E}|_{\bsl})\wedge\nu_2=\mathbf{0}.
\end{equation}
Here, we recall \eqref{eq:l}. Under the assumption \eqref{eq:lem41 cond}, using  \eqref{uu}, \eqref{eq:plm0} and \eqref{ss2}, 
we can obtain that
\begin{align}
&\nu_2\wedge(\nabla\wedge\mathbf{E}|_{\bsl})+\eta_2(\nu_2\wedge\mathbf{E}|_{\bsl})\wedge\nu_2=\sum_{l=n}^{+\infty}\Bigg\{\mathbf{i}k\bigg\{a_l^0\sqrt{l(l+1)}p_l(kr)c_l^0\cdot(\cos\phi_0,\sin\phi_0,0)^{\top} \notag \\
&+\frac{l(l+1)}{2\sqrt{l(l+1)}}\bigg((b_l^1-b_l^{-1})j_l(kr)c_l^1+(a_l^1+a_l^{-1}) q_l(kr)c_l^1\bigg) \begin{bmatrix}0\\0\\-\cos\phi_0\end{bmatrix}-\frac{l(l+1)\mathbf{i}}{2\sqrt{l(l+1)}}\bigg((b_l^1+b_l^{-1})\notag \\
&\times  j_l(kr)c_l^1+(a_l^1-a_l^{-1}) q_l(kr)c_l^1\bigg)\begin{bmatrix}0\\0\\-\sin\phi_0\end{bmatrix}\bigg\}+\boldsymbol{ \eta}_2\bigg\{-b_l^0\sqrt{l(l+1)}p_l(kr)c_l^0 \begin{bmatrix}0\\0\\1\end{bmatrix}+\frac{l(l+1)}{2\sqrt{l(l+1)}}\notag \\
&\times  \bigg((a_l^1-a_l^{-1})j_l(kr)c_l^1+(b_l^1+b_l^{-1})q_l(kr)c_l^1\bigg) \times \begin{bmatrix}\cos^2\phi_0\\-\sin\phi_0\cos\phi_0\\0\end{bmatrix}+\frac{l(l+1)\mathbf{i}}{2\sqrt{l(l+1)}}\bigg((a_l^1+a_l^{-1})\notag \\
&\times  j_l(kr)c_l^1-(b_l^1-b_l^{-1})q_l(kr)c_l^1\bigg)\times  \begin{bmatrix}\sin\phi_0\cos\phi_0\\ \sin^2\phi_0\\0\end{bmatrix}\bigg\}\Bigg\}. \label{j2}
\end{align}
Note that $\boldsymbol{\eta}_\ell$, $\ell=1,2$, have the expansions \eqref{eq:eta1 ex} and \eqref{eq:eta2 ex} respectively, where the coefficients of $r^0$ in \eqref{eq:eta1 ex} and \eqref{eq:eta2 ex}   are non zero number $\eta_1$ and $\eta_2$.  In view of Remark \ref{i2}, we know  that the lowest order  of \eqref{j2} with respect to the power of $r$ is $n-1$, which  is contributed by $p_{n}\big(kr\big)$ and $q_{n}\big(kr\big)$ in \eqref{j2}.  Substituting \eqref{eq:411 x31}  and \eqref{j2} into \eqref{z1}, comparing the coefficients of $r^{n-1}$ on both sides of  the first, the second, the third, component of \eqref{z1} respectively, we derive that \eqref{eq:z14}, \eqref{eq:z15} and \eqref{eq:z16}.

We can derive \eqref{eq:517a}, \eqref{eq:517b} and \eqref{eq:517c} by following similar arguments in deriving \eqref{eq:z14}, \eqref{eq:z15} and \eqref{eq:z16}. Indeed, the Fourier expansions of \eqref{j2} can be rewritten  with the starting  summation index  $n=1$. Hence  we can obtain \eqref{eq:517a}, \eqref{eq:517b} and \eqref{eq:517c} by comparing the coefficients of $r^0$ on both sides of  \eqref{z1} by virtue of \eqref{j2}.
\end{proof}


\begin{lemma}\label{base2}
Under the same setup in Lemma~\ref{lem:imp pi12}, one has the following linear relations:
\begin{equation}
\left\{\begin{split}
 &\beta^1_{11}(b_1^1+b_1^{-1})+\beta^1_{12}(b_1^1-b_1^{-1})+\beta^1_{13}a_1^0=0,
   \\
   &\beta^1_{21}(b_1^1+b_1^{-1})+\beta^1_{22}(b_1^1-b_1^{-1})+\beta^1_{23}a_1^0=0,
    \\
   &\beta^1_{31}(b_1^1+b_1^{-1})+\beta^1_{32}(b_1^1-b_1^{-1})+\beta^1_{33}a_1^0=0,\label{eq:lem51 1}
 \end{split}\right.
   \end{equation}
where
\begin{equation}\label{eq:matrix entry}
\begin{split}
\beta_{11}^1&=-\frac{4 \eta_2 c_1^1\cos^2\phi_0}{6\sqrt{2}},\quad
\beta_{12}^1=\frac{4\mathbf{i}\eta_2 c_1^1\sin\phi_0\cos\phi_0}{6\sqrt{2}},\quad\beta_{13}^1=\frac{\mathbf{i}k\sqrt{2}c_1^0\cos\phi_0}{3}, \\
\beta_{21}^1&=\frac{4\eta_2 c_1^1\sin\phi_0\cos\phi_0}{6\sqrt{2}},\quad \beta_{22}^1=\frac{4\mathbf{i}\eta_2 c_1^1\sin^2\phi_0}{6\sqrt{2}}, \quad \beta_{23}^1=\frac{\mathbf{i}k\sqrt{2}c_1^0\sin\phi_0}{3},\\
 \beta_{31}^1&=-\frac{4c_1^1(-\eta_1+\eta_2\cos\phi_0)}{6\sqrt{2}} , \quad \beta_{32}^1=\frac{4\eta_2 c_1^1\sin\phi_0\mathbf{i}}{6\sqrt{2}},\quad\beta_{33}^1=0.
\end{split}
\end{equation}
If we assume that there exists $n\in \mathbb N \backslash\{1\}$ such that \eqref{eq:lem41 cond} is fulfilled, then one has that
\begin{equation}\label{eq:lem51 2}
\left\{\begin{split}
 &\beta^n_{11}(b_n^1+b_n^{-1})+\beta^n_{12}(b_n^1-b_n^{-1})+\beta^n_{13}a_n^0=0,
   \\
   &\beta^n_{21}(b_n^1+b_n^{-1})+\beta^n_{22}(b_n^1-b_n^{-1})+\beta^n_{23}a_n^0=0,
    \\
   &\beta^n_{31}(b_n^1+b_n^{-1})+\beta^n_{32}(b_n^1-b_n^{-1})+\beta^n_{33}a_n^0=0,
\end{split}\right.
\end{equation}
where
\begin{align*}
\beta_{11}^n&=-\frac{\eta_2n(n+1)^2c_n^1\cos^2\phi_0}{2(2n+1)\sqrt{n(n+1)}},\\
\beta_{12}^n&=\frac{\mathbf{i}\eta_2n(n+1)^2c_n^1\sin\phi_0\cos\phi_0}{2(2n+1)\sqrt{n(n+1)}},\quad\beta_{13}^n=\frac{\mathbf{i}k\sqrt{n(n+1)}c_n^0\cos\phi_0}{2n+1}, \\
\beta_{21}^n&=\frac{\eta_2n(n+1)^2c_n^1\sin\phi_0\cos\phi_0}{2(2n+1)\sqrt{n(n+1)}},\quad \beta_{22}^n=\frac{\mathbf{i}\eta_2n(n+1)^2c_n^1\sin^2\phi_0}{2(2n+1)\sqrt{n(n+1)}}, \\
\beta_{23}^n&=\frac{\mathbf{i}k\sqrt{n(n+1)}c_n^0\sin\phi_0}{2n+1},\quad {\beta_{31}^n=-\frac{n(n+1)^2c_n^1(-\eta_1+\eta_2\cos\phi_0)}{2(2n+1)\sqrt{n(n+1)}}}, \\
\beta_{32}^n&=\frac{\mathbf{i}\eta_2n(n+1)^2c_n^1\sin\phi_0}{2(2n+1)\sqrt{n(n+1)}},\quad\beta_{33}^n=0.
\end{align*}
Furthermore, if
$	\alpha \neq \frac{1}{2}$ and $	\alpha \neq \frac{3}{2},$
then it holds that
\begin{equation}\label{eq:lem41 zero}
a_n^0=b_n^{\pm1}=0.
\end{equation}


\end{lemma}


\begin{proof}

Combining  \eqref{eq:517a}, \eqref{eq:517b} with \eqref{eq:beta 3rd}, we can obtain \eqref{eq:lem51 1}. Similarly, by virtue of  \eqref{eq:52c}, \eqref{eq:z14} and \eqref{eq:z15}, we can derive \eqref{eq:lem51 2}. After straightforward  calculations, it can be verified that the determinant of  coefficients matrices  \eqref{eq:lem51 2} is given by
 \begin{equation}\label{eq:det AB}
\begin{split}
 \left|{\mathcal B}_n
\right| &=-k\eta_2^2\left(\frac{n+1}{2n+1}\right)^3\frac{n\sqrt{n(n+1)}}{2}\big(c_n^1\big)^2c_n^0\sin^2\phi_0 \cos^2\phi_0,
\end{split}
\end{equation}
where $c_n^0,c_n^1$ are nonzero constants  defined in \eqref{sphe harmonic}. Recall that $\boldsymbol{\eta}_\ell$, $\ell=1,2$, have the expansions \eqref{eq:eta1 ex} and \eqref{eq:eta2 ex} respectively, where the coefficients of $r^0$ in \eqref{eq:eta1 ex} and \eqref{eq:eta2 ex}   are non zero number $\eta_1$ and $\eta_2$. 
 Since $\phi_0=\alpha \pi \neq \pi/2$, $\phi_0=\alpha \pi \neq 3\pi/2$,  $\eta_2 \neq 0$ and $k\in \R_+$, we conclude that  ${\mathcal B}_n$ are nonsingular, which readily implies \eqref{eq:lem41 zero}.
\end{proof}


The following  two important lemmas reveal the recursive relationships for $a_n^{\pm m}$ and $b_n^{\pm m}$, where $m=0,1,\ldots, n$, which will be used to  characterize the vanishing order of $\mathbf E$ with respect to the the corresponding dihedral angle of the edge-corner  ${\mathcal E}(\widetilde \Pi_1, \widetilde \Pi_2,\bsl) \Subset \Omega$ in   Theorem \ref{th:two imp}.

\begin{lemma}\label{lem:54}
Let $\mathbf{E}$ be a a solution to \eqref{eq:eig}, whose radial wave expansion in $B_{\rho_0}(\mathbf{0}) $ is given by \eqref{mix pi21}.  Consider a generalized impedance  edge-corner ${\mathcal E}(\widetilde{ \Pi}_1, \widetilde{ \Pi}_2,\bsl) \Subset \Omega$ with $\angle(\Pi_{1},\Pi_2)=\phi_0=\alpha \pi$, where $\alpha\in(0,2)$ and $\alpha \neq 1$. Suppose that the generalized impedance parameters $ \boldsymbol{ \eta}_j$ on $\widetilde{ \Pi}_j $, $j=1, 2$, are given by \eqref{eq:eta1 ex} and \eqref{eq:eta2 ex} respectively.  Assume that there exists   $n\in \mathbb N \backslash\{1\}$ such that
\begin{equation}\label{eq:lem54 cond}
a_l^{m}=b_l^{m}=0, \quad l =1,\ldots, n-1, \mbox{ and } m\in [l]_0.
\end{equation}
Then we have the following  recursive  linear equations:
   \begin{equation}\label{c2}
   \left\{
   \begin{split}
   0=&\mathbf{i}k\frac{\sqrt{n(n+1)}}{2n+1}c_n^0a_n^0-\frac{\eta_1(n+1) }{2(2n+1) }\frac{c_n^1(n+1)n}{\sqrt{n(n+1)}}(b_n^1+b_n^{-1}),
   \\
   0=&\mathbf{i}k\frac{\sqrt{n(n+1)}}{2n+1}c_n^1(a_n^1+a_n^{-1})-\frac{\eta_1(n+1)}{2(2n+1)}\frac{c_n^2}{\sqrt{n(n+1)}}(n+2)(n-1)(b_n^2+b_n^{-2})
    \\
&+ \frac{\eta_1 (n+1) }{2n+1}\frac{c_n^0}{\sqrt{n(n+1)}}b_n^0, 
 \\
0=&\mathbf{i}k\frac{\sqrt{n(n+1)}}{2n+1}c_n^{m}(a_n^{m}+a_n^{-m})- \frac{\eta_1(n+1)}{2(2n+1)}\frac{c_n^{m+1}}{\sqrt{n(n+1)}} (n+m+1)(n-m) 
 \\
&\times (b_n^{m+1}+b_n^{-(m+1)})  +\frac{\eta_1(n+1)}{2(2n+1) }\frac{c_n^{m-1}}{\sqrt{n(n+1)}}(b_n^{m-1}+b_n^{-(m-1)}), \, m=2,3,\ldots,n-1,
 \\
0=&\mathbf{i}k\frac{\sqrt{n(n+1)}}{2n+1}c_n^n(a_n^n+a_n^{-n})+\frac{\eta_1(n+1)}{2(2n+1) }\frac{c_n^{n-1}}{\sqrt{n(n+1)}}(b_n^{n-1}+b_n^{-(n-1)}),
\end{split}
\right. 	
\end{equation}
and
\begin{equation}\label{m2}
\left\{
\begin{split}
0=&\mathbf{i}k\frac{n+1}{2(2n+1)}\frac{c_n^1}{\sqrt{n(n+1)}}(n+1)n(a_n^1+a_n^{-1})+\eta_1 \frac{c_n^0\sqrt{n(n+1)}}{2n+1}b_n^0,\\
0=&\mathbf{i}k\frac{n+1}{2(2n+1) }\frac{c_n^2}{\sqrt{n(n+1)}}(n+2)(n-1)(a_n^2+a_n^{-2})-\mathbf{i}k \frac{n+1}{2n+1}\frac{c_n^0 }{\sqrt{n(n+1)}}a_n^0\\
&+\eta_1\frac{c_n^1 \sqrt{n(n+1)}}{2n+1}(b_n^1+b_n^{-1}), \\
0=&\mathbf{i}k\frac{n+1}{2(2n+1)}\frac{c_n^m}{\sqrt{n(n+1)}} (n+m)(n-m+1)(a_n^m+a_n^{-m})-\mathbf{i}k\cdot\frac{n+1}{2(2n+1)} \frac{c_n^{m-2}}{\sqrt{n(n+1)}}\\
&\times (a_n^{m-2}+a_n^{-(m-2)})+\eta_1\frac{c_n^{m-1}\sqrt{n(n+1)}}{2n+1}(b_n^{m-1}+b_n^{-(m-1)}),\quad m=3,4,\ldots,n, \\
0=&-\mathbf{i}k\frac{n+1}{2(2n+1)}\frac{c_n^{n-1}}{\sqrt{n(n+1)}} (a_n^{n-1}+a_n^{-(n-1)})+\eta_1\frac{c_n^n \sqrt{n(n+1)}}{2n+1}(b_n^n+b_n^{-n})=0.
\end{split}
\right.
\end{equation}
\end{lemma}

\begin{proof}
Since the generalized impedance condition \eqref{eq:imp2} associated with $\boldsymbol{ \eta}_1$  is imposed on $\widetilde {\Pi}_1$,  substituting \eqref{eq:lem54 cond}  into  \eqref{ss1}, by virtue of \eqref{uu},  we derive that
\begin{equation}\label{eq:im bd1n}
\begin{split}
\mathbf{0}
    =
    &\sum_{l=n}^{\infty}\sum_{m=-l }^{l}  \frac{1}{\sqrt{l(l+1)}}\Bigg\{\bigg(
    \mathbf{i}ka_l^ml(l+1)p_l(kr)c_l^mP_l^m+\boldsymbol{ \eta}_1 a_l^mj_l(kr)    c_l^m  \frac{\sgn(m)}{2} \big[P_{l-1}^{|m|+1}(\cos\theta)\\
    &+(l+|m|-1)(l+|m|)P_{l-1}^{|m|-1}(\cos\theta)\big]-\boldsymbol{ \eta}_1 b_l^mq_l(kr)\frac{c_l^m}{2}\big[(l+|m|)(l-|m|+1)\\
    &\times P_l^{|m|-1}(\cos\theta)-P_l^{|m|+1}(\cos\theta)\big]\bigg) \boldsymbol{e_1}(\theta,0) +\bigg(\mathbf{i}kb_l^m j_l(kr)  c_l^m  \frac{\sgn(m)}{2} \big[P_{l-1}^{|m|+1}(\cos\theta)\\
    &+(l+|m|-1)(l+|m|)P_{l-1}^{|m|-1}(\cos\theta)\big]+\mathbf{i}ka_l^m
    q_l(kr) \frac{c_l^m }{2}\big[(l+|m|)(l-|m|+1)\\
    &\times P_l^{|m|-1}(\cos\theta)-P_l^{|m|+1}(\cos\theta)\big]
    +\boldsymbol{ \eta}_1 b_l^ml(l+1)p_l(kr)c_l^mP_l^m \bigg) \boldsymbol{e_2}(\theta,0)\Bigg\},
    \end{split}
      \end{equation}
      where $\boldsymbol{e_1}(\theta, 0)$, $\boldsymbol{e_2}(\theta,0)$   $\boldsymbol{e_1}(\theta,\phi_0)$ and $\boldsymbol{e_2}(\theta,\phi_0)$ are defined in \eqref{eq:e1e2}.

  Recall that $\boldsymbol{\eta}_\ell$, $\ell=1,2$, have the expansion \eqref{eq:eta1 ex} and \eqref{eq:eta2 ex} respectively, where the coefficients of $r^0$ in \eqref{eq:eta1 ex} and \eqref{eq:eta2 ex}   are non zero number $\eta_1$ and $\eta_2$. The lowest order term in \eqref{eq:im bd1n} with respect to the power of $r$ is $r^{n-1}$, which is contributed by $p_n(kr)$ and $q_n(kr)$ from Remark \ref{i2}. Furthermore, it is noted that coefficients of $r^{n-1}$ in $  p_n(kr)$ and $  q_n(kr)$  are $\frac{k^{n-1} }{(2n+1) (2n-1)!!} $ and $\frac{(n+1)k^{n-1} }{(2n+1) (2n-1)!!} $ respectively. Due to the fact that  $ \boldsymbol{e}_{1}\left(\theta, \phi\right)$ and $   \boldsymbol{e}_{2}\left(\theta, \phi\right)$ are linearly independent for any $\theta$ and $\phi$,  from   Lemma \ref{lem:coeff0}, comparing the coefficient of $r^{n-1}$ both sides of \eqref{eq:im bd1n}  associated with $ \boldsymbol{e}_{1}\left(\theta, \phi\right)$ for $\phi=0$, we have
  \begin{align}
 0= &\mathbf{i}k\sum_{m=-n\atop m\neq 0 }^{n}  a_n^m\frac{\sqrt{n(n+1)}}{2n+1}c_n^mP_n^m(\cos\theta)-\eta_1\sum_{m=-n}^{n}b_n^m\cdot\frac{n+1}{2n+1}\frac{1}{\sqrt{n(n+1)}}\frac{c_n^m}{2}
  \cdot\bigg((n+m)\notag\\
  &\quad \times (n-m+1)P_n^{m-1}(\cos\theta)-P_n^{m+1}(\cos\theta)\bigg) +\eta_1b_n^0 \frac{n+1}{2n+1}\frac{c_n^0}{\sqrt{n(n+1)}}
   P_n^{1}(\cos\theta),\label{eq:457pi1}
     \end{align}     
     where for the index $m=0$ in \eqref{eq:im bd1n} we use the property \eqref{eq:pnm neg}, and $c_n^m$, $m=0,1,\ldots,n$, are nonzero constants defined in \eqref{sphe harmonic}. Utilizing the orthogonality condition \eqref{ortho3}, from \eqref{eq:457pi1} we can deduce \eqref{c2}.

      Similarly,  comparing the coefficient of $r^{n-1}$ both sides of \eqref{eq:im bd1n}  associated with $ \boldsymbol{e}_{2}\left(\theta, \phi\right)$ for $\phi=0$  in \eqref{eq:im bd1n}, from   Lemma \ref{lem:coeff0}, we  obtain the following $n+1$ equations:
\begin{align}
  0=&\mathbf{i}k\sum_{m=-n \atop m\neq 0}^{n}a_n^m\cdot\frac{n+1}{2n+1}\frac{1}{\sqrt{n(n+1)}}c_n^m\bigg((n+m)(n-m+1)P_n^{m-1}(\cos\theta)-P_n^{m+1}(\cos\theta)\bigg)\notag\\
  &-\mathbf{i}ka_n^0\cdot\frac{n+1}{2n+1}\frac{1}{\sqrt{n(n+1)}}c_n^0P_n^{1}(\cos\theta) +\eta_1\sum_{m=-n}^{n}b_n^m\frac{\sqrt{n(n+1)}}{2n+1}c_n^mP_n^m(\cos\theta) ,\label{eq:460 anbn}
 \end{align}   
 where for the index $m=0$ in \eqref{eq:im bd1n} we use the property \eqref{eq:pnm neg}.    By virtue of \eqref{eq:460 anbn},      utilizing the orthogonality condition \eqref{ortho3},  we can obtain \eqref{m2}.
   \end{proof}

\begin{lemma}\label{lem:55}
Under the same setup to Lemma~\ref{lem:54} and assuming that there exists  $n\in \mathbb N \backslash\{1\}$ such that \eqref{eq:lem54 cond} is fulfilled,  we have the following  recursive  linear equations:
\begin{align} 	
   0=&\mathbf{i}k\frac{\sqrt{n(n+1)}}{2n+1}c_n^0a_n^0-\frac{\eta_2(n+1) }{2(2n+1) }\frac{c_n^1(n+1)n}{\sqrt{n(n+1)}}(b_n^1e^{\bsi \alpha\cdot\pi} +b_n^{-1} e^{-\bsi \alpha\cdot\pi}), \notag
    \\
   0=&\mathbf{i}k\frac{\sqrt{n(n+1)}}{2n+1}c_n^1(a_n^1 e^{\bsi \alpha\cdot\pi}+a_n^{-1}e^{-\bsi \alpha\cdot\pi})-\frac{\eta_2(n+1)}{2(2n+1)}\frac{c_n^2 (n+2)(n-1)}{\sqrt{n(n+1)}}\notag
    \\
&\times (b_n^2 e^{\bsi 2\alpha\cdot\pi}+b_n^{-2}e^{-\bsi 2\alpha\cdot\pi})+ \frac{\eta_2 (n+1) }{2n+1}\frac{c_n^0}{\sqrt{n(n+1)}}b_n^0,  \notag
 \\
0=&\mathbf{i}k\frac{\sqrt{n(n+1)}}{2n+1}c_n^{m}(a_n^{m} e^{\bsi m\alpha\cdot\pi}+a_n^{-m}e^{-\bsi m\alpha\cdot\pi})- \frac{\eta_2(n+1)}{2(2n+1)}\frac{c_n^{m+1}}{\sqrt{n(n+1)}} (n+m+1)\notag
 \\
&\times (n-m) (b_n^{m+1}e^{\bsi (m+1)\alpha\cdot\pi}+b_n^{-(m+1)}e^{-\bsi (m+1)\alpha\cdot\pi})  +\frac{\eta_2(n+1)}{2(2n+1) }\frac{c_n^{m-1}}{\sqrt{n(n+1)}} \notag
\\
&\times  (b_n^{m-1} e^{\bsi (m-1)\alpha\cdot\pi} +b_n^{-(m-1)}e^{-\bsi (m-1)\alpha\cdot\pi} ), \quad  m=2,3,\ldots,n-1, \label{c22} 
\\
0=&\mathbf{i}k\frac{\sqrt{n(n+1)}}{2n+1}c_n^n(a_n^n e^{\bsi n\alpha\cdot\pi}+a_n^{-n}e^{-\bsi n \alpha\cdot\pi} )+\frac{\eta_2(n+1)}{2(2n+1) }\frac{c_n^{n-1}}{\sqrt{n(n+1)}} \notag
\\
&\times  (b_n^{n-1}e^{\bsi (n-1)\alpha \pi}+b_n^{-(n-1)}e^{-\bsi (n-1)\alpha \pi}),\notag
\end{align}
 and
 \begin{align}
0=&\mathbf{i}k\frac{n+1}{2(2n+1)}\frac{c_n^1}{\sqrt{n(n+1)}}(n+1)n(a_n^1 e^{\bsi\alpha \cdot \pi}+a_n^{-1} e^{-\bsi \alpha \cdot  \pi})+\eta_2 \frac{c_n^0\sqrt{n(n+1)}}{2n+1}b_n^0,\notag\\
0=&\mathbf{i}k\frac{n+1}{2(2n+1) }\frac{c_n^2}{\sqrt{n(n+1)}}(n+2)(n-1)(a_n^2 e^{2 \bsi\alpha \cdot \pi}+a_n^{-2} e^{-2 \bsi\alpha \cdot \pi} )-\mathbf{i}k \frac{n+1}{2n+1}\notag\\
&\times \frac{c_n^0 }{\sqrt{n(n+1)}}a_n^0 +\eta_2\frac{c_n^1 \sqrt{n(n+1)}}{2n+1}(b_n^1 e^{\bsi\alpha \cdot \pi} +b_n^{-1} e^{-\bsi\alpha \cdot \pi}), \label{m22} \\
0=&\mathbf{i}k\frac{n+1}{2(2n+1)}\frac{c_n^m (n+m)(n-m+1)}{\sqrt{n(n+1)}} (a_n^m e^{ \bsi m \alpha \cdot \pi} +a_n^{-m} e^{-\bsi m\alpha \cdot \pi} )-\mathbf{i}k\cdot\frac{n+1}{2(2n+1)}\notag\\
&\times  \frac{c_n^{m-2}}{\sqrt{n(n+1)}} (a_n^{m-2} e^{\bsi (m-2)\alpha \cdot \pi} +a_n^{-(m-2)} e^{-\bsi (m-2)\alpha \cdot \pi} )+\eta_2\frac{c_n^{m-1}\sqrt{n(n+1)}}{2n+1}\notag \\
&\times (b_n^{m-1} e^{\bsi (m-1)\alpha \cdot \pi} +b_n^{-(m-1)}e^{-\bsi (m-1)\alpha \cdot \pi} ),\quad m=3,4,\ldots,n,\notag \\
0=&-\mathbf{i}k\frac{n+1}{2(2n+1)}\frac{c_n^{n-1}}{\sqrt{n(n+1)}} (a_n^{n-1} e^{\bsi (n-1)\alpha \cdot \pi} +a_n^{-(n-1)} e^{-\bsi (n-1) \alpha \cdot \pi} )+\eta_2\frac{c_n^n \sqrt{n(n+1)}}{2n+1}\notag\\
&\times (b_n^n e^{\bsi n \alpha \cdot \pi}+b_n^{-n} e^{-\bsi n \alpha \cdot \pi} )=0.\notag
\end{align} 
 \end{lemma}

\begin{proof}  Since the generalized impedance condition \eqref{eq:imp2} associated with $\boldsymbol{ \eta}_2$  is imposed on $\widetilde {\Pi}_2$,  substituting \eqref{eq:lem54 cond}  into  \eqref{ss2}, by virtue of \eqref{uu},  we derive that
	  \begin{equation}\label{eq:im bd2n}
      \begin{split}
\mathbf{0}=
    &\sum_{l=n}^{\infty}\sum_{m=-l}^{l} \frac{e^{\bsi m\alpha\cdot\pi}}{\sqrt{l(l+1)}}\Bigg\{\bigg(
   \mathbf{i}ka_l^ml(l+1)p_l(kr)c_l^mP_l^m  +\boldsymbol{ \eta}_2 a_l^mj_l(kr) c_l^m \frac{\sgn(m)}{2} \big[P_{l-1}^{|m|+1}(\cos\theta)\\
   &+(l+|m|-1)(l+|m|)P_{l-1}^{|m|-1}(\cos\theta)\big]-\boldsymbol{ \eta}_2 b_l^mq_l(kr)\frac{c_l^m}{2}\big[(l+|m|)(l-|m|+1)\\
   &\times P_l^{|m|-1}(\cos\theta)-P_l^{|m|+1}(\cos\theta)\big]\bigg)
   \boldsymbol{e_1}(\theta,\phi_0) +\bigg(\mathbf{i}kb_l^mj_l(kr)  c_l^m  \frac{\sgn(m)}{2} \big[P_{l-1}^{|m|+1}(\cos\theta)\\
   &+(l+|m|-1)(l+|m|)P_{l-1}^{|m|-1}(\cos\theta)\big]+\mathbf{i}ka_l^m
    q_l(kr) \frac{c_l^m}{2}\big[(l+|m|)(l-|m|+1)\\
    &\times P_l^{|m|-1}(\cos\theta)-P_l^{|m|+1}(\cos\theta)\big]+\boldsymbol{ \eta}_2 b_l^ml(l+1)p_l(kr)c_l^mP_l^m\bigg)
   \boldsymbol{e_2}(\theta,\phi_0)\Bigg\},
      \end{split}
      \end{equation}
      where $\boldsymbol{e_1}(\theta, 0)$, $\boldsymbol{e_2}(\theta,0)$   $\boldsymbol{e_1}(\theta,\phi_0)$ and $\boldsymbol{e_2}(\theta,\phi_0)$ are defined in \eqref{eq:e1e2}.

Recall that $\boldsymbol{\eta}_\ell$, $\ell=1,2$, have the expansion \eqref{eq:eta1 ex} and \eqref{eq:eta2 ex} respectively, where the coefficients of $r^0$ in \eqref{eq:eta1 ex} and \eqref{eq:eta2 ex}   are non zero number $\eta_1$ and $\eta_2$.         Comparing the coefficient of $r^{n-1}$ both sides of  \eqref{eq:im bd2n} associated with $ \boldsymbol{e}_{1}\left(\theta, \phi\right)$ and  $ \boldsymbol{e}_{2}\left(\theta, \phi\right)$ for $\phi=\phi_0$ respectively,   utilizing the orthogonality condition \eqref{ortho3}, we  can derive \eqref{c22} and \eqref{m22}.
\end{proof}


The next theorem characterises the vanishing order of $\mathbf{E}$ to \eqref{eq:eig} at ${\mathbf 0} \in  {\mathcal E}(\widetilde \Pi_1, \widetilde \Pi_2,\bsl)$ with $\boldsymbol \eta_j\in\mathcal{A}(\bsl)$.

\begin{theorem}\label{th:two imp}
Let $\mathbf{E}$ be a a solution to \eqref{eq:eig}, whose radial wave expansion in $B_{\rho_0}(\mathbf{0}) $ is given by \eqref{mix pi21}.  Consider a generalized impedance  edge-corner ${\mathcal E}(\widetilde{ \Pi}_1, \widetilde{ \Pi}_2,\bsl) \Subset \Omega$ with $\angle(\Pi_{1},\Pi_2)=\phi_0=\alpha \pi$, where $\alpha\in(0,2)$ and $\alpha \neq 1$. Suppose that the generalized impedance parameters $ \boldsymbol{ \eta}_j$ on $\widetilde{ \Pi}_j $, $j=1, 2$, are given by \eqref{eq:eta1 ex} and \eqref{eq:eta2 ex} respectively. Then it holds that $\mathbf{E}$ vanishes up to the order $N$ at $\mathbf{0}$:
		\begin{equation}\notag 
		\mathrm{Vani}(\mathbf{E}; \mathbf{0})\geq \left\{\begin{split}
			&1,\quad \mbox{if }\alpha \neq \frac{1}{2} ,  \\ 
			& N \in\mathbb{N}\backslash\{1\} ,\, \mbox{if }    \alpha \neq \frac{q  }{p}, \, p=1,\ldots, N, \mbox{ and for a fixed } p,   q=1,\ldots, 2p-1. 
					\end{split}\right.
					\end{equation}
\end{theorem}

\begin{proof}
	We prove this theorem by induction. Assume that
	\begin{equation}\label{eq:alpha 12}
		\alpha \neq \frac 1 2 \mbox{ and } 		\alpha \neq \frac 3 2 .
	\end{equation}
Since the generalized impedance condition \eqref{eq:imp2} associated with $\boldsymbol{ \eta}_1$  is imposed on $\widetilde {\Pi}_1$, it yields \eqref{eq:im bd1n} when the summation index $n=1$.  Comparing the coefficient of $r^0$  associated with $\boldsymbol{e_2}(\theta,0)$ on both sides of \eqref{eq:im bd1n} for the case that the summation index $n=1$,  from Lemma \ref{base21}, we can obtain that
\begin{equation}\label{eq:543 a1}
0=\mathbf{i}k\frac{4c_1^1}{6\sqrt{2}}(a_1^1+a_1^{-1})+\eta_1\frac{\sqrt{2}c_1^0}{3}b_1^0.
\end{equation}
Combine \eqref{eq:543 a1} with  \eqref{eq:lem51 a1} and \eqref{eq:lem51 a2} from Lemma \ref{base2},  we derive that
\begin{align}\label{eq:543 matrix}
{\mathcal A}_1  \begin{bmatrix}
	a_1^1+a_1^{-1} \\ a_1^1-a_1^{-1} \\ b_1^0
\end{bmatrix} &=\mathbf{0},\quad {\mathcal A}_1 =\left(\alpha_{ij}^1\right)_{i,j=1}^3,
\end{align}
where
\begin{equation}\notag
\begin{split}
&\alpha_{11}=\frac{4 \mathbf{i}kc_1^1\sin^2\phi_0}{6\sqrt{2}},\quad\alpha_{12}=-\frac{4kc_1^1\sin\phi_0\cos\phi_0}{6\sqrt{2}},\quad\alpha_{13}=\frac{(-\eta_2\cos\phi_0-\eta_1)\sqrt{2}c_1^0}{3}\\
&\alpha_{21}=-\frac{4\mathbf{i}kc_1^1\sin\phi_0\cos\phi_0}{6\sqrt{2}},\quad\alpha_{22}=-\frac{4kc_1^1\sin^2\phi_0}{6\sqrt{2}},\quad
\alpha_{23}=-\frac{\eta_2\sqrt{2} c_1^0\sin\phi_0}{3}\\
&\alpha_{31}=\frac{4\mathbf{i}kc_1^1}{6\sqrt{2}},\quad\alpha_{32}=0,\quad \alpha_{33}=\frac{\sqrt{2}c_1^0 \eta_1}{3}.
\end{split}
\end{equation}
By direct calculations, it yields that
\[
 \left|
    {\mathcal A}_1
    \right|
    =-\mathbf{i}k^2\eta_1\left(\frac{2}{3}\right)^3\frac{\sqrt{2}}{2}\big(c_1^1\big)^2c_1^0\sin^2(\alpha \pi).
\]
Hence under the assumption \eqref{eq:alpha 12} and $\eta_1\neq 0$, by virtue of \eqref{eq:543 matrix} and $k\in \mathbb R_+$, it can be derived that $a_1^{\pm 1}=b_1^0=0$. Recall that  \eqref{eq:lem51 1} is given by Lemma \ref{base2}. In view of \eqref{eq:alpha 12}, $\alpha \in (0,1)$, $k\in \mathbb R_+$ and $\eta_2\neq 0$, using the fact that
\[
 \left|{\mathcal B}_1
\right| =-k\eta_2^2\left(\frac{2}{3}\right)^3\frac{\sqrt{2}}{2}\big(c_1^1\big)^2c_1^0\sin^2 (\alpha \pi)  \cos^2(\alpha \pi )\neq 0,
\]
where ${\mathcal B}_1$ is defined in \eqref{eq:lem51 1}, we can obtain that  $b_1^{\pm 1}=a_1^0=0$. Therefore, from Lemma \ref{lem:vani}, we prove that $\mathrm{Vani}(\mathbf{E}; \mathbf{0}) \geq 1$ under conditions \eqref{eq:alpha 12} and $\eta_\ell \neq 0$, $\ell=1,2$.

	Suppose that $N=2$, from the assumption in this theorem  we know that \eqref{eq:alpha 12} still holds. Since $a_1^{\pm 1}=b_{1}^{\pm 1} =a_1^0=b_1^0=0$, from Lemmas \ref{lem:54} and \ref{lem:55} we know that \eqref{c2}, \eqref{m2}, \eqref{c22} and \eqref{m22} for $n=2$ hold. Therefore we have
 \begin{equation}\label{12equations1}
   \left\{
   \begin{split}
   0=&\frac{\sqrt{6}c_2^0\mathbf{i}k}{5}a_2^0-\frac{18c_2^1\eta_1 }{10\sqrt{6} }(b_2^1+b_2^{-1}),
   \\
   0=&\frac{\sqrt{6}c_2^1\mathbf{i}k}{5}(a_2^1+a_2^{-1})-\frac{12c_2^2\eta_1}{10\sqrt{6}}(b_2^2+b_2^{-2})+ \frac{3c_2^0\eta_1  }{5\sqrt{6}}b_2^0,
    \\
   0=&\mathbf{i}k\frac{\sqrt{6}}{5}c_2^2(a_2^2+a_2^{-2})+\frac{3c_2^{1}\eta_1}{10\sqrt{6} }(b_2^{1}+b_2^{-1}),
\end{split}
\right. 	
\end{equation}
and
\begin{equation}\label{12equations2}
\left\{
\begin{split}
0=&\frac{18c_2^1\mathbf{i}k}{10\sqrt{6}}(a_2^1+a_2^{-1})+ \frac{\sqrt{6}c_2^0\eta_1}{5}b_2^0,\\
0=&\frac{12c_2^2\mathbf{i}k}{10\sqrt{6} }(a_2^2+a_2^{-2})- \frac{3c_2^0\mathbf{i}k}{5\sqrt{6}}a_2^0+\frac{\sqrt{6}c_2^1\eta_1}{5}(b_2^1+b_2^{-1}), \\
0=&-\frac{3c_2^{1}\mathbf{i}k}{10\sqrt{6}}(a_2^{1}+a_2^{-1})+\frac{\sqrt{6}c_2^2\eta_1}{5}(b_2^2+b_2^{-2})=0.
\end{split}
\right.
\end{equation}
Furthermore, it holds that
\begin{equation} \label{12equations3} 	
\left\{
\begin{split}
   0=&\frac{\sqrt{6}c_2^0\mathbf{i}k}{5}a_2^0-\frac{18c_2^1\eta_2 }{10\sqrt{6} }(b_2^1e^{\bsi \alpha\cdot\pi} +b_2^{-1} e^{-\bsi \alpha\cdot\pi}), 
    \\
   0=&\frac{\sqrt{6}c_2^1\mathbf{i}k}{5}(a_2^1 e^{\bsi \alpha\cdot\pi}+a_2^{-1}e^{-\bsi \alpha\cdot\pi})-\frac{12c_2^2\eta_2}{10\sqrt{6}} (b_2^2 e^{\bsi 2\alpha\cdot\pi}+b_2^{-2}e^{-\bsi 2\alpha\cdot\pi})+ \frac{3c_2^0\eta_2  }{5\sqrt{6}}b_2^0,  
\\
0=&\frac{\sqrt{6}c_2^2\mathbf{i}k}{5}(a_2^2 e^{\bsi 2\alpha\cdot\pi}+a_2^{-2}e^{-\bsi 2 \alpha\cdot\pi} )+\frac{3c_2^{1}\eta_2}{10\sqrt{6} }  (b_2^{1}e^{\bsi \alpha \pi}+b_2^{-1}e^{-\bsi \alpha \pi}),
\end{split}
\right.
\end{equation}
 \begin{equation} \label{12equations4}
 \left\{
\begin{split}
0=&\frac{18c_2^1\mathbf{i}k}{10\sqrt{6}}(a_2^1 e^{\bsi\alpha \cdot \pi}+a_2^{-1} e^{-\bsi \alpha \cdot  \pi})+ \frac{\sqrt{6}c_2^0\eta_2}{5}b_2^0,\\
0=&\frac{12c_2^2\mathbf{i}k}{10\sqrt{6} }(a_2^2 e^{2 \bsi\alpha \cdot \pi}+a_2^{-2} e^{-2 \bsi\alpha \cdot \pi} )- \frac{3c_2^0\mathbf{i}k}{5\sqrt{6}}a_2^0 +\frac{ \sqrt{6}c_2^1\eta_2}{5}(b_2^1 e^{\bsi\alpha \cdot \pi} +b_2^{-1} e^{-\bsi\alpha \cdot \pi}), \\
0=&-\frac{3c_2^{1}\mathbf{i}k}{10\sqrt{6}} (a_2^{1} e^{\bsi \alpha \cdot \pi} +a_2^{-1} e^{-\bsi  \alpha \cdot \pi} )+\frac{\sqrt{6}c_2^2\eta_2 }{5}(b_2^2 e^{\bsi 2 \alpha \cdot \pi}+b_2^{-2} e^{-\bsi 2 \alpha \cdot \pi} )=0.
\end{split}
\right.
\end{equation}

From Lemma \ref{lem:imp pi12},  \eqref{eq:52a} and \eqref{eq:52b} for $n=2$ can be written as
 \begin{align}\label{eq:545 n=2}
 \begin{split}
 0&=\frac{18c_2^1\sin^2\phi_0\mathbf{i}k}{10\sqrt{6}}(a_2^1+a_2^{-1})-\frac{18c_2^1\sin\phi_0\cos\phi_0k}{10\sqrt{6}}(a_2^1-a_2^{-1})+\frac{\sqrt{6}c_2^0(-\eta_2\cos\phi_0-\eta_1)}{5} b_2^0, \\
0&=-\frac{18c_2^1\sin\phi_0\cos\phi_0\mathbf{i}k}{10\sqrt{6}}
(a_2^1+a_2^{-1})-\frac{18c_2^1\sin^2\phi_0k}{10\sqrt{6}}(a_2^1-a_2^{-1})-\frac{\eta_2\sqrt{6} c_2^0\sin\phi_0}{5}b_2^0.
 \end{split}
\end{align}
Combing the first equation of \eqref{12equations2} with \eqref{eq:545 n=2}, we have
	\begin{equation}\label{eq:546 linear}
		{\mathcal A}_2 \begin{bmatrix}
	a_2^1+a_2^{-1} \\ a_2^1-a_2^{-1} \\ b_2^0
\end{bmatrix} =\mathbf{0},\quad {\mathcal A}_2 =\left(\alpha_{ij}^2\right)_{i,j=1}^3,
	\end{equation}
	where
	\begin{equation}\notag
\begin{split}
&\alpha_{11}^2=\frac{18c_2^1\sin^2\phi_0\mathbf{i}k}{10\sqrt{6}},\quad\alpha_{12}^2=-\frac{18c_2^1\sin\phi_0\cos\phi_0k}{10\sqrt{6}},\quad\alpha_{13}^2=\frac{\sqrt{6}c_2^0(-\eta_2\cos\phi_0-\eta_1)}{5}\\
&\alpha_{21}^2=-\frac{18c_2^1\sin\phi_0\cos\phi_0\mathbf{i}k}{10\sqrt{6}},\quad\alpha_{22}^2=-\frac{18c_2^1\sin^2\phi_0k}{10\sqrt{6}},\quad
\alpha_{23}^2=-\frac{\eta_2\sqrt{6} c_2^0\sin\phi_0}{5}\\
&\alpha_{31}^2=\frac{18c_2^1\mathbf{i}k}{10\sqrt{6}},\quad\alpha_{32}^2=0,\quad \alpha_{33}^2=\frac{\sqrt{6}c_2^0\eta_1}{5}.
\end{split}
\end{equation}
It can be computed directly that
\begin{equation}
\begin{split}
 \left|{\mathcal A}_2
\right| &=-\mathbf{i}k^2\eta_1\left(\frac{3}{5}\right)^3\frac{2\sqrt{6}}{2}\big(c_2^1\big)^2c_2^0\sin^2(\alpha \pi).
\end{split}
\end{equation}
Since  $\alpha \in (0,2)$, $\alpha\neq 1$, $\eta_1\neq 0$ and $k\in \mathbb R_+$, in view of \eqref{eq:546 linear}, we prove that $a_2^{\pm1}=b_2^0=0$.

Recall that $\mathbf{E}$  has the radial wave expansion \eqref{mix pi21}  at  $\mathbf{0}$.   	Due to that $\eta_1 \neq 0$, under the assumption  \eqref{eq:alpha 12}, by virtue of \eqref{eq:lem41 zero} in Lemma \ref{base2}, we have
	 \begin{equation}\label{eq:a1b10 new}
	a_1^0=b_1^{\pm1}=0.
\end{equation}

  By mathematical induction, if $\alpha \neq \frac{q  }{p}$, where $p=1,\ldots, n-1$ and for a fixed $p$, $ q=1,2,\ldots, 2p-1$, then
      $$
      \mathrm{Vani}(\mathbf{E}; \mathbf{0})\geq n-1.
      $$
      From Lemma \ref{lem:vani}, we know that
      \begin{equation}\label{eq:alm=0}
      	a_l^m=b_l^m=0,\quad m\in [l]_0, \quad l=1,2,\ldots,n.
      \end{equation}
 Therefore we know that \eqref{c2}, \eqref{m2}, \eqref{c22} and \eqref{m22} hold from Lemmas \ref{lem:54} and \ref{lem:55}.  In the following under the assumption
 \begin{equation}\label{eq:455 cond}
 \eta_\ell \neq 0 \mbox{ for } \ell=1, 2 \mbox{ and }	\alpha \neq \frac{q  }{p}, \, p=1,\ldots, n,
 \end{equation}
 where  for a fixed $p$, $ q=1,2,\ldots, 2p-1$,  we shall show that
 \begin{equation}\label{eq:4567 anm}
 	a_n^m=b_n^m=0,\quad \forall m \in [n]_0
 \end{equation}
 by utilizing the recursive equations of \eqref{c2}, \eqref{m2}, \eqref{c22} and \eqref{m22}. Indeed, combing the first equation of \eqref{m2} with \eqref{eq:lem51 a1} and \eqref{eq:lem51 a2}, we have
 \begin{equation}\label{eq:556 matrix}
 	{\mathcal A}_n  \begin{bmatrix}
	a_n^1+a_n^{-1} \\ a_n^1-a_n^{-1} \\ b_n^0
\end{bmatrix} =\mathbf{0},\quad {\mathcal A}_n =\left(\alpha_{ij}^n\right)_{i,j=1}^3,
 \end{equation}
 where
 \begin{equation}\notag
\begin{split}
&\alpha^n_{11}=\frac{\mathbf{i}kn(n+1)^2c_n^1\sin^2\phi_0}{2(2n+1)\sqrt{n(n+1)}},\quad \alpha^n_{12}=-\frac{kn(n+1)^2c_n^1\sin\phi_0\cos\phi_0}{2(2n+1)\sqrt{n(n+1)}},\\
&\alpha^n_{13}=\frac{(-\eta_2\cos\phi_0-\eta_1)\sqrt{n(n+1)}c_n^0}{2n+1},\quad \alpha^n_{21}=-\frac{\mathbf{i}kn(n+1)^2c_n^1\sin\phi_0\cos\phi_0}{2(2n+1)\sqrt{n(n+1)}},\\
&\alpha^n_{22}=-\frac{kn(n+1)^2c_n^1\sin^2\phi_0}{2(2n+1)\sqrt{n(n+1)}},\quad
\alpha^n_{23}=-\frac{\eta_2\sqrt{n(n+1)} c_n^0\sin\phi_0}{2n+1}, \\
&\alpha^n_{31}=\mathbf{i}k\frac{n(n+1)^2c_n^1}{2(2n+1)\sqrt{n(n+1)}},\quad\alpha^n_{32}=0,\quad \alpha^n_{33}=\eta_1\frac{\sqrt{n(n+1)}c_n^0}{2n+1}.
\end{split}
\end{equation}
 It can be derived that
 \begin{equation}\label{eq:557 deter}
  \left|{\mathcal A}_n
\right| =-\mathbf{i}k^2\eta_1\left(\frac{n+1}{2n+1}\right)^3\frac{n\sqrt{n(n+1)}}{2}\big(c_n^1\big)^2c_n^0\sin^2(\alpha \pi) .
 \end{equation}
Since $\alpha\in (0,2)$, $\alpha \neq 1$,  $\alpha \neq \frac{1}{2}$, $\alpha \neq \frac{3}{2}$ and $\eta_\ell \neq 0$, $\ell=1,2$, by virtue of \eqref{eq:556 matrix}, \eqref{eq:557 deter} and Lemma \ref{base2}, we have
\begin{equation}\label{eq:462 an1bn1+}
	a_n^{\pm1}=a_n^{0}=b_n^{\pm1}=b_n^{0}=0.
\end{equation}

Substituting \eqref{eq:462 an1bn1+} into the second equation of \eqref{c2}, \eqref{m2}, \eqref{c22} and \eqref{m22}, since $k\in \mathbb R_+$,  $\eta_\ell \neq 0$ for $\ell=1,2$ and $c_n^2\neq 0$, we obtain that
   \begin{equation}\notag
   	\left\{\begin{array}{l}
   		a_n^2 +a_n^{-2}=0,\\
   		a_n^2 e^{2 \bsi\alpha \cdot \pi}+a_n^{-2} e^{-2 \bsi\alpha \cdot \pi}=0,
   	\end{array}  \right. \quad
   	   	\left\{\begin{array}{l}
   		b_n^2 +b_n^{-2}=0,\\
   		b_n^2 e^{2 \bsi\alpha \cdot \pi}+b_n^{-2} e^{-2 \bsi\alpha \cdot \pi}=0,
   	\end{array}  \right.
   \end{equation}
 which can be shown to prove that $a_n^{\pm 2} = b_n^{\pm 2}=0$, since
\[
 \left|\begin{array}{cc}
    1 &  1 \\
    e^{\bsi 2\alpha\cdot\pi} & e^{-\bsi 2\alpha\cdot\pi}
    \end{array}\right|
    =-2\bsi\sin (2\alpha\pi)\neq0,
\]
under \eqref{eq:455 cond}.  Substituting
$$
a_n^{\pm 1} = b_n^{\pm 1}=a_n^{\pm 2} = b_n^{\pm 2}=0
$$
into the third equation of \eqref{c2}, \eqref{m2}, \eqref{c22} and \eqref{m22}, since $k\in \mathbb R_+$,   $\eta_\ell \neq 0$ for $\ell=1,2$ and $c_n^3 \neq 0$, we get that
   \begin{equation}\notag
   	\left\{\begin{array}{l}
   		a_n^3 +a_n^{-3}=0,\\
   		a_n^3 e^{3 \bsi\alpha \cdot \pi}+a_n^{-3} e^{-3 \bsi\alpha \cdot \pi}=0,
   	\end{array}  \right. \quad
   	   	\left\{\begin{array}{l}
   		b_n^3 +b_n^{-3}=0,\\
   		b_n^3 e^{3 \bsi\alpha \cdot \pi}+b_n^{-3} e^{-3 \bsi\alpha \cdot \pi}=0,
   	\end{array}  \right.
   \end{equation}
 which can be shown to prove that $a_n^{\pm 3} = b_n^{\pm 3}=0$, since
\[
 \left|\begin{array}{cc}
    1 &  1 \\
    e^{\bsi 3\alpha\cdot\pi} & e^{-\bsi 3\alpha\cdot\pi}
    \end{array}\right|
    =-2\bsi\sin (3\alpha\pi)\neq0,
\]
under \eqref{eq:455 cond}. Repeating the above procedures step by step, utilizing the recursive property   of \eqref{c2}, \eqref{m2}, \eqref{c22} and \eqref{m22}, we can prove \eqref{eq:4567 anm}. Generally, assume that we have proved that
\begin{equation}\notag
	a_n^{\pm m}=b_n^{\pm m}=0 \mbox{ for }m=0,1,\ldots, \ell.
\end{equation}
Substituting $a_n^{\pm (\ell-1)}=b_n^{\pm (\ell-2)}=0$ into the $\ell$-th equation of \eqref{c2} and \eqref{c22}, we can obtain that
\begin{equation}\label{eq:560}
	\left\{\begin{array}{l}
   		b_n^\ell +b_n^{-\ell}=0,\\
   		b_n^\ell e^{\bsi \ell \alpha \cdot \pi}+b_n^{-\ell} e^{-\bsi\ell \alpha \cdot \pi}=0,
   	\end{array}  \right.
\end{equation}
under the assumption $\eta_1\neq 0$ and $\eta_2\neq 0$. Substituting $a_n^{\pm (\ell-2)}=b_n^{\pm (\ell-1)}=0$ into the $\ell$-th equation of \eqref{m2} and \eqref{m22}, we can get that
\begin{equation}\label{eq:561}
	\left\{\begin{array}{l}
   		a_n^\ell  +a_n^{-\ell }=0,\\
   		a_n^\ell  e^{ \bsi \ell \alpha \cdot \pi}+a_n^{-2} e^{- \bsi \ell \alpha \cdot \pi}=0,
   	\end{array}  \right.
\end{equation}
Hence  from \eqref{eq:560} and \eqref{eq:561},  under \eqref{eq:455 cond} it yields that $a_n^{\pm \ell}=b_n^{\pm \ell}=0 $.

Therefore,  due to \eqref{eq:4567 anm}, by virtue of  Lemma \ref{lem:vani}, 
 we prove that
  $$
      \mathrm{Vani}(\mathbf{E}; \mathbf{0})\geq n,
      $$
      which completes the proof of this theorem.
\end{proof}

\section{Vanishing orders for an edge-corner ${\mathcal E} ( \widetilde{ \Pi}_1,  \widetilde{ \Pi}_2,\bsl)$ with $\boldsymbol{\eta}_j\in \mathcal{A}(\bsl)$ or $\boldsymbol{\eta}_j=0, \infty$}\label{sec:6}

In this section, we investigate the vanishing order of the solution $\mathbf{E}$ to \eqref{eq:eig} at an edge-corner point $\mathbf 0 \in {\mathcal E}(\widetilde \Pi_1,  \widetilde  \Pi_2,\bsl)$, where the generalized impedance edge-corner ${\mathcal E}(\widetilde \Pi_1,  \widetilde  \Pi_2,\bsl) \Subset \Omega$ with $\angle(\Pi_{1},\Pi_2)=\phi_0=\alpha \pi$, $\alpha\in(0,2)$ and $\alpha \neq 1$. The generalized impedance condition \eqref{eq:imp2} on $\widetilde \Pi_j $, $j=1,2$, are different. Namely, the  associated generalized impedance parameter of the generalized impedance  edge-corner ${\mathcal E}(\Pi_1, \Pi_2,\bsl)$ in Theorem \ref{thm:pec pmc} are $\boldsymbol{\eta}_1\equiv \infty$ and $\boldsymbol{\eta}_2\equiv 0$, where  we utilize Lemma \ref{lem:31} to reveal the vanishing order of $\mathbf{E}$  at $\mathbf 0$. On the other hand, in Theorems \ref{thm:imp pec} and \ref{thm:imp pmc}, we consider the case that $\boldsymbol{\eta}_2\in {\mathcal A}(\bsl)$ has the expansion \eqref{eq:eta2 ex} whereas the  associated generalized impedance parameter $\boldsymbol{\eta}_1$ could be either $\infty$ or $0$. The reflection principle \cite{Liu3,Liu09} are adopted to transform the corresponding generalized impedance edge-corner to be generalized impedance edge-corner intersected by two plane cells with the generalized impedance condition \eqref{eq:imp2} and two  associated generalized impedance parameters belonging to $\mathcal{A}(\bsl)$.

\begin{lemma}\label{lem:31}
Let $\mathbf{E}$ be a a solution to \eqref{eq:eig}, whose radial wave expansion in $B_{\rho_0}(\mathbf{0}) $ is given by \eqref{mix pi21}.  Consider a generalized impedance  edge-corner ${\mathcal E}(\widetilde{ \Pi}_1, \widetilde{ \Pi}_2,\bsl) \Subset \Omega$ with $\angle(\Pi_{1},\Pi_2)=\phi_0=\alpha \pi$, where $\alpha\in(0,2)$ and $\alpha \neq 1$. Suppose that the generalize impedance parameters $ \boldsymbol{ \eta}_1$ on $\widetilde{ \Pi}_1 $ satisfies (ii) in \eqref{eq:imp1} and  $ \boldsymbol{ \eta}_2$ on $\widetilde{ \Pi}_2 $ satisfies (i) in \eqref{eq:imp1}  respectively. It holds that
 \begin{subequations}
     \begin{align}
   & b_1^1+b_1^{-1}=0, \quad b_1^0=0, \label{eq:b1b10 pec}\\
    &a_1^1-a_1^{-1}=0,\label{eq:d1 pec}\\
    & b_{2}^m+b_{2}^{-m}=0,\quad m=1, 2,\mbox{ and } b_{2}^0=0, \label{eq:39 b1b10 pec}
       \end{align}
    \end{subequations}
    and
     \begin{subequations}
 \begin{align}
    & 
    a_1^1e^{\bsi\alpha\cdot\pi}+a_1^{-1}e^{-\bsi\alpha\cdot\pi}=0, \quad a_1^0=0, \label{eq:pmc 49} \\
   & 
    b_1^1e^{\bsi\alpha\cdot\pi}-b_1^{-1}e^{-\bsi\alpha\cdot\pi}=0,  \label{d1}\\
    &   a_{2}^me^{\bsi m\alpha\cdot\pi}+a_{2}^{-m}e^{-\bsi m\alpha\cdot\pi}=0,\quad m=1, 2,\mbox{ and } a_{2}^0=0. \label{eq:41c}
    \end{align}
    \end{subequations}
    Assume that there exits a $n\in \mathbb N$ such that
\begin{equation}\label{eq:lem31 cond}
		a_l^m=b_l^m=0,\quad l=1,2,\ldots,n-1,\quad m\in [l]_0,
	\end{equation}
then we have
\begin{subequations}
\begin{align}
&b_n^m+b_n^{-m}=0, \quad m=1,\ldots, n, \mbox{ and } b_n^0=0,\label{eq:lem31 33} \\
& a_n^me^{\bsi m\alpha\cdot\pi}+a_n^{-m}e^{-\bsi m\alpha\cdot\pi}=0, \quad m=1,\ldots, n, \mbox{ and } a_n^0=0, \label{eq:lem41 412}
\end{align}
\end{subequations}
and
\begin{subequations}
\begin{align}
& \sum_{m=1}^{n}mc_n^m(a_n^m-a_n^{-m})\frac{P_n^m(\cos\theta)}{\sin \theta}+\sum_{m=-(n+1)}^{n+1}\frac{c_{n+1}^m(n+2)}{2n+3}b_{n+1}^m\frac{\partial Y_{n+1}^m}{\partial\theta}\Big|_{\phi=0}=0, \label{eq:lem31 34} \\
&\sum_{m=1}^{n}mc_n^m(b_n^me^{\bsi m\alpha\cdot\pi}-b_n^{-m}e^{-\bsi m\alpha\cdot\pi})\frac{P_n^m(\cos\theta)}{\sin \theta}+\sum_{m=-(n+1)}^{n+1}\frac{c_{n+1}^m(n+2)}{2n+3}a_{n+1}^m\frac{\partial Y_{n+1}^m}{\partial\theta}\Big|_{\phi=\phi_0}=0, \label{eq:lem41 43} 
\end{align}
\end{subequations}
where $c_n^m$ are nonzero constants defined in \eqref{sphe harmonic} for m $=0,1,\ldots,n$. Furthermore, we have
\begin{subequations}
\begin{align}
&b_{n+1}^m+b_{n+1}^{-m}=0,\quad m=1,\ldots, n+1,\mbox{ and } b_{n+1}^0=0, \label{eq:bm1} \\
&  a_{n+1}^me^{\bsi m\alpha\cdot\pi}+a_{n+1}^{-m}e^{-\bsi m\alpha\cdot\pi}=0,\quad m=1,\ldots, n+1,\mbox{ and } a_{n+1}^0=0,\label{eq:b7}
\end{align}
\end{subequations}
where $c_{n+1}^m$ are nonzero constants defined in \eqref{sphe harmonic} for m$ =0,1,\ldots,n+1 $.
\end{lemma}


\begin{proof}
We first derive \eqref{eq:b1b10 pec},  \eqref{eq:d1 pec} and \eqref{eq:39 b1b10 pec}.  Since the generalized impedance condition \eqref{eq:imp2} associated with $\boldsymbol{ \eta}_1$  is imposed on $\widetilde {\Pi}_1$ where $\boldsymbol{\eta}_1\equiv \infty $, using \eqref{mix pi2}, we have
\begin{equation}\label{eq:38 pec}
\begin{split}
&		\mathbf{0}=  \sum_{l=1}^{\infty}\sum_{m=-l}^{l}-\frac{1}{\sqrt{l(l+1)}}\Bigg\{b_l^ml(l+1)p_l(kr)Y_l^m\Big|_{\phi=0}
  \boldsymbol{e_1}(\theta,0)\\
    & \hspace{2cm} +\bigg(
    a_l^mj_l\big(kr\big)\frac{m}{\sin\theta}Y_l^m\Big|_{\phi=0}+b_l^m\cdot q_l(kr)
   \frac{\partial{Y_l^m}}{\partial\theta}\Big|_{\phi=0}
    \bigg)\boldsymbol{e_2}(\theta,0)\Bigg\},
  \end{split}
	\end{equation}
	where $ \boldsymbol{e}_{1}\left(\theta, 0\right)$ and $   \boldsymbol{e}_{2}\left(\theta, 0\right)$ are defined in \eqref{eq:e1e2}.  From Remark \ref{i2}, the lowest order of \eqref{eq:38 pec} with respect to the power of $r$ is $r^0$, which is contributed by $p_1(kr)$ and $q_1(kr)$. Similarly, the second lowest order of \eqref{eq:38 pec} with respect to the power of $r$ is $r^1$, which is contributed by $j_1(kr)$, $p_2(kr)$ and $q_2(kr)$. Comparing the coefficient of $r^0$ and $r^1$ associated with  $ \boldsymbol{e}_{1}\left(\theta, 0 \right)$   on both sides of \eqref{eq:38 pec}, utilizing the orthogonality property  \eqref{ortho3}, we can obtain 	\eqref{eq:b1b10 pec} and \eqref{eq:39 b1b10 pec}. 
Substituting \eqref{eq:39 b1b10 pec} into  \eqref{eq:38 pec}, comparing the coefficient of  $r^1$ in the resulting equation \eqref{eq:38 pec}   associate with $ \boldsymbol{e}_{2}\left(\theta, 0\right)$, using Lemma \ref{lem:coeff0}, we deduce that
    \begin{equation}\label{ii1}
    \begin{split}
    &(a_1^1c_1^1-a_1^{-1}c_1^{-1})P_1^1(\cos\theta)=0,
     \end{split}
     \end{equation}
    where  $c_1^{\pm 1}$ are nonzero constants  defined in \eqref{sphe harmonic}.
In view of \eqref{ii1}, from \eqref{ortho3} and $c_1^1=c_1^{-1}\neq 0$, it yields that \eqref{eq:d1 pec}.

Since the generalized impedance condition \eqref{eq:imp2} associated with $\boldsymbol{ \eta}_2$  is imposed on $\widetilde {\Pi}_2$ where $\boldsymbol{\eta}_2\equiv 0 $,  by virtue of  \eqref{gg} it yields that
\begin{equation}\label{eq:331 pec a1b1}
\begin{split}
      \mathbf{0} ={\mathbf{i}k}&\sum_{l=1}^{\infty}\sum_{m=-l}^{l}\frac{1}{\sqrt{l(l+1)}}\Bigg\{a_l^ml(l+1)p_l(kr)Y_l^m \Big|_{\phi=\phi_0}
\cdot\boldsymbol{e_1}(\theta,\phi_0)\\
    &+\bigg( -b_l^mj_l(kr)\cdot\frac{m}{\sin\theta}Y_l^m \Big|_{\phi=\phi_0}+a_l^m
    q_l(kr)\cdot\frac{\partial{Y_l^m}}{\partial\theta}\Big|_{\phi=\phi_0}\bigg)
    \cdot\boldsymbol{e_2}(\theta,\phi_0)\Bigg\},
     \end{split}
	\end{equation}
	where $ \boldsymbol{e}_{1}\left(\theta, \phi_0\right)$ and $   \boldsymbol{e}_{2}\left(\theta, \phi_0\right)$ are defined in \eqref{eq:e1e2}. From Remark \ref{i2}, the lowest order of \eqref{eq:331 pec a1b1} with respect to the power of $r$ is $r^0$, which is contributed by $p_1(kr)$ and $q_1(kr)$. Similarly, the second lowest order of \eqref{eq:331 pec a1b1}  with respect to the power of $r$ is $r^1$, which is contributed by $j_1(kr)$, $p_2(kr)$ and $q_2(kr)$. Comparing the coefficient of $r^0$ and $r^1$ associated with    $ \boldsymbol{e}_{1}\left(\theta, \phi_0\right)$  on both sides of \eqref{eq:331 pec a1b1}, utilizing the orthogonality property  \eqref{ortho3}, we can obtain 	\eqref{eq:pmc 49} and \eqref{eq:41c}.

	Comparing the coefficient of $r^1$ in \eqref{eq:331 pec a1b1}  associate with $ \boldsymbol{e}_{2}\left(\theta, \phi_0\right)$, using Lemma \ref{lem:coeff0}, we deduce that
    \begin{equation}\label{iii1}
    \begin{split}
    &(b_1^1c_1^1e^{\bsi\alpha\cdot\pi}-b_1^{-1}c_1^{-1}e^{-\bsi\alpha\cdot\pi})P_1^1(\cos\theta)=0,
     \end{split}
     \end{equation}
    where  $c_1^{\pm 1}$ are nonzero constants  defined in \eqref{sphe harmonic}.
In view of \eqref{iii1}, from \eqref{ortho3} and $c_1^1=c_1^{-1}\neq 0$, it yields that \eqref{d1}.

Now we are in the position to prove \eqref{eq:lem31 33}, \eqref{eq:lem31 34} and \eqref{eq:bm1} under the assumption \eqref{eq:lem31 cond}.  Since the generalized impedance condition \eqref{eq:imp2} associated with $\boldsymbol{ \eta}_1$  is imposed on $\widetilde {\Pi}_1$ where $\boldsymbol{\eta}_1\equiv \infty $, substituting  \eqref{eq:lem31 cond} into \eqref{mix pi2} it yields that
\begin{equation}\label{eq:33 pec}
\begin{split}
&		\mathbf{0}=  \sum_{l=n}^{\infty}\sum_{m=-l}^{l}-\frac{1}{\sqrt{l(l+1)}}\Bigg\{b_l^ml(l+1)p_l(kr)Y_l^m\Big|_{\phi=0}
  \boldsymbol{e_1}(\theta,0)\\
    & \hspace{2cm} +\bigg(
    a_l^mj_l\big(kr\big)\frac{m}{\sin\theta}Y_l^m\Big|_{\phi=0}+b_l^m\cdot q_l(kr)
   \frac{\partial{Y_l^m}}{\partial\theta}\Big|_{\phi=0}
    \bigg)\boldsymbol{e_2}(\theta,0)\Bigg\},
  \end{split}
	\end{equation}
	where $ \boldsymbol{e}_{1}\left(\theta, 0\right)$ and $   \boldsymbol{e}_{2}\left(\theta, 0\right)$ are defined in \eqref{eq:e1e2}.
	
	The lowest order term in \eqref{eq:33 pec} with respect to the power of $r$ is $r^{n-1}$, which is contributed by $p_n(kr)$ and $q_n(kr)$ from Remark \ref{i2}. Since $ \boldsymbol{e}_{1}\left(\theta, \phi\right)$ and $   \boldsymbol{e}_{2}\left(\theta, \phi\right)$ are linearly independent for any $\theta$ and $\phi$,  where $   \boldsymbol{e}_{i}\left(\theta, \phi\right)$ are defined in \eqref{eq:e1e21},  from   Lemma \ref{lem:coeff0}, comparing the coefficient of $r^{n-1}$ both sides of \eqref{eq:33 pec} associated with $ \boldsymbol{e}_{1}\left(\theta, 0\right)$,	we can obtain
	\begin{equation}\notag
	\label{eq:lem31 33 n}
\begin{split}
&\sum_{m=0}^{n}c_n^m(b_n^m+b_n^{-m})P_n^m(\cos\theta)=0.
\end{split}
\end{equation}
 Utilizing the orthogonality property  \eqref{ortho3},  since  $c_n^m\neq 0$ for $m\in [n]_0$,  \eqref{eq:lem31 33} holds.

	From Remark \ref{i2} we know that the second lowest  order term in in \eqref{eq:33 pec} with respect to the power of $r$ is $r^{n}$, which is related to $j_n(kr)$, $p_{n+1}(kr)$ and $q_{n+1}(kr)$.  Since $ \boldsymbol{e}_{1}\left(\theta, \phi\right)$ and $   \boldsymbol{e}_{2}\left(\theta, \phi\right)$ are linear independently for any $\theta$ and $\phi$, comparing the coefficient of $r^{n}$ both sides of \eqref{eq:33 pec} associated with $ \boldsymbol{e}_{1}\left(\theta, 0\right)$, we can obtain
\begin{equation}\notag
\label{eq:bm1 n}
\begin{split}
&\sum_{m=0}^{n+1}c_{n+1}^m(b_{n+1}^m+b_{n+1}^{-m})P_{n+1}^m(\cos\theta)=0.
\end{split}
\end{equation}
 Using the orthogonality property  \eqref{ortho3}, together with the fact that $c_{n+1}^m\neq 0$ for $m\in [n+1]_0$, we see that \eqref{eq:bm1} holds.

Similarly, in view of Remark \ref{i2},  comparing the coefficient of $r^{n}$ both sides of \eqref{eq:33 pec} associated with $ \boldsymbol{e}_{2}\left(\theta, 0\right)$, we know that \eqref{eq:lem31 34} hold. 

We proceed to derive \eqref{eq:lem41 412}, \eqref{eq:lem41 43} and \eqref{eq:b7} under the assumption \eqref{eq:lem31 cond}.    Since the generalized impedance condition \eqref{eq:imp2} associated with $\boldsymbol{ \eta}_2$  is imposed on $\widetilde {\Pi}_2$ where $\boldsymbol{\eta}_2 \equiv 0$, substituting  \eqref{eq:lem31 cond} into \eqref{gg} it yields that
\begin{equation}\label{eq:331 pec}
\begin{split}
	     \mathbf{0}={\mathbf{i}k}&\sum_{l=n}^{\infty}\sum_{m=-l}^{l}\frac{1}{\sqrt{l(l+1)}}\Bigg\{a_l^ml(l+1)p_l(kr)Y_l^m \Big|_{\phi=\phi_0}
\cdot\boldsymbol{e_1}(\theta,\phi_0)\\
    &+\bigg( -b_l^mj_l(kr)\cdot\frac{m}{\sin\theta}Y_l^m \Big|_{\phi=\phi_0}+a_l^m
    q_l(kr)\cdot\frac{\partial{Y_l^m}}{\partial\theta}\Big|_{\phi=\phi_0}\bigg)
    \cdot\boldsymbol{e_2}(\theta,\phi_0)\Bigg\}.
     \end{split}
	\end{equation}
	where $ \boldsymbol{e}_{1}\left(\theta, \phi_0\right)$ and $   \boldsymbol{e}_{2}\left(\theta, \phi_0\right)$ are defined in \eqref{eq:e1e2}.
	
	The lowest order term in \eqref{eq:331 pec} with respect to the power of $r$ is $r^{n-1}$, which is contributed by $p_n(kr)$ and $q_n(kr)$ from Remark \ref{i2}. Since $ \boldsymbol{e}_{1}\left(\theta, \phi\right)$ and $   \boldsymbol{e}_{2}\left(\theta, \phi\right)$ are linear independently for any $\theta$ and $\phi$,  where $   \boldsymbol{e}_{i}\left(\theta, \phi\right)$ are defined in \eqref{eq:e1e2},  from   Lemma \ref{lem:coeff0}, comparing the coefficient of $r^{n-1}$ both sides of \eqref{eq:331 pec} associated with $ \boldsymbol{e}_{1}\left(\theta, \phi_0\right)$, 
	we can obtain
	\begin{equation}\label{eq:lem31 33 n}
\begin{split}
&\sum_{m=0}^{n}c_n^m(a_n^me^{\bsi m\alpha\cdot\pi}+a_n^{-m}e^{-\bsi m\alpha\cdot\pi})P_n^m(\cos\theta)=0.
\end{split}
\end{equation}
 Using the orthogonality property  \eqref{ortho3}, together with the fact that $c_n^m\neq 0$ for $m\in [n]_0$,  we can obtain \eqref{eq:lem41 412}.

	From Remark \ref{i2} we know that the second lowest order term in  \eqref{eq:331 pec} with respect to the power of $r$ is $r^{n}$, which is related to $j_n(kr)$, $p_{n+1}(kr)$ and $q_{n+1}(kr)$.  Since $ \boldsymbol{e}_{1}\left(\theta, \phi\right)$ and $   \boldsymbol{e}_{2}\left(\theta, \phi\right)$ are linearly independent for any $\theta$ and $\phi$, comparing the coefficient of $r^{n}$ both sides of \eqref{eq:331 pec} associated with $ \boldsymbol{e}_{1}\left(\theta, \phi_0\right)$, we can get
\begin{equation}\notag
\label{eq:bm1 n}
\begin{split}
&\sum_{m=0}^{n+1}c_{n+1}^m(a_{n+1}^me^{\bsi m\alpha\cdot\pi}+a_{n+1}^{-m}e^{-\bsi m\alpha\cdot\pi})P_{n+1}^m(\cos\theta)=0.
\end{split}
\end{equation}
 Utilizing the orthogonality property  \eqref{ortho3},  since  $c_{n+1}^m\neq 0$ for $m\in [n+1]_0$,  we derive \eqref{eq:b7}.

Similarly, in view of Remark \ref{i2},  comparing the coefficient of $r^{n}$ both sides of \eqref{eq:331 pec} associated with $ \boldsymbol{e}_{2}\left(\theta, \phi_0\right)$, we know that \eqref{eq:lem41 43} holds.

The proof is complete.  	
	
\end{proof}

\begin{theorem}\label{thm:pec pmc}
Under the same setup in Lemma~\ref{lem:31}, we have that
		\begin{align}\notag
			&\mathrm{Vani}(\mathbf{E}; \mathbf{0})\geq N,\quad \mbox{if }    \alpha \neq \frac{q }{2p}, \, p=1,\ldots, N,  
				\end{align}
				where  $N\in\mathbb{N}$  and for a fixed $p$, $ q=1,2,\ldots, 4p-1.$

\end{theorem}

\begin{proof}
We prove this theorem by induction.  Assume that
\begin{equation}\label{eq:61 cond}
	\alpha \neq \frac{1}{2} \mbox{ and } \alpha \neq \frac{3}{2},
\end{equation}
we shall prove that $\mathrm{Vani}(\mathbf{E}; \mathbf{0})\geq 1$. Since the generalized impedance condition \eqref{eq:imp2} associated with $\boldsymbol{ \eta}_1$  is imposed on $\widetilde {\Pi}_1$ where $\boldsymbol{\eta}_1\equiv \infty $, from Lemma \ref{lem:31} we know that \eqref{eq:b1b10 pec} and \eqref{eq:d1 pec} hold. Similarly, since the generalized impedance condition \eqref{eq:imp2} associated with $\boldsymbol{ \eta}_2$  is imposed on $\widetilde {\Pi}_2$ where $\boldsymbol{\eta}_2 \equiv 0$, from Lemma \ref{lem:31} it yields that \eqref{eq:pmc 49} and \eqref{d1}.

Combing \eqref{eq:b1b10 pec},  \eqref{eq:d1 pec}, \eqref{eq:pmc 49} and \eqref{d1}, it yields that
\begin{equation}\label{eq:66 eqn}
   	\left\{\begin{array}{l}
   		a_1^1 -a_1^{-1}=0,\\
   		a_1^1 e^{\bsi\alpha \cdot \pi}+a_1^{-1} e^{- \bsi\alpha \cdot \pi}=0,
   	\end{array}  \right.
   	\left\{\begin{array}{l}
   		b_1^1 +b_1^{-1}=0,\\
   		b_1^1 e^{\bsi\alpha \cdot \pi}-b_1^{-1} e^{- \bsi\alpha \cdot \pi}=0.
   	\end{array}  \right.
   \end{equation}
Under \eqref{eq:61 cond} we have
$$
 \left|\begin{array}{cc}
    1 &  -1 \\
    e^{\bsi \alpha\cdot\pi} & e^{-\bsi \alpha\cdot\pi}
    \end{array}\right|
    =2\cos  (\alpha\cdot\pi) \neq0,
$$
which implies that $a_1^{\pm1}=b_1^{\pm1}=0$ from \eqref{eq:66 eqn}. Since $a_1^0=b_1^0=0$, from Lemma \ref{lem:vani}, we prove $\mathrm{Vani}(\mathbf{E}; \mathbf{0})\geq 1$ under the assumption \eqref{eq:61 cond}.

   Assume that
   \begin{equation}\label{eq:67 assump}
   	\alpha \neq \frac{1}{2},\quad 	\alpha \neq \frac{1}{4}, \quad \alpha \neq \frac{3}{4},\quad \alpha \neq \frac{5}{4}, \quad \alpha \neq \frac{3}{2}   \mbox{ and } 	\alpha \neq \frac{7}{4},
   \end{equation}
   which implies that $\mathrm{Vani}(\mathbf{E}; \mathbf{0})\geq 1$. Hence we have
\begin{equation}\label{eq:67 a1b1=0}
	a_1^{\pm1}=b_1^{\pm1}=a_1^0=b_1^0=0
\end{equation}
from Lemma \ref{lem:vani}. Since the generalized impedance condition \eqref{eq:imp2} associated with $\boldsymbol{ \eta}_1$  is imposed on $\widetilde {\Pi}_1$ where $\boldsymbol{\eta}_1\equiv \infty $, from Lemma \ref{lem:31} we have
\begin{equation}\label{eq:68 a2b2=0}
b_2^0=0,\quad b_2^1+b_2^{-1}=0,\quad  b_2^2+b_2^{-2}=0
\end{equation}
by \eqref{eq:lem31 33} and
\begin{equation}\label{eq:610 b3}
	 b_{3}^m+b_{3}^{-m}=0, \quad m=1,2, 3, \mbox{ and } b_{3}^0=0
\end{equation}
by \eqref{eq:bm1}.  Substituting \eqref{eq:610 b3} into the first equation of \eqref{eq:lem31 34}, it yields that
\begin{equation}\label{eq:610 a21}
	a_2^1 -a_2^{-1}=0,\quad a_2^2 -a_2^{-2}=0
\end{equation}
by noting $c_{3}^{m}=c_{3}^{-m} \neq 0$ for $ m=1,2, 3$, where $c_{3}^{m}$ and $c_{3}^{-m}$ are defined in \eqref{sphe harmonic}.

Similarly, in view of \eqref{eq:67 a1b1=0}, using Lemma \ref{lem:31},  we obtain that
    \begin{equation}\label{eq:611 a2 pmc}
a_2^0=0,\quad a_2^1 e^{\bsi\alpha \cdot \pi}+a_2^{-1} e^{-\bsi\alpha \cdot \pi}=0,\quad a_2^2 e^{\bsi2\alpha \cdot \pi}+a_2^{-2} e^{-\bsi2\alpha \cdot \pi}=0
\end{equation}	
by \eqref{eq:lem41 412} and
 \begin{equation}\label{eq:613 a3}
 	 a_{3}^m e^{\bsi m \alpha
 	 \pi }+a_{3}^{-m}e^{-\bsi m \alpha
 	 \pi }=0, \quad m=1,2, 3, \mbox{ and } a_{3}^0=0
 \end{equation}
 by \eqref{eq:b7}.  Substituting \eqref{eq:613 a3} into the second equation of \eqref{eq:lem41 43}, it yields that
\begin{equation}\label{eq:612 a2}
	b_2^1 e^{\bsi\alpha \cdot \pi}-b_2^{-1} e^{- \bsi\alpha \cdot \pi}=0,\quad b_2^2 e^{\bsi 2\alpha \cdot \pi}-b_2^{-2} e^{-\bsi2\alpha \cdot \pi}=0
	\end{equation}
	by 	using the fact that $c_{3}^{m}=c_{3}^{-m} \neq 0$ for $ m=1,2$ and the definition of $Y_{3}^m(\theta,\phi)$, where $c_{3}^{m}$ and $c_{3}^{-m}$ are defined in \eqref{sphe harmonic}.
Combing \eqref{eq:68 a2b2=0}, \eqref{eq:610 a21} and \eqref{eq:611 a2 pmc} with \eqref{eq:612 a2}, we obtain that
\begin{equation}\label{eq:613 four eqn}
\begin{split}
   &	\left\{\begin{array}{l}
   		a_2^1 -a_2^{-1}=0,\\
   		a_2^1 e^{\bsi\alpha \cdot \pi}+a_2^{-1} e^{- \bsi\alpha \cdot \pi}=0,
   	\end{array}  \right.
   	\left\{\begin{array}{l}
   		b_2^1 +b_2^{-1}=0,\\
   		b_2^1 e^{\bsi\alpha \cdot \pi}-b_2^{-1} e^{- \bsi\alpha \cdot \pi}=0,
   	\end{array}  \right. \\
   	& \left\{\begin{array}{l}
   		a_2^2 -a_2^{-2}=0,\\
   		a_2^2 e^{\bsi 2\alpha \cdot \pi}+a_2^{-2} e^{- \bsi 2\alpha \cdot \pi}=0,
   	\end{array}  \right.
   	\left\{\begin{array}{l}
   		b_2^2 +b_2^{-2}=0,\\
   		b_2^1 e^{\bsi 2\alpha \cdot \pi}-b_2^{-2} e^{- \bsi 2\alpha \cdot \pi}=0.
   	\end{array}  \right.
   	\end{split}
   \end{equation}
Under the assumption \eqref{eq:67 assump} it is easy to see that
$$
 \left|\begin{array}{cc}
    1 &  -1 \\
    e^{\bsi \alpha\cdot\pi} & e^{-\bsi \alpha\cdot\pi}
    \end{array}\right|
    =2\cos  (\alpha\cdot\pi) \neq0,\quad  \left|\begin{array}{cc}
    1 &  -1 \\
    e^{2\bsi \alpha\cdot\pi} & e^{-2\bsi \alpha\cdot\pi}
    \end{array}\right|
    =2\cos  (2\alpha\cdot\pi) \neq0
$$
which imply that $a_2^{\pm1}=b_2^{\pm1}=a_2^{\pm 2}=b_2^{\pm 2}=0$  in view of \eqref{eq:613 four eqn}. Due to \eqref{eq:68 a2b2=0} and \eqref{eq:611 a2 pmc},  we have $a_2^0=b_2^0=0$, hence from Lemma \ref{lem:vani} we prove $\mathrm{Vani}(\mathbf{E}; \mathbf{0})\geq 2$ under the assumption \eqref{eq:67 assump}.

By the induction, we assume that
 \begin{equation}\label{eq:614 assump}
   \alpha \neq \frac{2q+1  }{2p}, \, p=1,\ldots, n, \mbox{ for a fixed }p, \ \ q=0,1,\ldots, 2p-1. 
   \end{equation}
Therefore, we know that $\mathrm{Vani}(\mathbf{E}; \mathbf{0})\geq n-1$ from the induction under the assumption \eqref{eq:614 assump}, which implies that
\begin{equation}\label{eq:615 an-1=0}
	a_{l}^m=0 \mbox{ for } l=1,\ldots, n-1 \mbox{ and } m \in [l]_0.
\end{equation}
from Lemma \ref{lem:vani}.

Due to \eqref{eq:615 an-1=0} and the fact that the generalized impedance condition \eqref{eq:imp2} associated with $\boldsymbol{ \eta}_1$  is imposed on $\widetilde {\Pi}_1$ where $\boldsymbol{\eta}_1 \equiv \infty$, from Lemma \ref{lem:31}, we have
\begin{equation}\label{eq:616 bn1}
 b_n^m+b_n^{-m}=0, \quad m=1,\ldots, n, \mbox{ and } b_n^0=0
 \end{equation}
by  \eqref{eq:lem31 33} and
\begin{equation}\label{eq:616 bn+1}
	 b_{n+1}^m+b_{n+1}^{-m}=0, \quad m=1,\ldots, n+1, \mbox{ and } b_{n+1}^0=0
\end{equation}
by \eqref{eq:bm1}. Substituting \eqref{eq:616 bn+1} into the first equation of \eqref{eq:lem31 34}, it yields that
\begin{equation}\label{eq:617 an1}
	a_2^m -a_2^{-m}=0,\quad m=1,\ldots, n,
\end{equation}
by noting $c_{n+1}^{m}=c_{n+1}^{-m} \neq 0$ for $ m=1,\ldots, n$, where $c_{n+1}^{m}$ and $c_{n+1}^{-m}$ are defined in \eqref{sphe harmonic}.

Similarly, due to \eqref{eq:615 an-1=0} and the fact that the generalized impedance condition \eqref{eq:imp2} associated with $\boldsymbol{ \eta}_2$  is imposed on $\widetilde {\Pi}_2$ where $\boldsymbol{\eta}_2 \equiv 0$, using Lemma \ref{lem:31}, we get that
\begin{equation}\label{eq:618 an1}
 a_n^m e^{\bsi m \alpha
 	 \pi }+a_n^{-m} e^{-\bsi m \alpha
 	 \pi }=0, \quad m=1,\ldots, n, \mbox{ and } a_n^0=0
 \end{equation}
 by \eqref{eq:lem41 412} and
 \begin{equation}\label{eq:620 an+1}
 	 a_{n+1}^m e^{\bsi m \alpha
 	 \pi }+a_{n+1}^{-m}e^{-\bsi m \alpha
 	 \pi }=0, \quad m=1,\ldots, n+1, \mbox{ and } a_{n+1}^0=0
 \end{equation}
 by \eqref{eq:b7}.
 Substituting \eqref{eq:620 an+1} into the second equation of \eqref{eq:lem41 43}, it yields that
 \begin{equation}\label{eq:621 a2}
	b_n^m e^{\bsi m \alpha \cdot \pi}-b_n^{-m} e^{- \bsi m\alpha \cdot \pi}=0,\quad m=1,\ldots, n
		\end{equation}
by 	using the fact that $c_{n+1}^{m}=c_{n+1}^{-m} \neq 0$ for $ m=1,\ldots, n$ and the definition of $Y_{n+1}^m(\theta,\phi)$, where $c_{n+1}^{m}$ and $c_{n+1}^{-m}$ are defined in \eqref{sphe harmonic}.

	Combing \eqref{eq:616 bn1}, \eqref{eq:617 an1} and \eqref{eq:618 an1} with \eqref{eq:621 a2}, we obtain that
\begin{equation}\label{eq:622 two eqn}
   	\left\{\begin{array}{l}
   		a_n^m -a_n^{-m}=0,\\
   		a_n^m e^{\bsi m\alpha \cdot \pi}+a_n^{-m} e^{- \bsi m\alpha \cdot \pi}=0,
   	\end{array}  \right.
   	\left\{\begin{array}{l}
   		b_n^m +b_n^{-m}=0,\\
   		b_n^m e^{\bsi m\alpha \cdot \pi}-b_n^{-m} e^{- \bsi m\alpha \cdot \pi}=0,
   	\end{array}  \right. \quad m=1,\ldots, n.
   \end{equation}
Under the assumption \eqref{eq:614 assump} it is not difficult  to see that
$$
 \left|\begin{array}{cc}
    1 &  -1 \\
    e^{\bsi m \alpha\cdot\pi} & e^{-\bsi m\alpha\cdot\pi}
    \end{array}\right|
    =2\cos  (m\alpha\cdot\pi) \neq0,
$$
which imply that $a_n^{\pm m}=b_n^{\pm m}=0$  in view of \eqref{eq:622 two eqn}. Due to \eqref{eq:616 bn1} and \eqref{eq:618 an1},  we have $a_n^0=b_n^0=0$, hence from Lemma \ref{lem:vani} we prove $\mathrm{Vani}(\mathbf{E}; \mathbf{0})\geq n$ under the assumption \eqref{eq:67 assump}.

   The proof is complete.
\end{proof}

In the following two theorems, we consider the generalized impedance edge-corner  ${\mathcal E}(\widetilde \Pi_1,  \widetilde \Pi_2,\bsl)$ where  the generalize impedance parameter $ \boldsymbol{ \eta}_2$ on $\widetilde{ \Pi}_2 $ satisfies (iii) in \eqref{eq:imp1} and has the expansion \eqref{eq:eta2 ex}, whereas the generalize impedance parameter $ \boldsymbol{ \eta}_1$ on $\widetilde{ \Pi}_1 $ satisfies either (i) or  (ii) in \eqref{eq:imp1}. In the sequel, we shall make use of the reflection principles for the Maxwell equations from \cite{Liu3,Liu09}.

For any two-dimensional plane $\Pi \in \mathbb R^3 $, let $\nu_\Pi$ and ${\mathcal R}_\Pi$ be respectively the unit normal to $\Pi$ and the reflection with respect to $\Pi$ in $\mathbb R^3$. 


\begin{lemma}\cite[Theorems 2.1 and 2.2]{Liu09}\label{lem:reflection}
Consider a generalized impedance edge-corner ${\mathcal E}(\widetilde \Pi_1,  \widetilde  \Pi_2,\bsl) \Subset \Omega$ with $\angle(\Pi_{1},\Pi_2)=\phi_0=\alpha \pi$, where $\alpha\in(0,1)$. Assume that  the generalize impedance parameter $ \boldsymbol{ \eta}_2$ on $\widetilde{ \Pi}_2 $ satisfies (iii) in \eqref{eq:imp1} and has the expansion  \eqref{eq:eta2 ex} while  the generalize impedance parameter $ \boldsymbol{ \eta}_1$ on $\widetilde{ \Pi}_1 $ satisfies (ii) in \eqref{eq:imp1} (i.e.,  $ \boldsymbol{ \eta}_1 \equiv \infty $). Recall that $\Pi_1$ be a plane containing $\widetilde \Pi_1$. Let  ${\widetilde {\Pi}_2}'={\mathcal R}_{ \Pi_1}(\widetilde \Pi_2)$. Then 
	\begin{equation}\label{eq:lem62}
\nu_{\widetilde \Pi_2' } \wedge (\nabla\wedge \mathbf{E})+ \widetilde{ \boldsymbol{ \eta}}_2 (\nu_{\widetilde \Pi_2' }\wedge\mathbf{E})\wedge\nu_{\widetilde \Pi_2' } =\mathbf 0 	\mbox{ on } \widetilde \Pi_2',
	\end{equation}
	where $\nu_{\Pi_2' } $ is the unit normal to $\Pi_2'$ directed to the interior of  ${\mathcal E}(\widetilde \Pi_1, \widetilde \Pi_2',\bsl) $ and $\widetilde{ \boldsymbol{ \eta}}_2 (\mathbf x) =\boldsymbol{ \eta}_2({\mathcal R}_{\Pi_1} (\mathbf x) )$ for $\mathbf{x} \in \widetilde{\Pi}_2' $.
	
	 Similarly, consider a generalized impedance edge-corner ${\mathcal E}(\widetilde \Pi_1,  \widetilde  \Pi_2,\bsl) \Subset \Omega$ with $\angle(\Pi_{1},\Pi_2)=\phi_0=\alpha \pi$, where $\alpha\in(0,1)$. Assume that  the generalize impedance parameter $ \boldsymbol{ \eta}_2$ on $\widetilde{ \Pi}_2 $ satisfies (iii) in \eqref{eq:imp1} and has the expansion  \eqref{eq:eta2 ex} while  the generalize impedance parameter $ \boldsymbol{ \eta}_1$ on $\widetilde{ \Pi}_1 $ satisfies (i) in \eqref{eq:imp1} (i.e.,  $ \boldsymbol{ \eta}_1 \equiv 0 $). Recall that $\Pi_1$ be a plane containing $\widetilde \Pi_1$. Let  ${\widetilde {\Pi}_2}'={\mathcal R}_{ \Pi_1}(\widetilde \Pi_2)$. Then 
	\begin{equation} \notag
\nu_{\widetilde \Pi_2' } \wedge (\nabla\wedge \mathbf{E})+ \widetilde{ \boldsymbol{ \eta}}_2 (\nu_{\widetilde \Pi_2' }\wedge\mathbf{E})\wedge\nu_{\widetilde \Pi_2' } =\mathbf 0 	\mbox{ on } \widetilde \Pi_2',
	\end{equation}
	where $\nu_{\Pi_2' } $ is the unit normal to $\Pi_2'$ directed to the interior of  ${\mathcal E}(\widetilde \Pi_1, \widetilde \Pi_2',\bsl) $ and $\widetilde{ \boldsymbol{ \eta}}_2 (\mathbf x) =\boldsymbol{ \eta}_2({\mathcal R}_{\Pi_1} (\mathbf x) )$ for $\mathbf{x} \in \widetilde{\Pi}_2' $.
	\end{lemma}

\begin{theorem}\label{thm:imp pec}
Let $\mathbf{E}$ be a solution to \eqref{eq:eig}. Consider a generalized impedance edge-corner ${\mathcal E}(\widetilde \Pi_1,  \widetilde  \Pi_2,\bsl) \Subset \Omega$ with $\angle(\Pi_{1},\Pi_2)=\phi_0=\alpha \pi$, where $\alpha\in(0,2)$ and $\alpha \neq 1$. Assume that  the generalize impedance parameter $ \boldsymbol{ \eta}_2$ on $\widetilde{ \Pi}_2 $ satisfies (iii) in \eqref{eq:imp1} and has the expansion  \eqref{eq:eta2 ex} while  the generalize impedance parameter $ \boldsymbol{ \eta}_1$ on $\widetilde{ \Pi}_1 $ satisfies (ii) in \eqref{eq:imp1} (i.e.,  $ \boldsymbol{ \eta}_1 \equiv \infty $).   Then
		\begin{align}\label{eq:Th44 cond}
			&\mathrm{Vani}(\mathbf{E}; \mathbf{0})\geq N,\quad \mbox{if }    \alpha \neq \frac{q  }{2p}, \, p=1,\ldots, N,  
				\end{align}
				where  $N\in\mathbb{N}$  and for a fixed $p$, $ q=1,2,\ldots, 4p-1.$
\end{theorem}

\begin{proof}
Let 	$\widetilde \Pi_2'={\mathcal R}_{{\Pi}_1}(\widetilde \Pi_2)$, where ${\Pi}_1$ is  a plane containing $\Pi_1$. With the help of Lemma \ref{lem:reflection}, we know that  $\mathbf{E}$ satisfies the generalized impedance boundary condition \eqref{eq:lem62} on $\widetilde \Pi_2'$. Since $\mathbf{x} \in \widetilde{\Pi}_2 $, we have the spherical coordinate of $\mathbf{x}=(r, \theta, \phi_0)$, where $0\leq r \leq h$, $\theta \in [-\pi,\pi]$ and $\phi_0=\alpha \pi $.  It is clear that the spherical coordinate of  ${\mathcal R}_{\Pi_1} (\mathbf x) )$, where $\mathbf{x} \in \widetilde{\Pi}_2$, is given by
\begin{equation}
	\notag
(r, \theta, \phi_1), \mbox{ where } \phi_1=2-\alpha \in (0,2). 	
\end{equation}  
Recall that  $ \boldsymbol{ \eta}_2$  has the expansion  \eqref{eq:eta2 ex}. Although $\mathbf{x} \in \widetilde{\Pi}_2$ and ${\mathcal R}_{\Pi_1} (\mathbf x) ) \in  \widetilde{\Pi}_2'$ have different azimuthal angles but they have the same  polar angle $\theta$, hence from Definition \ref{def:class1}, we know that $\widetilde{\boldsymbol{\eta}}_2 $ has the same expansion \eqref{eq:eta2 ex} as ${\boldsymbol{\eta}}_2 $.

Furthermore, the dihedral angle between $\widetilde \Pi_2$ and $\widetilde \Pi_2'$ satisfies 
$$
\angle(\widetilde \Pi_2,\widetilde \Pi_2' )= \begin{cases}
2\alpha \pi \in (0,\pi],  \hspace{1.3cm} \alpha \in (0,1/2),\\[5pt]
	2(1-\alpha)\pi\in (0,\pi],\quad \alpha \in [1/2,1),\\[5pt]	
	2(\alpha-1)\pi\in (0,\pi],\quad \alpha \in (1,3/2),\\[5pt]	
	2(2-\alpha)\pi\in (0,\pi],\quad \alpha \in [3/2,2),
	\end{cases}
$$

We divide our remaining proof into four  separate cases. Recall that that the Maxwell system \eqref{eq:eig} is invariant under rigid motions. Without loss of generality, we assume that the  generalized impedance edge-corner ${\mathcal E}(\widetilde \Pi_2,  \widetilde  \Pi_2',\bsl) \Subset \Omega$ are placed as shown in Figure \ref{fig:coordinate1}. 
		
			\medskip
	
	\noindent {\bf Case 1.}~If $\alpha \in (0,1/2)$, then $2\alpha \in (0,1)$. By virtue of Theorem \ref{th:two imp}, if 
\begin{equation}\label{eq:623 alpha}
2\alpha \neq \frac{q}{p}, \quad p=1,\ldots,N, \mbox{ for a fixed } p, \, q=1,\ldots, p-1,  
\end{equation}
we have $\mathrm{Vani}(\mathbf{E}; \mathbf{0})\geq N$. It is easy to see that \eqref{eq:623 alpha} is equivalent to 
\begin{equation}\label{eq:624}
	\alpha \neq \frac{q}{2p}, \quad p=1,\ldots,N, \mbox{ for a fixed } p, \, q=1,\ldots, p-1.
\end{equation}

		\noindent {\bf Case 2.}~If $\alpha \in [1/2,1)$, then $2(1-\alpha) \in (0,1]$. By virtue of Theorem \ref{th:two imp},  if 
\begin{equation}\label{eq:624 alpha}
2(1-\alpha) \neq \frac{q}{p}, \quad p=1,\ldots,N, \mbox{ for a fixed } p, \, q=1,\ldots, p,  
\end{equation}
we have $\mathrm{Vani}(\mathbf{E}; \mathbf{0})\geq N$. It is easy to see that \eqref{eq:624 alpha}  is equivalent to 
\begin{equation}\label{eq:626}
	\alpha \neq \frac{q}{2p}, \quad p=1,\ldots,N, \mbox{ for a fixed } p, \, q=p,\ldots, 2p-1.
\end{equation}

\noindent {\bf Case 3.}~If $\alpha \in (1,3/2)$, then $2(\alpha-1) \in (0,1)$. By virtue of Theorem \ref{th:two imp},  if 
\begin{equation}\label{eq:624 alpha1}
2(\alpha-1) \neq \frac{q}{p}, \quad p=1,\ldots,N, \mbox{ for a fixed } p, \, q=1,\ldots, p-1,  
\end{equation}
we have $\mathrm{Vani}(\mathbf{E}; \mathbf{0})\geq N$. It is easy to see that \eqref{eq:624 alpha1}  is equivalent to 
\begin{equation}\label{eq:626 1}
	\alpha \neq \frac{q}{2p}, \quad p=1,\ldots,N, \mbox{ for a fixed } p, \, q=2p+1,\ldots, 3p-1.
\end{equation}

\noindent {\bf Case 4.}~If $\alpha \in [3/2,2)$, then $2(2-\alpha) \in (0,1]$. By virtue of Theorem \ref{th:two imp},  if 
\begin{equation}\label{eq:624 alpha2}
2(2-\alpha) \neq \frac{q}{p}, \quad p=1,\ldots,N, \mbox{ for a fixed } p, \, q=1,\ldots, p-1,  
\end{equation}
we have $\mathrm{Vani}(\mathbf{E}; \mathbf{0})\geq N$. It is easy to see that \eqref{eq:624 alpha2}  is equivalent to 
\begin{equation}\label{eq:626 2}
	\alpha \neq \frac{q}{2p}, \quad p=1,\ldots,N, \mbox{ for a fixed } p, \, q=3p,\ldots, 4p-1.
\end{equation}

\medskip

In view of \eqref{eq:624}, \eqref{eq:626}, \eqref{eq:626 1} and \eqref{eq:626 2}, we finish the proof of this theorem.    
\end{proof}

With the help of Lemma \ref{lem:reflection}, using the similar argument for proving Theorem \ref{thm:imp pec}, we can prove the following theorem, where the detailed proof is omitted. 
\begin{theorem}\label{thm:imp pmc}
Let $\mathbf{E}$ be a solution to \eqref{eq:eig}. Consider a generalized impedance edge-corner ${\mathcal E}(\widetilde \Pi_1,  \widetilde  \Pi_2,\bsl) \Subset \Omega$ with $\angle(\Pi_{1},\Pi_2)=\phi_0=\alpha \pi$, where $\alpha\in(0,1)$. Assume that  the generalize impedance parameter $ \boldsymbol{ \eta}_2$ on $\widetilde{ \Pi}_2 $ satisfies (iii) in \eqref{eq:imp1} and has the expansion  \eqref{eq:eta2 ex} while  the generalize impedance parameter $ \boldsymbol{ \eta}_1$ on $\widetilde{ \Pi}_1 $ satisfies (i) in \eqref{eq:imp1} (i.e.,  $ \boldsymbol{ \eta}_1 \equiv 0 $).  Then
		\begin{align} \notag
			&\mathrm{Vani}(\mathbf{E}; \mathbf{0})\geq N,\quad \mbox{if }    \alpha \neq \frac{q  }{2p}, \, p=1,\ldots, N,  
				\end{align}
				where  $N\in\mathbb{N}$  and for a fixed $p$, $ q=1,2,\ldots, 4p-1$.

\end{theorem}

\section{Irrational intersections and infinite vanishing orders}\label{sec4}

From the results derived in Sections \ref{sec:5} to \ref{sec:6}, one can identify that the vanishing order of the eigenfunction $\mathbf E$ at a generalized impedance edge-corner relies on the degree of the dihedral angle of the underlying corner. Next, we introduce the irrational and rational edge-corners, and then, based on the results in Sections  \ref{sec:5} to \ref{sec:6}, we show that the vanishing order of the eigenfunction at an irrational edge-corner is generically infinity and hence it vanishes identically in $\Omega$, namely strong uniqueness continuation principle holds in such a case.

\begin{definition}
	Let ${\mathcal E}(\widetilde \Pi_1, \widetilde \Pi_2,\bsl)$ be an edge-corner defined in Section \ref{sec:Intro} and the corresponding dihedral angle of $\widetilde \Pi_1$ and $\widetilde \Pi_2$ is denoted by $\phi_0=\alpha \pi $, $\alpha\in(0,2)$ and $\alpha \neq 1$. If $\alpha$ is an irrational number, then the edge-corner is called irrational. If $\alpha$ is a rational number of the form $q/p$ with $p, q\in\mathbb{N}$ being irreducible, the edge-corner is called rational and $p$ is referred to as its rational degree. 
\end{definition}

%

We readily have the following theorem from Theorems~\ref{th:two imp}, \ref{thm:pec pmc}, \ref{thm:imp pec} and \ref{thm:imp pmc}.

\begin{theorem}\label{ir-2nodal}
Let $\mathbf{E}$ be a solution to \eqref{eq:eig}. Consider an irrational generalized impedance edge-corner ${\mathcal E}(\widetilde \Pi_1,  \widetilde  \Pi_2,\bsl) \Subset \Omega$ with $\angle(\Pi_{1},\Pi_2)=\phi_0=\alpha \pi$, where $\alpha\in(0,2)$ and $\alpha \neq 1$. Under the same requirement on $\boldsymbol{\eta}_j$, $j=1,2$, to either one from Theorems~\ref{th:two imp}, \ref{thm:pec pmc}, \ref{thm:imp pec} and \ref{thm:imp pmc}, 
	it holds that
	\begin{equation*}\label{result1}
	\mathrm{Vani}({\mathbf E}; {\mathbf 0})=+\infty,\quad {\mathbf 0}\in \bsl.
	\end{equation*}
\end{theorem}

%
%
%
%
%
%

    \section{Applications to inverse electromagnetic scattering problems}\label{sec5}



In this section, we consider two applications of the UCP results established in the previous sections to the inverse electromagnetic scattering problems. In what follows, we first present the mathematical formulation of the inverse problem of determining an impenetrable obstacle from its associated electromagnetic far-field measurement. It is a prototypical model problem for many real applications including radar/sonar, non-destructive testing and medical imaging.

\subsection{Unique identifiability results for inverse obstacle scattering problems}

Let $\Omega\subset\mathbb{R}^3$ be a bounded Lipschitz domain such that $\mathbb{R}^3\backslash\bar{\Omega}$ is connected, and the
incident electric and magnetic fields be of the form
\begin{equation}\label{eq:p1n}
\begin{aligned}
\mathbf{E}^{{i}}\big(\mathbf{x}\big) &:=\mathbf{p} {e}^{\mathbf{i} k \mathbf{x} \cdot \mathbf{d}},\ \ \mathbf{H}^{{i}}\big(\mathbf{x}\big) &:=\frac{1}{\mathbf{i}k}\nabla\wedge \mathbf{p}\,  {e}^{\mathbf{i} k \mathbf{x} \cdot \mathbf{d}}=\mathbf{d} \wedge \mathbf{p} \, {e}^{\mathbf{i} k \mathbf{x} \cdot \mathbf{d}},
\end{aligned}
\end{equation}
which are known as the time-harmonic electromagnetic plane waves, with ${\mathbf p}  \in \mathbb{R}^{3}\backslash\{\mathbf{0}\}, k \in \mathbb R_+$ and ${\mathbf d} \in \mathbb{S}^{2}:=\left\{\mathbf{x} \in \mathbb{R}^{3} ;|\mathbf{x}|=1\right\}$ representing respectively the polarization, wave number and direction of propagation, and it holds that $\mathbf{p}\perp\mathbf{d}$. The associated forward scattering problem can be  described by the following the time-harmonic Maxwell equations (cf. \cite{CK}):
\begin{equation}\label{eq:forward}
\begin{cases}
&\nabla\wedge\mathbf{E}-\mathbf{i} k \mathbf{H}=0 \quad \text { in } \quad \mathbb{R}^{3} \backslash \overline{\Omega}, \\
& \nabla\wedge \mathbf{H}+\mathbf{i} k \mathbf{E}=0 \quad \text { in } \quad \mathbb{R}^{3} \backslash\overline{\Omega}, \\
&\mathbf{E}({\mathbf x})=\mathbf{E}^{{i}}({\mathbf x})+\mathbf{E}^{s}({\mathbf x} ),\\
& \mathbf{H}({\mathbf x})=\mathbf{H}^{{i}}({\mathbf x})+\mathbf{H}^{s}({\mathbf x} ),\\
& \mathscr{B}({\mathbf E})={\mathbf 0}\hspace*{1.75cm}\mbox{on}\ \ \partial\Omega,\medskip\\
&\lim _{|{\mathbf x}| \rightarrow \infty}\left(\mathbf{H}^{s} \wedge {\mathbf x}-|{\mathbf x} | \mathbf{E}^{s}\right)={\mathbf 0},
\end{cases}
\end{equation}
where $\mathbf{E}=\left(E_{1}, E_{2}, E_{3}\right)$ and $\mathbf{H}=\left(H_{1}, H_{2}, H_{3}\right)$ are respectively the total electric and magnetic fields formed by the incident fields $\mathbf{E}^{{i}}({\mathbf x} )$ and $\mathbf{H}^{{i}}({\mathbf x} )$ and scattered fields $\mathbf{E}^{s}({\mathbf x} )$ and $\mathbf{H}^{s}({\mathbf x} )$. The last equation of \eqref{eq:forward} is the  Silver-M\"uller radiation condition. The boundary condition $\mathscr{B}(\mathbf{E} )$ on $\partial \Omega$ could be either of the following three conditions:
\begin{enumerate}
	\item the Dirichlet condition (corresponding to that $\Omega$ is a perfectly electric conducting (PEC) obstacle):
\begin{equation}\label{eq:83}
	\mathscr{B}(\mathbf{E})=\nu \wedge \mathbf{E};
\end{equation}
	\item the Neumann condition (corresponding to that $\Omega$ is a perfectly magnetic conducting (PMC) obstacle):
\begin{equation}\label{eq:84}
 	\mathscr{B}(\mathbf{E})=\nu \wedge ( \nabla \wedge \mathbf{E}) ;
 \end{equation}
	\item the impedance condition (corresponding to that $\Omega$ is an impedance obstacle):
	\begin{equation} \label{eq:bound imp}
		  \mathscr{B}(\mathbf{E})=\nu \wedge ( \nabla \wedge \mathbf{E})  +\boldsymbol \eta( \nu \wedge \mathbf{E}) \wedge \nu  ,\ \Re(\boldsymbol \eta)\geq 0 \mbox{ and } \Im(\boldsymbol\eta)<0,
	\end{equation}
	\end{enumerate}
where $\nu$ denotes the exterior unit normal vector to $\partial\Omega$ and $\boldsymbol\eta\in L^\infty(\Omega)$.  We would also like to point out that the conditions $\Re(\boldsymbol\eta)\geq 0 \mbox{ and } \Im(\boldsymbol\eta)<0$ are the physical requirement. 


In what follows, in order to ease the exposition and similar to our notation in \eqref{eq:imp1}--\eqref{eq:imp1}, we unify the three types of boundary conditions as
\begin{equation}\label{bound}
\mathscr{B}(\mathbf{E})=\nu \wedge ( \nabla \wedge \mathbf{E})  +\boldsymbol\eta( \nu \wedge \mathbf{E}) \wedge \nu \quad\mbox{on } \partial\Omega,
\end{equation}
where the cases that $\boldsymbol\eta=\infty$ and $\boldsymbol\eta=0$ stand for the Dirichlet and Neumann boundary conditions respectively.

For the forward scattering problem \eqref{eq:forward}, it is known that there exists a unique pair of solutions $({\mathbf E}, {\mathbf H}) \in$ $H_{\mathrm{loc }}(\mathrm{curl} , \mathbb{R}^{3} \backslash\overline{\Omega}) \times H_{\mathrm{loc }}(\mathrm{curl}, \mathbb{R}^{3} \backslash\overline{\Omega})$ (cf. \cite{Ned}). Furthermore, the radiating fields $\mathbf{E}^{s}$ and $\mathbf{H}^{s}$  to \eqref{eq:forward} possess the following asymptotic expansions
\begin{equation}\label{eq:far}
\begin{split}
	\mathbf{E}^{s}(\mathbf{x} ;  \Omega, k, \mathbf{d}, \mathbf{p})&=\frac{{e}^{\mathbf{i} k \mathbf{x}  \cdot {\mathbf d} }}{|\mathbf{x} |}\left\{  \mathbf{E}_{\infty}(\hat{\mathbf{x}} ; \Omega,  k, \mathbf{d}, \mathbf{p})+\mathcal{O}\left(\frac{1}{|\mathbf{x} |}\right)\right\} \quad \text { as }   \quad|\mathbf{x} | \rightarrow \infty, \\
	\mathbf{H}^{s}(\mathbf{x} ;  \Omega,  k, \mathbf{d}, \mathbf{p})&=\frac{{e}^{\mathbf{i} k \mathbf{x}  \cdot {\mathbf d} }}{|\mathbf{x} |}\left\{\mathbf{H}_{\infty}(\hat{\mathbf{x}} ;  \Omega, k, \mathbf{d}, \mathbf{p})+\mathcal{O}\left(\frac{1}{|\mathbf{x} |}\right)\right\} \quad \text { as } \quad|\mathbf{x} | \rightarrow \infty,
\end{split}
\end{equation}
which hold uniformly in the angular variable $\hat{\mathbf{x} }=\mathbf{x}  /|\mathbf{x} | \in \mathbb{S}^{2} .$ The functions $\mathbf{E}_{\infty}(\hat{\mathbf{x} })$ and $\mathbf{H}_{\infty}(\hat{\mathbf{x} })$ in \eqref{eq:far} are called, respectively, the electric and magnetic far field patterns, and both are analytic on the entire unit sphere $\mathbb{S}^{2}$. As above and also in what follows, the notation $\mathbf{U}(\mathbf{x} ; \Omega, \mathbf{p}, k, \mathbf{d})$ will be frequently used to specify the dependence of a given function $\mathbf{U}$ on the scatterer $\Omega,$ the polarization $\mathbf{p},$ the wave number $k$ and the incident direction $\mathbf{d}$.

The inverse  electromagnetic obstacle scattering problem corresponding to \eqref{eq:forward} is to recover $\Omega$ (and $\boldsymbol \eta$ as well in the impedance case) by the knowledge of the far-field pattern $ \mathbf{E}_{\infty}(\hat{\mathbf{x}} ; \Omega, \mathbf{p}, k, \mathbf{d})$ (or equivalently  $ \mathbf{H}_{\infty}(\hat{\mathbf{x}} ; \Omega, \mathbf{p}, k, \mathbf{d})$). By introducing an operator $\mathcal{F}$ which sends the obstacle to the corresponding far-field pattern, defined by the forward scattering system \eqref{eq:forward}, the aforementioned inverse problem can be formulated as
\begin{equation}\label{inverse}
\mathcal{F}(\Omega, \boldsymbol \eta)=  \mathbf{E}_{\infty}(\hat{\mathbf{x}} ; \Omega, k, \mathbf{d}, \mathbf{p}).
\end{equation}
It can be directly verified that the inverse problem \eqref{inverse} is nonlinear and moreover it is ill-conditioned (cf. \cite{CK}).  It is a longstanding problem that one can establish the one-to-one correspondence for \eqref{inverse} by a single far-field pattern or a finite number of far-field patterns (namely with a fixed triplet of $k$, $\mathbf{d}$ and $\mathbf{p}$ or a finite number of triplets of $k$, $\mathbf d$ and $\mathbf{p}$); see the recent survey paper \cite{CK18} by Colton and Kress for more discussions about the historical developments of this fundamental problem.

Under the assumption that $\Omega$ is a polyhedral obstacle associated with $\boldsymbol\eta\equiv 0$ or $\boldsymbol \eta\equiv \infty$, the unique correspondence, a.k.a unique identifiability, for the inverse problem \eqref{inverse} by a single far-field measurement was established in the literature; see \cite{LiuA,LRX,Liu3,Liu09}. However, it is still unclear whether one can establish the unique identifiability for an impedance obstacle of the polyhedral shape, even for the case that $\boldsymbol\eta$ is a nonzero constant, and a fortiori in our present paper $\boldsymbol\eta$ is a generalised impedance parameter which can be $0$, $\infty$ or a variable function. To be more specific about the generalised impedance obstacle, we introduce the following definition.

\begin{definition}\label{ir obstacle}

Let $\Omega$ be an open and bounded polyhedron in $\mathbb{R}^3$. Hence, $\partial\Omega$ possesses finitely many edge-corners that are formed by the intersections of any two adjacent faces of $\partial\Omega$. $\Omega$ is said to be irrational if all of its edge-corners are irrational, otherwise it is called rational, and the smallest degree among the rational degrees of all of its rational corners is referred to the degree of the polyhedron. 
\end{definition}

\begin{definition}\label{def:so1}
$(\Omega, \boldsymbol\eta)$ is said to be an admissible polyhedral obstacle if $\Omega$ is an open bounded polyhedron and $\boldsymbol\eta$ fulfils the following requirements. 
\begin{enumerate}
\item For each face of $\partial\Omega$, say $\widetilde\Pi$, and each edge of $\widetilde\Pi$, say $\bsl$, there exists a neighbourhood $\Sigma_{\bsl}:=B_\rho(\bsl)\cap \widetilde\Pi$ with $\rho\in\mathbb{R}_+$ and $B_\rho(\bsl):=\{\mathbf{x}\in\mathbb{R}^3; |\mathbf{x}-\mathbf{x}'|<\rho, \forall \mathbf{x}'\in\bsl \}$, such that either $\boldsymbol \eta|_{\Sigma_{\bsl}}=0$, or $\boldsymbol \eta|_{\Sigma_{\bsl}}=\infty$, or $\eta|_{\Sigma_{\bsl}}\in\mathcal{A}(\bsl)$. 
\item On any open subset of the other part of $\partial\Omega$ other than the neighbourhood of each edge of $\partial\Omega$ introduced in (1), $\boldsymbol \eta$ can be $0$, or $\infty$ or $\eta\in L^\infty$.

\item In the case $\boldsymbol \eta\in L^\infty$, one has that $\Re(\boldsymbol \eta)\geq 0$ and $\Im(\boldsymbol \eta)<0$. 
\end{enumerate}
\end{definition}

\begin{definition}\label{def6}
	$\Omega$ is said to be an admissible complex polyhedral obstacle if it consists of finitely many admissible polyhedral obstacles.
	That is,
	\begin{equation*}\label{eq:r2a}
	(\Omega, \boldsymbol \eta)=\bigcup_{j=1}^l (\Omega_j, \boldsymbol \eta_j),
	\end{equation*}
	where $l\in\mathbb{N}$ and each $(\Omega_j, \boldsymbol \eta_j)$ is an admissible polyhedral obstacle.
	Here, we define
	\begin{equation*}\label{eq:r2b}
	\boldsymbol \eta=\sum_{j=1}^l \boldsymbol \eta_j\chi_{\partial\Omega_j}.
	\end{equation*}
	Moreover, $\Omega$ is said to be irrational if all of its component polyhedral obstacles are irrational, otherwise it is said to be rational. For the latter case, the smallest degree among all the degrees of its rational components is defined to be the degree of the complex obstacle $\Omega$.
\end{definition}

Next, we first derive a local unique identifiability result in determining an admissible complex irrational polyhedral obstacle by a single far-field pattern.

\begin{theorem}\label{thm:uniqueness1}
	Consider a fixed triplet of $k\in\mathbb{R}_+$, $\mathbf{d}\in\mathbb{S}^2$ and $\mathbf{p}\in\mathbb{R}^3\backslash\{\mathbf{0}\}$.
	Let $(\Omega, \eta)$ and $(\widetilde\Omega, \widetilde\eta)$ be
	two admissible complex irrational obstacles, with $\mathbf{E}_\infty$ and $\widetilde{\mathbf{E}}_\infty$ being
	their corresponding far-field patterns
	and $\mathbf{G}$ being
	the unbounded connected component of $\mathbb{R}^3\backslash\overline{(\Omega\cup\widetilde\Omega)}$.
	If $\mathbf{E}_\infty$ and $\widetilde{\mathbf{E}}_\infty$ are the same in the sense that
	\begin{equation}\label{eq:cond1}
 \mathbf{E}_{\infty}(\hat{\mathbf{x}} ; \Omega, k, \mathbf{d}, \mathbf{p})=\widetilde {\mathbf{E}}_{\infty}(\hat{\mathbf{x}} ; \widetilde {\Omega},  k, \mathbf{d}, \mathbf{p}), \ \
	\mbox{for } ~~  ~~\mbox{all} ~~\hat{\mathbf x}\in\mathbb{S}^2,
	\end{equation}
	then
	$
	(\partial \Omega \backslash \partial \overline{ \widetilde{\Omega }} )\bigcup (\partial \widetilde{\Omega } \backslash \partial  \overline{ \Omega } )
	$
	cannot possess an edge-corner on $\partial \mathbf{G}$.
	Moreover,
	\begin{equation}\label{eta}
	\boldsymbol \eta=\widetilde{\boldsymbol \eta}\quad\mbox{on}\quad \partial\Omega\cap\partial\widetilde{\Omega}\cap\partial\mathbf{G}.
	\end{equation}
\end{theorem}
\begin{proof}
	    We prove the theorem by contradiction. Assume that
		$(\partial \Omega \backslash \partial \overline{ \widetilde{\Omega }} )\bigcup (\partial \widetilde{\Omega } \backslash \partial  \overline{ \Omega } )$ has an edge  corner $\mathbf x_c$ on $\partial \mathbf{G}$.  Then, $\mathbf x_c$ is either located at $\partial\Omega$ or $\partial\widetilde\Omega$. Without loss of generality, we assume that $\mathbf x_c$ is an edge corner of $\partial\widetilde\Omega$, which also indicates that $\mathbf{x}_c$ lies outside $\Omega$. Let $h\in\mathbb{R}_+$ be sufficiently small such that $B_h(\mathbf{x}_c)\Subset\mathbb{R}^2\backslash\overline \Omega $, then we have
		\begin{equation*}\label{eq:aa2}
		B_h(\mathbf x_c)\cap \partial\widetilde\Omega=\widetilde\Pi_\ell,\quad \ell=1,2, 
		\end{equation*}
		where $\widetilde\Pi_\ell$ are two flat subsets lying on the faces of $\widetilde\Omega$ that intersect at $\mathbf x_c$. Moreover, for the subsequent use, we let $h$ be smaller than $\rho$, where $\rho$ is the parameter in Definition~\ref{def:so1}. Hence we have an edge-corner $\mathcal{E}(\widetilde\Pi_1,\widetilde\Pi_2,\bsl)\in \partial \mathbf G$ with $\mathbf{x}_c\in \bsl $, where $\mathbf{G}$ is the unbounded connected component of $\mathbb{R}^3\backslash\overline{(\Omega\cup\widetilde\Omega)}$.  By \eqref{eq:cond1} and the Rellich theorem (cf. \cite{CK}), we know that
		\begin{equation}\label{eq:aa3}
		\mathbf{E}(\mathbf x; k, \mathbf{d},\mathbf{p})=\widetilde{\mathbf{E}}(\mathbf x; k, \mathbf{d},\mathbf{p}),\quad {\mathbf x} \in\mathbf{G}.
		\end{equation}
		Since $\widetilde\Pi_\ell \subset\partial\mathbf{G}$, $\ell=1,2$, combining \eqref{eq:aa3} with the generalized boundary condition \eqref{bound} on $\partial\widetilde\Omega$, it is easy to obtain that
		\begin{equation}\label{eq:aa4}
\nu_\ell \wedge (\nabla\wedge\mathbf{E})+\widetilde{\boldsymbol\eta} (\nu_\ell  \wedge\mathbf{E})\wedge\nu_\ell=\nu_\ell \wedge (\nabla\wedge\mathbf{\widetilde E})+\widetilde{\boldsymbol \eta}(\nu_\ell \wedge\mathbf{\widetilde E})\wedge\nu_\ell=\mathbf 0 \mbox{ on } \widetilde\Pi_\ell.
		\end{equation}
 
 We consider the following  two  separate cases, depending on the values of $\widetilde{\boldsymbol\eta}$ on $\widetilde\Pi_\ell$ associated with the edge-corner $\mathcal{E}(\widetilde\Pi_1, \widetilde\Pi_2,\bsl)$
		
			\medskip
		
\noindent{\bf Case 1.}~$\widetilde{\boldsymbol\eta}\big|_{\widetilde\Pi_\ell}=0$ or $\widetilde{\boldsymbol\eta}\big|_{\widetilde\Pi_\ell}=\infty$, $\ell=1, 2$. We only consider the case $\widetilde{\boldsymbol\eta}\big|_{\widetilde\Pi_\ell}=\infty$ and the other case can be treated in a similar manner. First, we note that one has from \eqref{eq:aa4}, 
\begin{equation}\label{eq:p1}
(\nu_\ell\wedge\mathbf{E})\wedge\nu_\ell=0\quad\mbox{on}\ \ \widetilde\Pi_\ell,\ \ell=1, 2.
\end{equation}
Let $\widehat{\Pi}_\ell$ denote the full flat extension of $\widetilde\Pi_\ell$ within $\mathbb{R}^3\backslash\overline{\Omega}$. We claim that at least one of $\widehat\Pi_\ell$ is bounded. In fact, if on the contrary that both $\widehat\Pi_1$ and $\widehat\Pi_2$ are unbounded, then one has from analytic continuation (noting that $\mathbf{E}$ is real analytic in $\mathbb{R}^3\backslash\overline{\Omega}$) and \eqref{eq:p1} that
\begin{equation}\label{eq:p2}
\lim_{|\mathbf{x}|\rightarrow\infty, \mathbf{x}\in\widehat\Pi_\ell}\left| (\nu_\ell\wedge\mathbf{E})\wedge\nu_\ell \right|=0, \ \ell=1, 2. 
\end{equation}
Using \eqref{eq:far}, we note that $\mathbf{E}^s(\mathbf{x})\rightarrow \mathbf{0}$ as $|\mathbf{x}|\rightarrow\infty$, and hence we further have from \eqref{eq:p2} that
\begin{equation}\label{eq:p3}
\lim_{|\mathbf{x}|\rightarrow\infty, \mathbf{x}\in\widehat\Pi_\ell}\left| (\nu_\ell\wedge\mathbf{E}^i)\wedge\nu_\ell \right|=0, \ \ell=1, 2, 
\end{equation} 
which together with \eqref{eq:p1n} readily implies that $|(\nu_\ell\wedge\mathbf{E})\wedge\nu_\ell|=0$, $\ell=1,2$. But this is impossible since $\nu_1$ and $\nu_2$ are linearly independent. Without loss of generality, we can assume that $\widehat\Pi_1$ is bounded. Clearly, $\widehat\Pi_1$ and part of $\partial\Omega$ form a bounded domain in $\mathbb{R}^3\backslash\overline{\Omega}$, and we denote it as $\Omega_1$. It is noted from \eqref{eq:aa4} that one has
\begin{equation}\label{eq:p4}
\nu \wedge (\nabla\wedge\mathbf{E})+\widetilde{\boldsymbol\eta} (\nu  \wedge\mathbf{E})\wedge\nu=\mathbf{0}\ \mbox{on}\ \partial\Omega_1\backslash\widehat{\Pi}_1\ \mbox{and}\ \nu \wedge (\nabla\wedge\mathbf{E})=0\ \mbox{on}\ \widehat{\Pi}_1.
\end{equation}
We next show that $\widetilde{\boldsymbol\eta}$ can only take $0$ or $\infty$ on $\partial\Omega_1\backslash\widehat{\Pi}_1$. Indeed, we assume on the contrary that there exists a nonempty open subset $\Lambda_1\subset \partial\Omega_1\backslash\widehat{\Pi}_1$ such that $\widetilde{\boldsymbol\eta} \in L^\infty(\Lambda_1)$ with $\Re(\widetilde{\boldsymbol\eta})\geq 0$ and $\Im(\widetilde{\boldsymbol\eta} )<0$, and on $(\partial\Omega_1\backslash\widehat{\Pi}_1)\backslash\overline{\Lambda_1}$, $\widetilde{\boldsymbol \eta} $ takes either $0$ or $\infty$. Noting that the Maxwell equations, namely the first two equations in \eqref{eq:forward} are satisfied in $\Omega_1$, we have from Green's formula that  
\begin{equation}\label{eq:p5}
\begin{split}
\mathbf{i}k\int_{\Omega_1} |\mathbf{H}|^2=&\int_{\Omega_1}(\nabla\wedge\mathbf{E})\cdot\overline{\mathbf{H}}=\int_{\Omega_1}\mathbf{E}\cdot(\nabla\wedge\overline{\mathbf{H}})+\int_{\Omega_1} (\overline{\mathbf{H}}\wedge\nu)\cdot\mathbf{E}\\
=&\mathbf{i}k\int_{\Omega_1}|\mathbf{E}|^2+\int_{\partial\Omega_1}(\overline{\mathbf{H}}\wedge\nu)\cdot\mathbf{E}=\mathbf{i}k\int_{\Omega_1}|\mathbf{E}|^2+\int_{\Lambda_1}(\overline{\mathbf{H}}\wedge\nu)\cdot\mathbf{E},
\end{split}
\end{equation}
where in deriving the last equality, we make use of the fact that $(\overline{\mathbf{H}}\wedge\nu)\cdot\mathbf{E}=0$ on $\partial\Omega_1\backslash\overline{\Lambda_1}$. Using the fact that $\Im(\widetilde{\boldsymbol\eta})<0$ on $\Lambda_1$, one can readily infer from \eqref{eq:p5} that $\nu\wedge\mathbf{E}|_{\Lambda_1}=\mathbf{0}$, which together with \eqref{eq:p4} further implies that $\nu\wedge\mathbf{H}|_{\Lambda_1}=0$. Hence, by the Holmgren's uniqueness principle (cf. \cite{CK}), one has that 
\begin{equation}\label{eq:aa51}
{\mathbf E}(\mathbf x; k, \mathbf{d},\mathbf{p}) =\mathbf{0} \mbox{ in } \mathbb{R}^3\backslash\overline{\Omega},
\end{equation}
which in particular yields that
	\begin{equation}\label{eq:aa6}
	\lim_{|\mathbf x|\rightarrow\infty} \left|{\mathbf E}(\mathbf x; k, \mathbf{d},\mathbf{p})\right|={\mathbf 0}.
	\end{equation}
	But this contradicts to the fact that follows from \eqref{eq:far}:
	\begin{equation}\label{eq:aa61}
	\lim_{|\mathbf x|\rightarrow\infty} \left|{\mathbf E}(\mathbf x; k, \mathbf{d},\mathbf{p})\right|=\lim_{|\mathbf x|\rightarrow\infty} \left|\mathbf{p}{e}^{\mathbf{i} k \mathbf{x} \cdot \mathbf{d}}+\mathbf{E}^{s}(\mathbf{x} ;  k, \mathbf{d},  \mathbf{p})\right|=|\mathbf{p}| \neq 0. 
	\end{equation}
Hence, we actually can find a polyhedral domain $\Omega_1\subset\mathbb{R}^3\backslash\Omega$ such that one has on $\partial\Omega_1$, either $\nu\wedge\mathbf{E}=0$ or $\nu\wedge\mathbf{H}=0$. The situation is reduced to that was considered in \cite{LiuA} and \cite{Liu09}. It is noted that in \cite{Liu09}, two far-field patterns are used to handle the above situation. However, the pair of incident fields $(\mathbf{E}^i, \mathbf{H}^i)$ in \eqref{eq:p1n} in our current case is chosen slightly different from that in \cite{Liu09}, which enables one to apply the path-argument from \cite{LiuA} to arrive at a contradiction by starting from $\Omega_1$.

\medskip
	
	\noindent {\bf Case 2.}~$\widetilde{\boldsymbol \eta}\big|_{\widetilde\Pi_\ell}\in\mathcal{A}(\bsl)$, $\ell=1, 2$; or one of $\widetilde{\boldsymbol \eta}\big|_{\widetilde\Pi_\ell}$ belongs to $\mathcal{A}(\bsl)$, and the other one takes $0$ or $\infty$; or one of $\widetilde{\boldsymbol \eta}\big|_{\widetilde\Pi_\ell}$ is $0$ and the other one is $\infty$. This falls exactly to the situation considered in Theorem~\ref{ir-2nodal}. By the irrationality of the edge-corner as well as the strong uniqueness continuation principle in Theorem~\ref{ir-2nodal}, we readily \eqref{eq:aa51}, which again leads to the contradiction \eqref{eq:aa61}.

\medskip

It remains to prove \eqref{eta}, and we establish it by contradiction.
	Let $\Gamma\subset \partial\Omega\cap\partial\widetilde\Omega\cap\partial\mathbf{G}$ be an open subset such that $\boldsymbol \eta\neq \widetilde{\boldsymbol\eta}$ on $\Gamma$. By taking a smaller subset of $\Gamma$ if necessary, we can assume that $\eta$ (respectively
	$\widetilde\eta$) is either $L^\infty$ or $0$ or $\infty$ on $\Gamma$. Clearly, one has $\mathbf{E}=\widetilde {\mathbf E}$ in $\mathbf{G}$. Hence it holds that
	\begin{equation*}\label{eq:bb6}
	( \nu  \wedge \mathbf{E}) \wedge \nu=( \nu \wedge \widetilde{ \mathbf{E}} ) \wedge \nu \mbox{ and } \nu \wedge (\nabla \wedge \mathbf{E}  ) =\nu \wedge (\nabla \wedge \widetilde{ \mathbf{E}}   )\quad\mbox{on}\ \ \Gamma.
	\end{equation*}
	and
	$$
	\nu \wedge ( \nabla \wedge \mathbf{E})  +\boldsymbol \eta( \nu  \wedge \mathbf{E}) \wedge \nu =\mathbf {0},\ \ \nu \wedge ( \nabla \wedge \widetilde{\mathbf{E}})  +\widetilde{\boldsymbol \eta} ( \nu  \wedge \widetilde {\mathbf{E}}) \wedge \nu =\mathbf{0} \quad\mbox{on}\ \ \Gamma.
	$$
	Combining with the assumption that $\boldsymbol \eta\neq\widetilde{\boldsymbol \eta}$ on $\mathcal{E}$,  we can directly deduce that
	\[
	\nu\wedge\mathbf{E}=\nu\wedge\mathbf{H}=0\quad\mbox{on}\ \Gamma,
	\]
	which in turn yields by the Holmgren's uniqueness principle (cf. \cite{CK}) that $\mathbf{E} =\mathbf{0}$ in $\mathbb{R}^3\backslash\Omega$. Therefore, we arrive at the same contradiction as that in \eqref{eq:aa6} and \eqref{eq:aa61}, which readily proves \eqref{eta}.
	
	The proof is complete. 
\end{proof}

 It is recalled that the convex hull of $\Omega$, denoted by $\mathcal{CH}(\Omega)$, is the smallest convex set that contains $\Omega$. As a direct consequence of Theorem \ref{thm:uniqueness1}, we next show that the convex hull of a complex irrational obstacle can be uniquely determined by one far-field measurement. Furthermore, the boundary impedance  parameter $\eta$ can be partially identified as well.  In fact we have
\begin{corollary}\label{co:84}
Consider a fixed triplet of $k\in\mathbb{R}_+$, $\mathbf{d}\in\mathbb{S}^2$ and $\mathbf{p}\in\mathbb{R}^3\backslash\{\mathbf{0}\}$.
	Let $(\Omega, \eta)$ and $(\widetilde\Omega, \widetilde\eta)$ be
	two admissible complex irrational obstacles, with $\mathbf{E}_\infty$ and $\widetilde{\mathbf{E}}_\infty$ being
	their corresponding far-field patterns.
	If $\mathbf{E}_\infty$ and $\widetilde{\mathbf{E}}_\infty$ satisfy \eqref{eq:cond1}, then  one has that
\begin{equation}\label{eq:cond2}
\mathcal{CH}(\Omega)=\mathcal{CH}(\widetilde\Omega):=\Sigma,
\end{equation}
and
\begin{equation}\label{eq:cond3}
\boldsymbol \eta=\widetilde{\boldsymbol \eta}\ \ \mbox{on}\ \ \partial\Omega\cap\partial\widetilde\Omega\cap \partial\Sigma.
\end{equation}	
\end{corollary}

Corollary~\ref{co:84} implies that if the underlying polyhedral obstacle is convex, then one can uniquely determine the obstacle as well as its boundary impedance by a single far-field pattern.  As a further application of the UCP results established in this work, we consider the unique determination of a rather general class of non-convex obstacles. To that end, we first introduce the aforesaid class of non-convex obstacles.  

In the sequel, we denote by ${\boldsymbol  P}_{S}(\mathbf{x})$ the projection of a point $\mathbf{x}\in\mathbb{R}^3$ onto a set $S$. Let $\partial (\mathcal{  CH}(\Omega
	))=\{\Sigma_\ell~|\ell=1,\ldots, N\}$, where $\Sigma_\ell$, $\ell=1,\ldots, N$ are the finitely many faces of $\mathcal{  CH}(\Omega
	)$. Let $\mathcal{V}(\Omega)$ and $\mathcal{V}(\mathcal{CH}(\Omega))$ denote, respectively, the sets of vertices of $\Omega$ and $\mathcal{CH}(\Omega)$. It is known that $\mathcal{V}(\mathcal{CH}(\Omega))\subset\mathcal{V}(\Omega)$. For any vertex $\mathbf{v} \in  \mathcal{V}(\Omega) \backslash \mathcal{V}(\mathcal{CH}(\Omega))$, we consider the projection, ${\boldsymbol P}_{\Sigma_j} (\mathbf{v})$, where $\Sigma_j\subset\partial(\mathcal{CH}(\Omega))$ is a face. It is assumed that there exists at least one $\Sigma_j$ such that $\mathbf{v}-{\boldsymbol P}_{\Sigma_j}(\mathbf{v})\subset\mathbb{R}^3\backslash\Omega$. Then for a face $\Sigma_l\subset\partial(\mathcal{CH}(\Omega))$ we say that $\mathbf{v}\vdash\Sigma_l$ if
 \begin{equation}\label{eq:projection}
 \mathbf{v}-{\boldsymbol P}_{\Sigma_\ell}(\mathbf{v})=\argmin_{\mathbf{v}-{\boldsymbol P}_{\Sigma_j} (\mathbf{v}) \in \mathbb R^3 \backslash \Omega, \forall \Sigma_j\subset\partial(\mathcal{CH}(\Omega))  } \left|\mathbf{v}-{\boldsymbol P}_{\Sigma_j} (\mathbf{v})  \right|.
\end{equation}

\begin{definition}\label{def:85}
	Let $\Omega$ be an admissible polyhedral obstacle, and let $\Sigma_l$ be a given face of $\partial(\mathcal{CH}(\Omega))$, and $\mathcal{V}_\mathcal{C}$ be a given set of finitely many, discrete and distinct points on $\Sigma_l$. $\Omega$ is said to be uniformly concave with respect to $\mathcal{V}_\mathcal{C}$ if $\forall\mathbf{v}\in\mathcal{V}(\Omega)\backslash\mathcal{V}(\mathcal{CH}(\Omega))$, $\mathbf{v}\vdash\Sigma_l$ and 
	\[
	\{{\boldsymbol P}_{\Sigma_l}(\mathbf{v})~|~\mathbf{v}\in\mathcal{V}(\Omega)\backslash\mathcal{V}(\mathcal{CH}(\Omega))\}=\mathcal{V}_\mathcal{C}. 
	\]
\end{definition}

\begin{figure}[htbp]
	\centering
\vspace*{-1cm}	\includegraphics[width=0.3\linewidth]{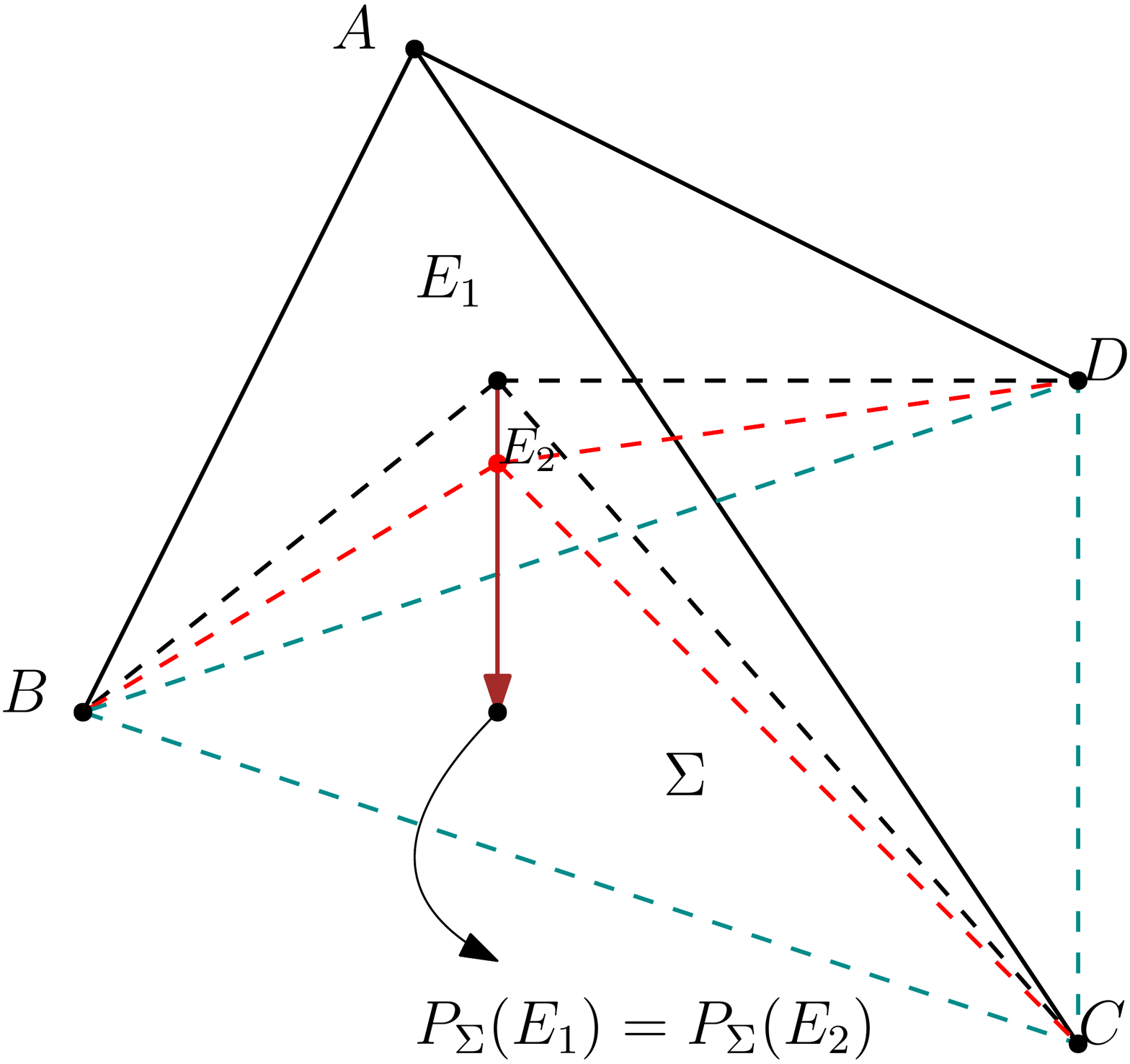}\\[-25pt]
	\caption{Schematic illustration of two different uniformly concave hexahedrons $ABCDE_1$ and $ABCDE_2$ with $\mathcal{CH}(ABCDE_1)$=$\mathcal{CH}(ABCDE_2)=ABCD$.}
	\label{fig:convex}
\end{figure}

As a simple illustrating example of Definition \ref{def:85}, we consider two different uniformly concave hexahedrons $\Omega_1:= ABCDE_1$ and $\Omega_2:=ABCDE_2$ that are shown in Figure \ref{fig:convex}. It is easy to see that $\Omega_1$ and $\Omega_2$  has the same convex hull, which is the tetrahedron $ABCD$. The vertexes $E_1$ and $E_2$ corresponding to $\Omega_1$ and $\Omega_2$ have the same projecting point on the face $\Sigma:=BCD$ of the convex hull $ABCD$. It is pointed out that the vertex corner $\mathcal{V}(BE_2C,CE_2D,BE_2D, E_2)\in \partial \mathbf G$, where $BE_2C,CE_2D,BE_2D$ are faces of $\Omega_2$ and $\mathbf{G}=\mathbb R^3 \backslash (\Omega_1 \cup \Omega_2
	)$.

\begin{theorem}\label{thm:uniqueness2}
	Consider a fixed triplet of $k\in\mathbb{R}_+$, $\mathbf{d}\in\mathbb{S}^2$ and $\mathbf{p}\in\mathbb{R}^3\backslash\{\mathbf{0}\}$.
	Let $(\Omega, \eta)$ and $(\widetilde\Omega, \widetilde\eta)$ be
	two uniformly concave irrational admissible polyhedral obstacles with respect to the set $\mathcal{V}_{\mathcal C}$, with $\mathbf{E}_\infty$ and $\widetilde{\mathbf{E}}_\infty$ being
	their corresponding far-field patterns.
	If $\mathbf{E}_\infty$ and $\widetilde{\mathbf{E}}_\infty$ satisfy \eqref{eq:cond1}, then
	$$
	\Omega =\widetilde \Omega \mbox{ and } \boldsymbol \eta=\widetilde{\boldsymbol \eta}.
	$$
\end{theorem}
\begin{proof}
	We prove this theorem by contradiction. Assume that $\Omega \neq \widetilde \Omega $ but \eqref{eq:cond1} is still fulfilled. From Corollary \ref{co:84}, we have  $\mathcal{CH}(\Omega)= \mathcal{CH}(\widetilde{ \Omega} )$, which implies that the vertices of  $\Omega$ contributing to $\mathcal{CH}(\Omega)$ are the same as the corresponding vertices of $\widetilde \Omega$ contributing to  $\mathcal{CH}(\widetilde{ \Omega} )$. We shall prove that there must exit an edge-corner $\mathcal{E}(\Pi_1,\Pi_2,\mathbf{x}_c) \in \partial \mathbf G$, where   $\mathbf{G}$ is
	the unbounded connected component of $\mathbb{R}^3\backslash\overline{(\Omega\cup\widetilde\Omega)}$. Since $\Omega \neq \widetilde \Omega $, there exits an edge $\bsl \subset \partial \Omega \backslash \partial \widetilde \Omega$ or  $\bsl \subset \partial \widetilde \Omega \backslash  \partial \Omega$. Without loss of generality, we assume that $\bsl \subset  \partial \widetilde \Omega \backslash \partial \Omega  $.
	In the sequel, we let ${\mathbf a}_{\boldsymbol l } $ and ${\mathbf b}_{\boldsymbol{l} }$ denote the two vertices of the line segment $\bsl  $.  We divide our remaining proof into two  separate cases.
		
			\medskip
	
	\noindent {\bf Case 1.}~Suppose that ${\mathbf a}_{\boldsymbol l } \in \mathcal{V}(\mathcal{CH}(\widetilde \Omega))$ and ${\mathbf b}_{\boldsymbol{l} } \in \mathcal{V}(\mathcal{CH}(\widetilde \Omega)) $. Therefore, $\bsl \subset \partial \mathbf{G} \cap \partial \widetilde{ \Omega}$.  There exits a point $\mathbf{x}_c \in \bsl$ and a sufficient small $h\in \mathbb R_+$ such that
	\begin{equation*}\label{eq:aa2}
		B_h(\mathbf x_c)\cap \partial\widetilde\Omega=\widetilde\Pi_\ell,\quad \ell=1,2, 
		\end{equation*}
		where $\widetilde\Pi_\ell$ are two flat subsets lying on the  faces of $\widetilde\Omega$ that intersect at $\mathbf x_c$. Clearly, $\mathbf{x}_c \in \bsl$ is an edge-corner point.
		
		\medskip
	
	\noindent {\bf Case 2.}~Suppose that there exits at least one of ${\mathbf a}_{\boldsymbol l }$ and ${\mathbf b}_{\boldsymbol l }$ belonging to  $ \mathcal{V} (\widetilde \Omega )\backslash \mathcal{V}(\mathcal{CH}(\widetilde \Omega)) $; namely, $ \mathbf{x}_c\in \mathcal{V} (\widetilde \Omega )\backslash \mathcal{V}(\mathcal{CH}(\widetilde \Omega)) $, where $ \mathbf{x}_c$ could be either ${\mathbf a}_{\boldsymbol l }$ or ${\mathbf b}_{\boldsymbol l }$.   Since $\Omega$ and $\widetilde\Omega$ are uniformly concave admissible polyhedral obstacles with respect to the set $\mathcal{V}_{\mathcal C}$, there exits a face $\Sigma_\ell \Subset  \partial(\mathcal{CH}(\Omega))$ such that $ \mathbf{x}_c \vdash \Sigma_\ell  $ and $  \mathcal{V}_{\mathcal C} \Subset \Sigma_\ell$. Furthermore, we know that there exits a vertex ${\mathbf x}_{c,\Omega } \in  \mathcal{V}(\Omega) \backslash \mathcal{V}(\mathcal{CH}(\Omega)  $ such that
	$$
	{\mathbf x}_{c,\Omega }\vdash \Sigma_\ell, \quad  \boldsymbol{ P}_{\Sigma_\ell } \left({\mathbf x}_{c,\Omega }\right)=\boldsymbol{ P}_{\Sigma_\ell } \left({\mathbf x}_{c}\right) \in \mathcal{V}_{\mathcal C}. 
	$$
Since ${\mathbf x}_{c,\Omega }$ and ${\mathbf x}_{c}$ are distinct, it holds that
	$$
	{\mathrm d}\left(  {\mathbf x}_{c }, \Sigma_\ell  \right)\neq  {\mathrm d}\left(  {\mathbf x}_{c,\Omega }, \Sigma_\ell  \right), 	
	$$
	where $ \mathrm {d}\left(  {\mathbf x}_{c }, \Sigma_\ell  \right) $ is the distance  between $ {\mathbf x}_{c }$ and $\Sigma_\ell$.  Without loss of generality, we may assume that ${\mathrm d}\left(  {\mathbf x}_{c }, \Sigma_\ell  \right)< {\mathrm d}\left(  {\mathbf x}_{c,\Omega }, \Sigma_\ell  \right)$.   Hence, one can conclude that
	$$
	 {\mathbf x}_{c } \in \partial \mathbf G,
	$$
	which also indicates that $\mathbf{x}_c$ lies outside $\Omega$. Let $h\in\mathbb{R}_+$ be sufficiently small such that $B_h({\mathbf x}_c)\Subset\mathbb{R}^2\backslash\overline \Omega $, then due to the fact that  $\mathcal{V}_{\mathcal C}$ is discrete and distinct we can conclude that
			\begin{equation*}\label{eq:aa2}
		B_h(\mathbf x_c)\cap \partial\widetilde\Omega=\widetilde \Pi_\ell,\quad \ell=1,2, 
		\end{equation*}
		where $\widetilde \Pi_\ell$ are two  plane cells lying on the  faces of $\widetilde\Omega$ that intersect at $\mathbf x_c$.


The remaining proof is similar to the that of Theorem \ref{thm:uniqueness1},  which is omitted.
	 \end{proof}
	 
Finally, we remark that in this section, we only consider the case that the underlying obstacle is irrational in order to make use of the strong unique continuation principle in Theorem~\ref{ir-2nodal}. That is, in the contradiction argument in proving Theorems~\ref{thm:uniqueness1} and \ref{thm:uniqueness2}, one can find an edge-corner that can lead to the vanishing of the total wave field outside the obstacle by the strong unique continuation principle in Theorem~\ref{ir-2nodal}. However, we would like to emphasize that the same argument would work for the case that the underlying obstacle is of a general polyhedral shape, subject to a some slight modification. In fact, in such a case, it may happen that the edge-corner in the contradiction argument is rational, and hence instead of Theorem~\ref{ir-2nodal}, one would need to make use of the finite vanishing order results in   
Theorems~\ref{th:two imp}, \ref{thm:pec pmc}, \ref{thm:imp pec} and \ref{thm:imp pmc} to obtain that the total wave field is ``small" around the edge-corner (compared to the totally vanishing in the irrational case). Hence, a contradiction can be obtained if one requires that the total wave field outside the obstacle is everywhere ``big", which can be fulfilled in certain scenarios of practical interest, see e.g. \cite{CDL1}. Nevertheless, we shall not explore this direction any further in this paper. 

\subsection{Information-encoding for inverse problems and generalised Holmgren's uniqueness principle }

We recall the classical Holmgren's theorem for an elliptic PDO $\mathcal{P}$ with real-analytic coefficients (cf. \cite{TF}). If $\mathcal{P}\mathbf{u}$ is real analytic in a connected open neighbourhood of $\Omega$, then $\mathbf{u}$ is also real-analytic. The Holmgren's theorem applied to $\mathbf{u}=(\mathbf{E},\mathbf{H})$ in \eqref{eq:eig}, we immediately see that $(\mathbf{E},\mathbf{H})$ is real-analytic in $\Omega$. Let $\Gamma$ be an analytic surface in $\Omega$. Suppose that 
\begin{equation}\label{eq:cond1l}
\nu\wedge\mathbf{E}=\mathbf{0}\quad\mbox{and}\quad\nu\wedge\mathbf{H}=0\quad\mbox{on}\ \ \Gamma, 
\end{equation}
then by the Cauchy-Kowalevski theorem, one readily has that $\mathbf{E}=\mathbf{H}\equiv 0$ in $\Omega$. This is known as the Holmgren's uniqueness principle. In fact, in the proofs of Theorems~\ref{thm:uniqueness1} and \ref{thm:uniqueness2}, we have made use of the Holmgren's principle in the case that $\Gamma$ is an open subset of a plane. In the sequel, to ease the exposition and with a bit abuse of notations, we simply refer to $\Gamma$ as a plane in such a case, though it may actually be an open subset of a plane. Our results established in Theorems~\ref{th:two imp}, \ref{thm:pec pmc}, \ref{thm:imp pec}, \ref{thm:imp pmc} and \ref{ir-2nodal} can be regarded as generalizing the Holmgren's uniqueness principle as discussed in what follows. 

Suppose that there are two planes $\widetilde\Pi_1$ and $\widetilde\Pi_2$ which intersect at a line segment $\bsl$ within $\Omega$ (see Fig.~\ref{fig:coordinate1}), and 
\begin{equation}\label{eq:gg1}
\nu\wedge\mathbf{E}=\mathbf{0}\ \ \mbox{on}\ \widetilde\Pi_1\quad\mbox{and}\quad\nu\wedge\mathbf{H}=\mathbf{0}\ \mbox{on}\ \widetilde\Pi_2. 
\end{equation}
Let $\angle(\widetilde\Pi_1,\widetilde\Pi_2)=\alpha\pi$. Suppose that $\alpha=1/N$ with $N\in\mathbb{N}$. Then according to Theorem~\ref{thm:pec pmc}, we know that the vanishing order of $\mathbf{E}$ around $\bsl$ is at least $N$. Letting $N\rightarrow \infty$, we see that in the limiting case, one has \eqref{eq:cond1l} with $\widetilde{\Pi}_1=\widetilde\Pi_2=\Gamma$ as well as that the vanishing order becomes infinity. That is, the classical Holmgren's uniqueness principle associated a plane $\Gamma$ for the Maxwell system \eqref{eq:eig} is the limiting case of our result in Theorem~\ref{thm:pec pmc}. It is surprisingly interesting that we have generalised such an observation in three aspect. First, the angle between the two intersecting planes is not infinitesimal and hence the vanishing order may be finite. Second, if the angle is irrational, not necessarily infinitesimal, the vanishing order is still infinity. Third, the homogeneous condition on the plan can be the much more general impedance condition. 

The application to inverse problem of the above observation can be described as follows. In inverse problems with electromagnetic probing, one usually sends a pair of incident fields and then collects the corresponding scattered wave data away from the inhomogeneous object; see \eqref{inverse} associated with \eqref{eq:forward}. In the following, we first take \eqref{inverse} as a specific exam elucidate the basic idea. Usually, the collection of the data is made on an analytic surface, say $\Gamma$, in the form $(\nu\wedge\mathbf{E}|_{\Gamma}, \nu\wedge\mathbf{H}|_{\Gamma})$. Then by the Holmgren's principle, we know that the information encoded into $(\nu\wedge\mathbf{E}|_{\Gamma}, \nu\wedge\mathbf{H}|_{\Gamma})$ is equivalent to knowing the electromagnetic fields outside the scattering obstacle, namely $\mathbb{R}^3\backslash\overline{\Omega}$, and hence is equivalent to the far-field pattern $\mathbf{E}_\infty/\mathbf{H}_\infty$. According to Theorem~\ref{ir-2nodal}, the measurement data can also be collected as $(\nu\wedge\mathbf{H}+\boldsymbol \eta_1\nu\wedge\mathbf{E}|_{\widetilde\Pi_1}, \nu\wedge\mathbf{H}+\boldsymbol \eta_2\nu\wedge\mathbf{H}|_{\widetilde\Pi_2})$ as long as $\widetilde\Pi_1$ and $\widetilde\Pi_2$ can intersect within $\mathbb{R}^3\backslash\overline{\Omega}$ with an irrational angle. Clearly, due to the analytic extension, it is not necessary for $\widetilde\Pi_1$ and $\widetilde\Pi_2$ to really intersect each other. The irrational intersection seems to be too restrictive and one can relax it to be a rational intersection with a large degree. Clearly, this conceptual information encoding technique also work for the other inverse electromagnetic scattering problem where the underlying object is not necessary an impenetrable obstacle as that considered in \eqref{inverse}. We hope that it might find practical applications in some special situations.

\end{document}